\documentclass[final,onefignum,onetabnum]{siamonline220329}

\pdfoutput=1 % for arxiv submission
\usepackage{lipsum}
\usepackage{amsfonts}
\usepackage{graphicx}
\usepackage{epstopdf}
\usepackage{algorithmic}
\usepackage{amsopn}
\DeclareMathOperator{\diag}{diag}

\usepackage{amssymb} % for end of environment symbols and asymptotics
\usepackage[caption=false]{subfig}
\usepackage{booktabs} % for tables
\usepackage{mathtools} % to DeclarePairedDelimiter for ip, abs, norm
\usepackage{bbm} % for indicator function
\usepackage{mathrsfs} % for additional math fonts for letters
\usepackage{stmaryrd} % for ``mapsfrom''

\let\originalleft\left % Fixes \left and \right spacing issues. See discussion at http://tex.stackexchange.com/questions/2607/spacing-around-left-and-right
\let\originalright\right
\renewcommand{\left}{\mathopen{}\mathclose\bgroup\originalleft}
\renewcommand{\right}{\aftergroup\egroup\originalright}

\newcommand{\N}{\mathbb{N}}

\newcommand{\R}{\mathbb{R}}
\newcommand{\T}{\mathbb{T}}
\newcommand{\C}{\mathbb{C}}
\newcommand{\cP}{\mathcal{P}}

\newcommand{\cL}{\mathcal{L}}
\newcommand{\cF}{\mathcal{F}}
\newcommand{\cH}{\mathcal{H}}

\newcommand{\cC}{\mathcal{C}}
\newcommand{\cS}{\mathcal{S}}
\newcommand{\cG}{\mathcal{G}}
\newcommand{\cR}{\mathcal{R}}
\newcommand{\cE}{\mathcal{E}}
\newcommand{\cI}{\mathcal{I}}
\newcommand{\cJ}{\mathcal{J}}
\newcommand{\cK}{\mathcal{K}}

\newcommand{\sL}{\mathscr{L}}
\newcommand{\al}{\alpha}
\newcommand{\ep}{\varepsilon}
\DeclareMathOperator{\id}{Id} % identity
\newcommand{\lap}{\Delta} % Laplacian
\newcommand{\HS}{\mathrm{HS}} % Hilbert--Schmidt

\DeclareMathOperator{\im}{Im} % image space

\newcommand{\ZZ}{W} % Chi square random variable
\newcommand{\Ld}{L^{\dagger}} % ground truth linear operator
\newcommand{\ld}{l^{\dagger}} % ground truth linear operator sequence
\newcommand{\Ldmat}{\mathsf{L}^{\dagger}} % ground truth linear operator in mixed coordinates
\newcommand{\Lmat}{\mathsf{L}} % linear operator in mixed coordinates
\newcommand{\Ndm}{N_{\delta^{-}}} % high probability sample size
\newcommand{\post}{\mu^{D_N}} % posterior measure
\newcommand{\postseq}{\mu_{\mathrm{seq}}^{D_N}} % posterior measure on sequence
\newcommand{\priorseq}{\mu_{\mathrm{seq}}} % prior measure on sequence
\newcommand{\KZstarnormal}{K_{Z}^{*}K_{Z}^{\vphantom{*}}} % K_Z^*K_Z
\newcommand{\dtn}{\Lambda_{\sigma}} % dtn map eit
\newcommand{\ntd}{\mathscr{R}_{\sigma}} % ntd map eit
\newcommand{\dom}{\mathcal{D}} % domain of linear operator
\newcommand{\setoneton}{\{1, \ldots, N\}} % set of indices up to N
\newcommand{\disk}{\mathbb{D}} % unit disk
\newcommand{\ER}{\E^{D_N}\cE_N} % expected excess risk
\newcommand{\EGG}{\E^{D_N}\abs{\cG_N}} % expected gen gap

\newcommand{\set}[1]{\mathsf{#1}}
\DeclarePairedDelimiterX{\iptemp}[2]{\langle}{\rangle}{#1, #2}
\newcommand{\ip}{\iptemp}
\DeclarePairedDelimiterX{\normtemp}[1]{\lVert}{\rVert}{#1}
\newcommand{\norm}{\normtemp}
\DeclarePairedDelimiterX{\abstemp}[1]{\lvert}{\rvert}{#1}
\newcommand{\abs}{\abstemp}
\DeclarePairedDelimiterX{\trtemp}[1]{(}{)}{#1}
\newcommand{\tr}{\operatorname{tr}\trtemp}
\DeclarePairedDelimiterX{\SEtemp}[2]{(}{)}{#1, #2}
\newcommand{\SE}{\operatorname{SE}\SEtemp}
\DeclarePairedDelimiterX{\SGtemp}[1]{(}{)}{#1}
\newcommand{\SG}{\operatorname{SG}\SGtemp}

\newcommand{\defeq}{\coloneqq} % definition equal in math mode
\newcommand{\eqdef}{\eqqcolon} % equal definition in math mode
\newcommand{\mmin}{\wedge} % stat min
\newcommand{\mmax}{\vee} % stat max
\newcommand{\condbar}{\, \vert \,}
\newcommand{\lcondbar}{\, \big\vert \,}

\newcommand{\lllcondbar}{\, \bigg\vert \,}
\DeclarePairedDelimiterX{\floor}[1]{\lfloor}{\rfloor}{#1} % floor

\renewcommand{\P}{\operatorname{\mathbb{P}}} % probability
\newcommand{\E}{\operatorname{\mathbb{E}}} % expectation
\newcommand{\comp}{\textsf{c}} % complement of set
\newcommand{\Var}{\operatorname{Var}} % variance
\newcommand{\Cov}{\operatorname{Cov}} % covariance (operator)
\newcommand{\one}{\mathbbm{1}} % indicator function
\newcommand{\diid}{\stackrel{\mathrm{iid}}{\sim}} % distributed iid tilde symbol
\newcommand{\normal}{\mathcal{N}} % normal distribution
\newcommand{\avgn}[2]{\overline{{#1}{#2}}^{(N)}} % average of product notation

\newcommand{\grad}{\nabla}

\def\qas{\quad\text{as}\quad}
\def\qa{\quad\text{and}\quad}
\def\qf{\quad\text{for}\quad}
\def\qw{\quad\text{where}\quad}
\def\qin{\quad\text{in}\quad}
\def\qor{\quad\text{or}\quad}

\newcommand{\sfit}[1]{\textsf{\small{#1}}} % item label font

\ifpdf
\DeclareGraphicsExtensions{.eps,.pdf,.png,.jpg}
\else
\DeclareGraphicsExtensions{.eps}
\fi
\graphicspath{ {figures/} }

\usepackage{enumitem}
\setlist[enumerate]{leftmargin=.5in}
\setlist[itemize]{leftmargin=.5in}

\newsiamremark{remark}{Remark}
\newsiamremark{hypothesis}{Hypothesis}
\crefname{hypothesis}{Hypothesis}{Hypotheses}
\newsiamthm{claim}{Claim}
\newsiamremark{example}{Example}
\newsiamremark{notation}{Notation}
\newsiamthm{result}{Result}
\newsiamthm{assumption}{Assumption}
\crefname{assumption}{Assumption}{Assumptions}
\newsiamthm{condition}{Condition}
\crefname{condition}{Condition}{Conditions}
\newsiamthm{problem}{Problem}
\newsiamthm{fact}{Fact}
\crefname{fact}{Fact}{Facts}

\headers{Learning Linear Operators from Noisy Data}{M. V. de Hoop, N. B. Kovachki, N. H. Nelsen, and A. M. Stuart}

\author{Maarten V. de Hoop\thanks{Simons Chair in Computational and Applied Mathematics and Earth Science, Rice University, Houston, TX 77005 USA (\email{mdehoop@rice.edu}).}
	\and Nikola B. Kovachki\thanks{NVIDIA AI, NVIDIA, Santa Clara, CA 95051 USA (\email{nkovachki@nvidia.com}).}
	\and Nicholas H. Nelsen\thanks{Division of Engineering and Applied Science, California Institute of Technology, Pasadena, CA 91125 USA (\email{nnelsen@caltech.edu}, \email{astuart@caltech.edu}).}
	\and Andrew M. Stuart\footnotemark[4]
}

\title{Convergence Rates for Learning Linear Operators from Noisy Data\thanks{Received by the editors August 30, 2021; accepted for publication (in revised form) November 1, 2022.
		\funding{MVdH is supported by the Simons Foundation under the MATH + X program, U.S. Department of Energy, Office of Basic Energy Sciences, Chemical Sciences, Geosciences, and Biosciences Division under grant number DE-SC0020345, the National Science Foundation (NSF) under grant DMS-1815143, and the corporate members of the Geo-Mathematical Imaging Group at Rice University. NHN is supported by the NSF Graduate Research Fellowship Program under grant DGE-1745301. AMS is supported by NSF (grant DMS-1818977) and AFOSR (MURI award number FA9550-20-1-0358 -- Machine Learning and Physics-Based Modeling and Simulation). NBK, NHN, and AMS are supported by NSF (grant AGS-1835860) and ONR (grant N00014-19-1-2408).
}}}

\ifpdf
\hypersetup{
  pdftitle={Convergence Rates for Learning Linear Operators from Noisy Data},
  pdfauthor={M. V. de Hoop, N. B. Kovachki, N. H. Nelsen, and A. M. Stuart}
}
\fi

\begin{document}

\maketitle

% REQUIRED
\begin{abstract}
This paper studies the learning of linear operators between infinite-dimensional Hilbert spaces. The training data comprises pairs of random input vectors in a Hilbert space and their noisy images under an unknown self-adjoint linear operator. Assuming that the operator is diagonalizable in a known basis, this work solves the equivalent inverse problem of estimating the operator's eigenvalues given the data. Adopting a Bayesian approach, the theoretical analysis establishes posterior contraction rates in the infinite data limit with Gaussian priors that are not directly linked to the forward map of the inverse problem. The main results also include learning-theoretic generalization error guarantees for a wide range of distribution shifts. These convergence rates quantify the effects of data smoothness and true eigenvalue decay or growth, for compact or unbounded operators, respectively, on sample complexity. Numerical evidence supports the theory in diagonal and non-diagonal settings.
\end{abstract}

% REQUIRED

% REQUIRED

\section{Introduction}\label{sec:intro}
The supervised learning of operators between Hilbert spaces provides a natural framework for the acceleration of scientific computation and discovery. This framework can lead to fast surrogate models that approximate expensive existing models or to the discovery of new models that are consistent with observed data when no first principles model exists. To develop some of the fundamental principles of operator learning, this paper concerns (Bayesian) nonparametric linear regression under random design. Let \( H \) be a real infinite-dimensional Hilbert space and $ L $ be an unknown---possibly unbounded and in general densely 
defined on $H$---self-adjoint linear operator from its domain in $ H $ into 
$H$ itself. We study the following linear operator learning problem.
\subparagraph{Main Problem}\label{prob:intro_model}
\emph{Let \( \{x_n\}\subset H \) be random design vectors and $ \{\xi_n\} $ be noise vectors. Given the training data pairs $ \{(x_n,y_n)\}_{n=1}^{N} $ with sample size $ N\in\N $, where
	\begin{equation}\label{eqn:intro_model}
	y_n=Lx_n + \gamma\xi_n \qf n\in\setoneton\qa \gamma>0\, ,
	\end{equation}
	find an estimator \( L^{(N)} \) of $ L $ that is accurate when evaluated outside of the samples \( \{x_n\} \).
}

The estimation of $ L $ from the data \cref{eqn:intro_model} is generally an ill-posed linear inverse problem \cite{vito2005learning}. In principle, the chosen reconstruction procedure should be consistent: the estimator \( L^{(N)} \) converges to the true \( L \) as \( N\to\infty \). The rate of this convergence is equivalent to the \emph{sample complexity} of the estimator, which determines the efficiency of statistical estimation. The sample complexity \( N(\ep)\in\N \) is the number of samples required for the estimator to achieve an error less than a fixed tolerance \( \ep>0 \). It quantifies the difficulty of \hyperref[prob:intro_model]{Main Problem}.

In modern scientific machine learning problems where operator learning is used, the demand on data from different operator learning architectures often outpaces the availability of computational or experimental resources needed to generate the data. Ideally, theoretical analysis of sample complexity should reveal guidelines for how to reduce the requisite data volume. To that end, the broad purpose of this paper is to provide an answer to the question:
\begin{center}
	\vspace{2mm}
	\emph{What factors can reduce sample size requirements for linear operator learning}?
	\vspace{2mm}
\end{center}

Our goal is not to develop a practical procedure to regress linear operators between infinite-dimensional vector spaces. Various methods already exist for that purpose, including those based on (functional) principal component analysis (PCA) \cite{bhattacharya2020model,crambes2013asymptotics,hormann2015note}. Instead, we aim to strengthen the rather sparse but slowly growing theoretical foundations of operator learning.
	
We overview our approach to solve \hyperref[prob:intro_model]{Main Problem} in \cref{sec:intro_ideas}. We summarize one of our main convergence results in \cref{sec:intro_main_result}. In \cref{sec:intro_examples}, we illustrate examples to which our theory applies. \Cref{sec:intro_literature} surveys work related to ours. The primary contributions of this paper and its organization are given in \cref{sec:intro_contribution,sec:intro_outline}, respectively.

\subsection{Key ideas}\label{sec:intro_ideas}
In this subsection, we communicate the key ideas of our methodology at an informal level and distinguish our approach from similar ones in the literature.

\subsubsection{Operator learning as an inverse problem}\label{sec:intro_ideas_ol_as_ip}
We cast \hyperref[prob:intro_model]{Main Problem} as a Bayesian inverse problem with a \emph{linear operator as the unknown object to be inferred from data}. Suppose the input training data \( \{x_n\} \) from \cref{eqn:intro_model} are independent and identically distributed (i.i.d.) according to a (potentially unknown) centered probability measure \( \nu \) on \( H \) with finite second moment. Let \( \Lambda\colon H\to H \) be the covariance operator of \( \nu \) with orthonormal eigenbasis \( \{\phi_k\} \). Let the \( \{\xi_n\} \) be i.i.d. \( \normal(0,\id_H) \) Gaussian white noise processes independent of \( \{x_n\} \). Writing \( Y=(y_1,\ldots, y_N) \), \( X=(x_1,\ldots, x_N) \), and \( \Xi=(\xi_1,\ldots,\xi_N) \) yields the concatenated data model
\begin{equation}\label{eqn:intro_ideas_ip_operator}
Y=K_{X} L + \gamma \Xi\, .
\end{equation}
The forward operator of this inverse problem is \( K_{X}\colon T \mapsto (Tx_1,\ldots, Tx_N) \). Under a Gaussian prior \( L\sim \normal(0,\Sigma) \), the solution is the Gaussian posterior \( L\condbar (X,Y)\). For a fixed orthonormal basis \( \{\varphi_j\} \) of \( H \), it will be convenient to identify \cref{eqn:intro_ideas_ip_operator} with the countable inverse problem
\begin{equation}\label{eqn:intro_ideas_ip_matrix}
y_{jn} = \sum\nolimits_{k=1}^{\infty}x_{kn}\Lmat_{jk}\ + \gamma\xi_{jn} \qf j\in\N \qa n\in\setoneton\,,
\end{equation}
where \(\xi_{jn}\diid\normal(0,1)\), \( x_{kn}=\ip{\phi_k}{x_n}_{H} \), and \( \Lmat_{jk}=\ip{\varphi_j}{L\phi_k}_{H} \). See \cref{sec:setup_bayes_model} for details.

\subsubsection{Comparison to nonparametric inverse problems}\label{sec:intro_ideas_compare}
In contrast, most theoretical studies of Bayesian inverse problems concern the \emph{estimation of a vector} \( f\in H_1 \) from data
\begin{equation}\label{eqn:intro_ideas_knapik_ip}
Y'=Kf+N^{-1/2}\xi\,,\qw \xi\sim\normal(0,\id_{H_2})
\end{equation}
and \( K\colon H_1\to H_2 \) is a known bounded linear operator between Hilbert spaces \( H_1 \) and \( H_2 \). This is a signal in white noise model. 
Under a prior on \( f \), the asymptotic behavior of the posterior \( f\condbar Y' \) as the noise tends to zero (\( N\to\infty \)) is of primary interest. Many analyses of \cref{eqn:intro_ideas_knapik_ip} consider the \emph{singular value decomposition} (SVD) of \( K\) \cite{agapiou2018posterior,agapiou2014bayesian,cavalier2008nonparametric,knapik2011bayesian,ray2013bayesian}. Projecting \( f \) into its coordinates \( \{f_k\} \) in the basis of right singular vectors \( \{\phi'_k\} \) of \( K \) and writing \( \{Y_j'\} \) for observations of the stochastic process \( Y' \) on the basis of left singular vectors of \( K \) yields
\begin{equation}\label{eqn:intro_ideas_knapik_sequence}
Y_j' = \kappa_j f_j + N^{-1/2}\xi_j \qf j\in\N\, ,
\end{equation}
where the \( \{\xi_j\} \) are i.i.d. \( \normal(0,1) \) and \( \{\kappa_j\} \) are the singular values of \( K \). Obtaining a sequence space model of this form is always possible if \( K \) is a compact operator \cite[sect. 1.2]{cavalier2008nonparametric}.

Some notable differences between the traditional inverse problem \cref{eqn:intro_ideas_knapik_ip} and the operator learning inverse problem \cref{eqn:intro_ideas_ip_operator} are evident. \Cref{eqn:intro_ideas_ip_operator} is directly tied to (functional) regression, while \cref{eqn:intro_ideas_knapik_ip} is not. The unknown \( f \) is a vector while \( L \) is an unknown operator. A more major distinction is that \( K \) in \cref{eqn:intro_ideas_knapik_ip} is deterministic and arbitrary, while \( K_X \) in \cref{eqn:intro_ideas_ip_operator} is a \emph{random forward map} defined by point evaluations. Their sequence space representations also differ. \Cref{eqn:intro_ideas_knapik_sequence} is diagonal with a singly-indexed unknown \( \{f_j\} \), while \cref{eqn:intro_ideas_ip_matrix} is non-diagonal (because the SVD of \( K_X \) was not invoked) with a doubly-indexed unknown \( \{\Lmat_{jk}\} \).
Thus, our work deviates significantly from existing studies.

\subsubsection{Diagonalization leads to eigenvalue learning}\label{sec:intro_ideas_diag}
The technical core of this paper concerns the sequence space representation \cref{eqn:intro_ideas_ip_matrix} of \hyperref[prob:intro_model]{Main Problem} in the ideal setting that a diagonalization of $L$ is known.
\begin{assumption}[diagonalizing eigenbasis given for $L$]\label{assump:intro_ideas_diagonal}
	The unknown linear operator $L$ from \hyperref[prob:intro_model]{Main Problem} is diagonalized in the known orthonormal basis $\{\varphi_j\}_{j\in\N}\subset H$.
\end{assumption}
Under this assumption and denoting the eigenvalues of \( L \) by \( \{l_{j}\} \), \cref{eqn:intro_ideas_ip_matrix} simplifies to
\begin{equation}\label{eqn:intro_ideas_ip_diagonal}
y_{jn} = \ip{\varphi_j}{x_n}_{H}\,l_j\ + \gamma\xi_{jn} \qf j\in\N \qa n\in\setoneton\,.
\end{equation}
In general, the random coefficient \( \ip{\varphi_j}{x_n}_{H} \) depends on every \( \{x_{kn}\}_{k\in\N} \) from \cref{eqn:intro_ideas_ip_matrix}. 
To summarize, under \cref{assump:intro_ideas_diagonal} we obtain a white noise sequence space regression model with \emph{correlated random coefficients}. Inference of the full operator is reduced to only that of its eigenvalue sequence.  \Cref{eqn:intro_ideas_ip_diagonal} is at the heart of our analysis of linear operator learning. The convergence results we establish for this model may also be of independent interest.

Our proof techniques in this diagonal setting closely follow those in the paper \cite{knapik2011bayesian}, which studies posterior contraction for \cref{eqn:intro_ideas_knapik_sequence} in a simultaneously diagonalizable Bayesian setting. However, our work exhibits some crucial differences with \cite{knapik2011bayesian} which we now summarize.

\begin{enumerate}[label={(D\arabic*)}]
	\item \label{item:intro_ideas_diag_map}
	(\sfit{forward map}) The coefficients \( \{\ip{\varphi_j}{x_n}_{H}\} \) in our problem \cref{eqn:intro_ideas_ip_diagonal} are random variables (r.v.s), while in \cite{knapik2011bayesian} the singular values \( \{\kappa_j\} \) in \cref{eqn:intro_ideas_knapik_sequence} are fixed by \( K \). Also, the law of \( \{\ip{\varphi_j}{x_n}_{H}\} \) may not be known in practice; only the samples \( \{x_n\} \) may be given.
	
	\item \label{item:intro_ideas_diag_link}
	(\sfit{link condition}) Unlike in \cite{knapik2011bayesian}, our prior covariance operator \( \Sigma \) \emph{is not linked to the SVD} of the forward map \( K_X \). That is, we do not assume simultaneous diagonalizability.
	
	\item \label{item:intro_ideas_diag_prior}
	(\sfit{prior support}) The Gaussian prior we induce on \( \{l_j\} \) is supported on a (potentially) much larger sequence space in the scale \( \cH^{s} \) (relative to \( \{\varphi_j\} \), with \( s\in\R \)),\footnote{
		The Sobolev-like sequence Hilbert spaces \( \cH^s=\cH^s(\N;\R) \) are defined for \( s\in\R \) by
		\[
		\cH^{s}(\N;\R)\defeq \bigl\{v\colon\N\to\R \lcondbar \textstyle\sum_{j=1}^{\infty}j^{2s}\abs{v_j}^{2}<\infty\bigr\}\,.
		\]
		They are equipped with the natural
		$ \{j^{s}\} $-weighted $ \ell^2(\N;\R) $ inner-product and norm. We will usually interpret these spaces as defining a smoothness scale \cite[sect. 2]{gugushvili2020bayesian} of vectors relative to the orthonormal basis \( \{\varphi_j\} \) of \( H \).
	}
	instead of just the space \( \ell^{2}(\N;\R) \) (relative to \( \{\phi'_k\} \)) charged by the prior on \( \{f_j\} \) in \cite{knapik2011bayesian}.
	
	\item \label{item:intro_ideas_diag_norm}
	(\sfit{reconstruction norm}) Solution convergence for \cref{eqn:intro_ideas_ip_diagonal} is in \( \cH^{-s} \) norms relative to \( \{\varphi_j\} \), while only the \( \ell^{2}(\N;\R) \) norm relative to \( \{\phi'_k\} \) (i.e., \( H_1 \) norm) is considered in \cite{knapik2011bayesian}.
\end{enumerate}

These differences deserve further elaboration.
\subparagraph*{\Cref{item:intro_ideas_diag_map}}
If \( x_n\in H \) almost surely (a.s.), then \( \ip{\varphi_j}{x_n}_{H}\to 0 \) a.s. as \( j\to\infty \) in \cref{eqn:intro_ideas_ip_diagonal}, just as \( \kappa_j \to 0 \) if \( K \) in \cref{eqn:intro_ideas_knapik_ip} is compact. However, we later observe that our \( K_X \) \emph{is not compact}.

\subparagraph*{\Cref{item:intro_ideas_diag_link}}
The authors in \cite{knapik2011bayesian} assume that the eigenbasis of the prior covariance of \( f \) is precisely \( \{\phi'_k\} \), the right singular vectors of \( K \) in \cref{eqn:intro_ideas_knapik_ip}. This direct link condition between the prior and \( K \) ensures that the implied prior (and posterior) on \( \{f_j\} \) is an infinite product measure. Our analysis of \cref{eqn:intro_ideas_ip_diagonal} still induces an infinite product prior on \( \{l_j\} \) \emph{without using the SVD of the forward operator \( K_X \)}. Instead, we make mild assumptions that only weakly link \( K_X \) to the prior covariance operator \( \Sigma \). See \cref{eqn:intro_ideas_linkcov} for a relevant smoothness condition.

\subparagraph*{\Cref{item:intro_ideas_diag_prior}}
The reason we work with a sequence prior having support on sets larger than \( \ell^2 \) is to include \emph{unbounded operators} (with eigenvalues \( \abs{l_j}\to\infty \) as \( j\to\infty \)) in the analysis.\footnote{Note, however, that unbounded operators with continuous spectra \cite{colbrook2021computing} are beyond the scope of this paper.}

\subparagraph*{\Cref{item:intro_ideas_diag_norm}}\label{par:item:intro_ideas_diag_norm}
Only the \( H_1 \) estimation error is considered in \cite{knapik2011bayesian} because the unknown quantity is a vector \( f\in H_1 \). Since our unknown is an operator, we also consider the \emph{prediction error} \cite{cai2006prediction} on new test inputs (see \cref{sec:setup_bayes_testerror}). This relates to the \( \cH^{-s} \) norms in \ref{item:intro_ideas_diag_norm}.

\subsection{Main result}\label{sec:intro_main_result}
Here and in the sequel, we assume that there is some fixed ground truth operator that underlies the observed output data.
\begin{assumption}[true linear operator]\label{assump:intro_ideas_truth}
	The data \( Y \), observed as \( \{y_{jn}\} \) in \cref{eqn:intro_ideas_ip_diagonal}, is generated according to \cref{eqn:intro_ideas_ip_operator} for a fixed self-adjoint linear operator \( L=\Ld \) with eigenvalues \( \{\ld_j\} \).
\end{assumption}

Under \cref{eqn:intro_ideas_ip_diagonal}, we study the performance of the posterior \( \{l_j\}\condbar (X,Y)\) (and related point estimators) for estimating the true \( \{\ld_j\}\) in the limit of infinite data. The following concrete theorem is representative of more general convergence results established later in the paper.
\begin{theorem}[asymptotic convergence rate with Gaussian design]\label{thm:intro_ideas_thm}
	Suppose \cref{assump:intro_ideas_diagonal,assump:intro_ideas_truth} hold with \( \{\ld_j\}\in\cH^{s} \) for some \( s\in\R \). Let \( \nu=\normal(0,\Lambda) \) be a Gaussian measure satisfying
	\begin{equation}\label{eqn:intro_ideas_linkcov}
	c_1^{-1}j^{-2\al}\leq \ip{\varphi_j}{\Lambda\varphi_j}_{H} \leq c_1j^{-2\al} \quad\text{for all sufficiently large}\ \,  j\in\N
	\end{equation}
	for some \( c_1\geq 1 \) and \( \al > 1/2 \). Let \( \bigotimes_{j=1}^{\infty}\normal(0,\sigma_j^2) \) be the prior on \( \{l_j\} \) in \cref{eqn:intro_ideas_ip_diagonal} with variances \( \{\sigma_j^2\} \) satisfying \( c_2^{-1}j^{-2p}\leq \sigma_j^2\leq c_2j^{-2p} \) for all sufficiently large \( j\in\N \) for some \( c_2\geq 1 \) and \( p\in\R \). Denote by \( P^{D_N} \) the posterior distribution for \( \{l_j\} \) arising from the observed data \( D_N\defeq (X,Y) \).
	Fix \( \al'\in [0,\al+1/2) \).  If \( \min\{\al,\al'\} + \min\{p - 1/2, s\} >0 \), then there exists a constant \( C>0 \), independent of the sample size \( N \), such that
	\begin{equation}\label{eqn:intro_ideas_thm}
	\E^{D_N}\E^{\{l_i^{(N)}\}_{i=1}^{\infty}\sim P^{D_N}}\sum\nolimits_{j=1}^{\infty} j^{-2\al'}\abs[\big]{\ld_j - l_j^{(N)}}^2
	\leq CN^{-\bigl(\frac{\al'+\min\{p-1/2, s\}}{\al+p}\bigr)}
	\end{equation}
	for all sufficiently large \( N \). The first expectation in \cref{eqn:intro_ideas_thm} is over the joint law of \( D_N \).
\end{theorem}

\Cref{eqn:intro_ideas_thm} shows that, on average, posterior sample eigenvalue estimates converge in \( \cH^{-\al'} \) to the true eigenvalues of \( \Ld \) in the infinite data limit. The hypothesis \cref{eqn:intro_ideas_linkcov}, which controls the regularity of the data $\{x_n\}$, is immediately satisfied if, e.g., \( \Lambda \) is a Mat\'ern-like covariance operator with eigenvectors \( \{\varphi_j\} \). \Cref{thm:intro_ideas_thm}, whose proof is in \cref{app:proofs}, is a consequence of \cref{thm:upper-expect}, which is valid for a much larger class of input data measures.

Nonetheless, \cref{thm:intro_ideas_thm} nearly tells the whole story. The convergence rate exponent in \cref{eqn:intro_ideas_thm} shows that the regularity of the ground truth, data, and prior each have an influence on sample complexity. \Cref{fig:intuition_lap} summarizes this complex relationship. The figure, and this paper more generally, reveals three fundamental principles of (linear) operator learning: 
\begin{enumerate}[label={(P\arabic*)}]
	\item \label{item:intro_ideas_result_output}
	(\sfit{smoothness of outputs}) The ground truth operator becomes statistically more efficient to learn whenever the smoothness of its (noise-free) outputs \emph{increases}. Moreover, as the degree of smoothing of the operator \emph{increases}, sample complexity improves.
	
	\item \label{item:intro_ideas_result_input}
	(\sfit{smoothness of inputs}) \emph{Decreasing} the smoothness of input training data improves sample complexity (in norms that do not depend on the training distribution itself).\footnote{
		If the norm used to measure error depends on the training data distribution, this may no longer be true. For example, in-distribution error (train and test on the same distribution) would correspond in \cref{thm:intro_ideas_thm} to setting \( \al'=\al \) (see \cref{sec:setup_bayes_testerror}). In this case, \emph{increasing} \( \al \) would improve sample complexity.
	}
	
	\item \label{item:intro_ideas_result_shift}
	(\sfit{distribution shift}) As the smoothness of samples from the input test distribution \emph{increases}, average out-of-distribution prediction error improves.
\end{enumerate}

Below, we discuss how the principles \ref{item:intro_ideas_result_output}~to~\ref{item:intro_ideas_result_shift} manifest in \cref{thm:intro_ideas_thm} and \cref{fig:intuition_lap}.

\subparagraph*{\cref{item:intro_ideas_result_output}}
In \cref{thm:intro_ideas_thm}, the left side of~\cref{eqn:intro_ideas_thm} is equivalent to the expected prediction error over some input test distribution (see \cref{sec:setup_bayes_testerror} for details). Increasing \( \al' \) increases the regularity of test samples. Assuming for simplicity that \( s=p-1/2 \), the convergence rate in \cref{eqn:intro_ideas_thm} is \( N^{-(\al' + s)/(\al + s + 1/2)} \) as \( N\to\infty \). Thus, besides large \( \al' \), it is beneficial to have large regularity exponents \( \al'+s \) of the operator's evaluation on sampled test inputs or large regularity exponents \( s \) of the true operator's eigenvalues. Indeed, \cref{fig:sweep_rate_al,fig:sweep_rate_out,fig:sweep_rate_forward_inverse} suggest that unbounded operators (whose eigenvalues grow without bound) are more difficult to learn than bounded (eigenvalues remain bounded) or compact ones (eigenvalues decay to zero).

\subparagraph*{\cref{item:intro_ideas_result_input}}
\emph{Training inputs} with low smoothness are favorable. This is quantified in \cref{thm:intro_ideas_thm} by \emph{decreasing} \( \al \), which means that the \( \{x_n\} \) become ``rougher'' (\cref{fig:sweep_rate_al}).

\subparagraph*{\cref{item:intro_ideas_result_shift}}
\Cref{fig:sweep_rate_out} illustrates that \emph{increasing} \( \al' \) in \cref{thm:intro_ideas_thm} improves the error.

We reinforce \cref{item:intro_ideas_result_output,item:intro_ideas_result_input,item:intro_ideas_result_shift} throughout the rest of the paper.

\begin{figure}[htbp]%
	\centering
	\subfloat[Varying $ \al $ and \( s \)]{
		\includegraphics[width=0.32\textwidth]{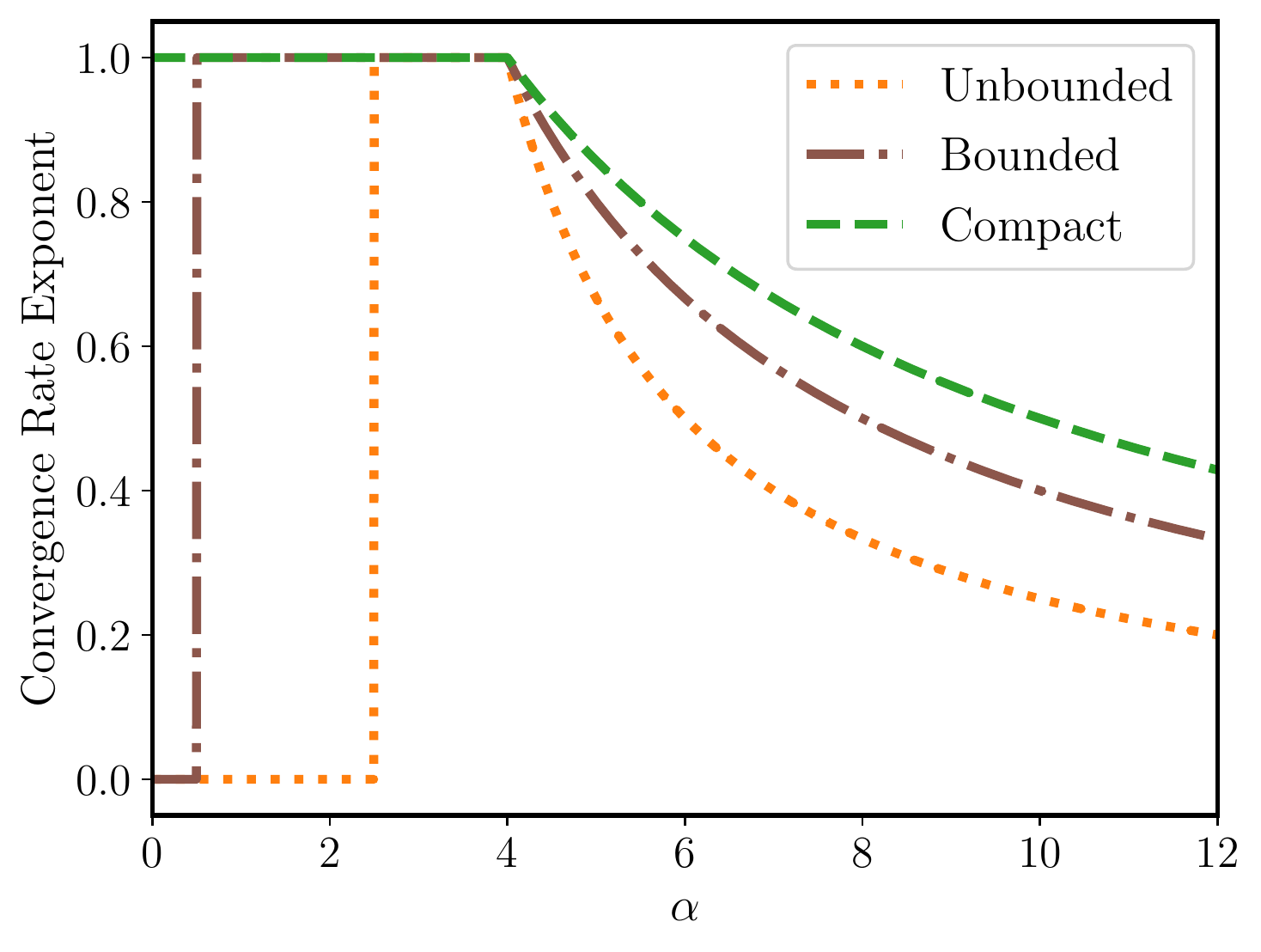}
		\label{fig:sweep_rate_al}}\hfill%
	\subfloat[Varying $ \al' $ and \( s \)]{
		\includegraphics[width=0.304\textwidth]{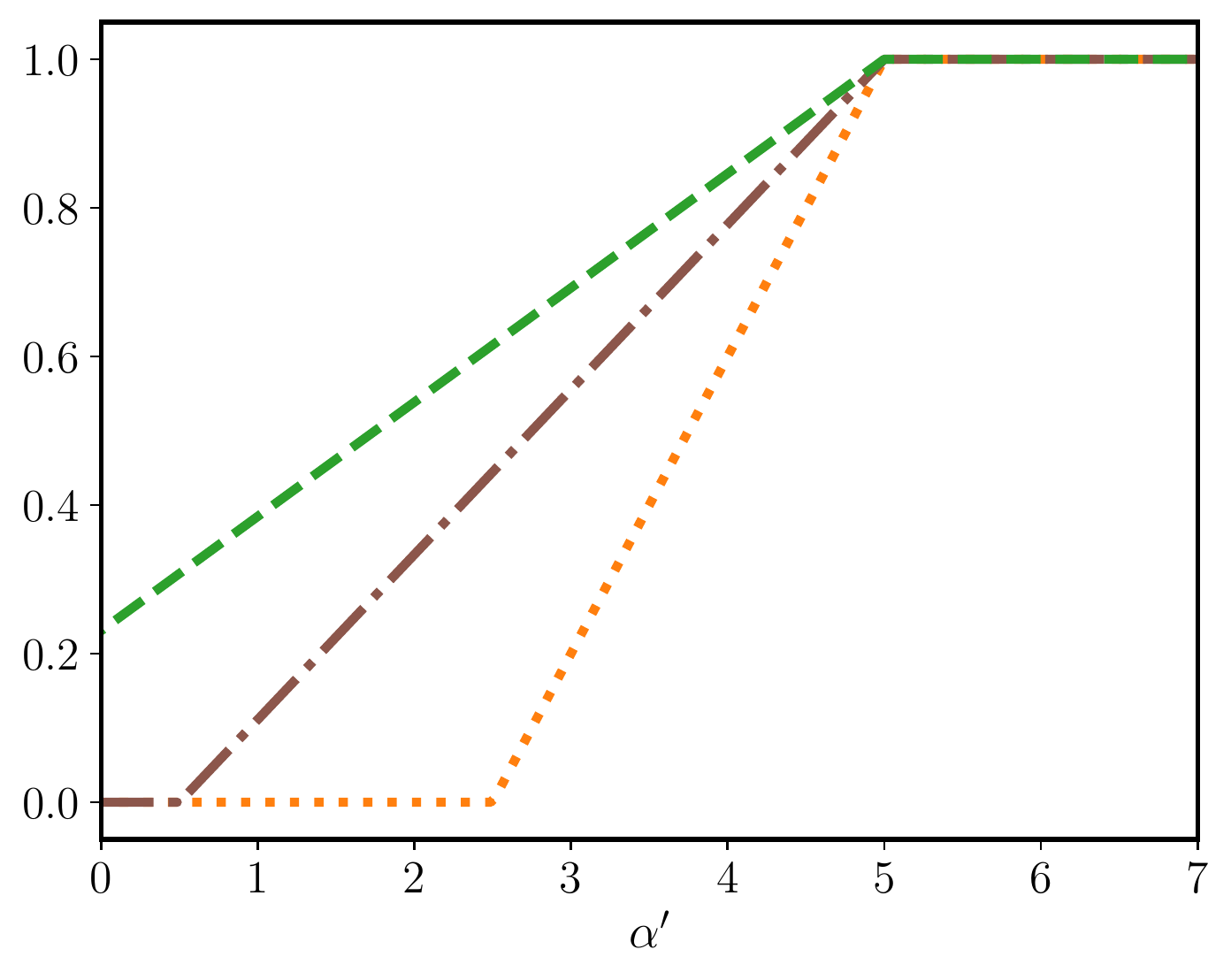}
		\label{fig:sweep_rate_out}}\hfill%
	\subfloat[Varying $ p $, \( \al \), and \( s \)]{
		\includegraphics[width=0.303\textwidth]{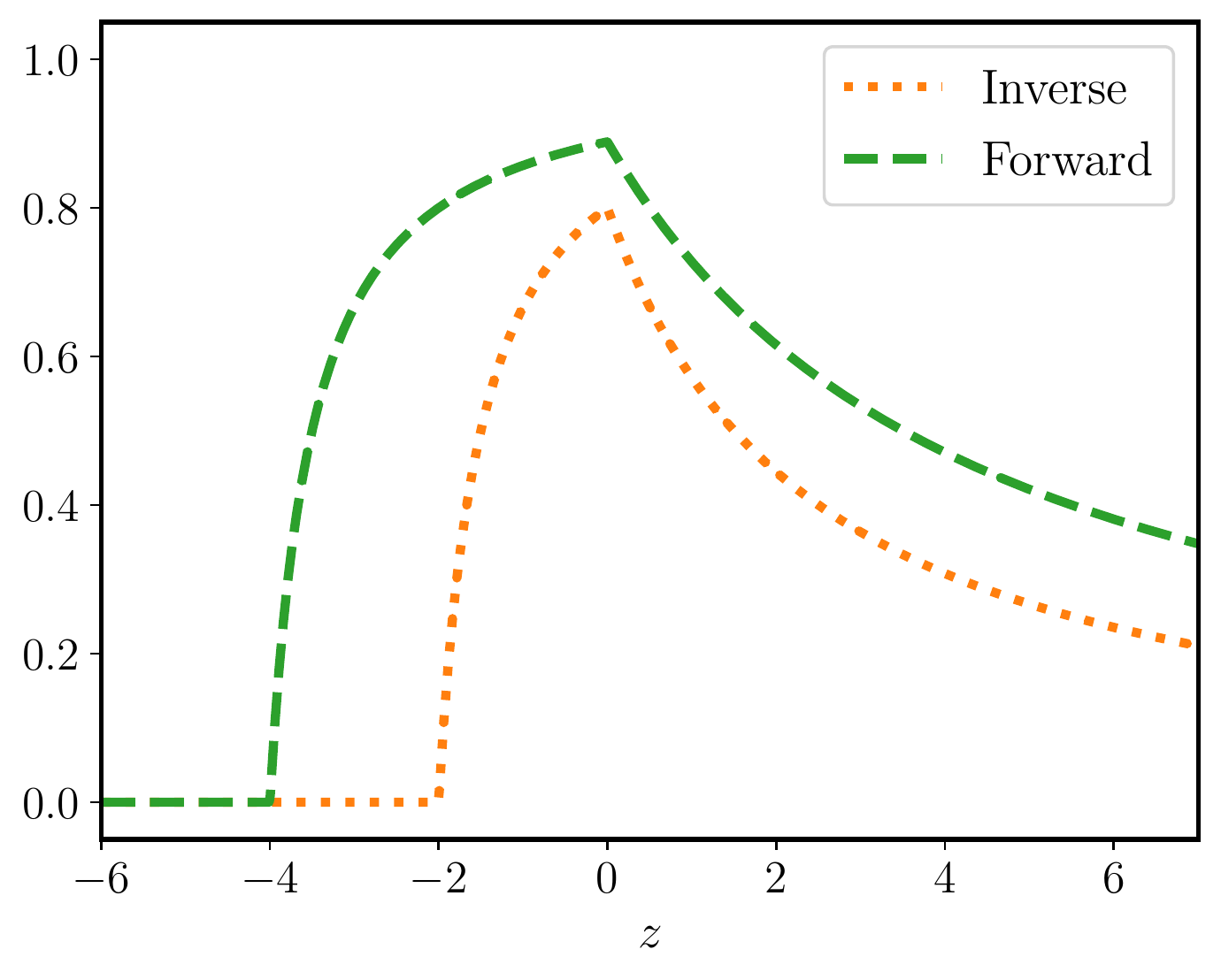}
		\label{fig:sweep_rate_forward_inverse}}
	\caption{Fundamental principles of linear operator learning. The theoretical convergence rate exponents (from \cref{thm:theory_expectation_rvsharp}) corresponding to unbounded $( -\lap ,\, s<-2.5)$, bounded $ (\id,\, s<-1/2) $, and compact $ ((-\lap)^{-1},\, s<1.5) $ true operators are displayed (see principle \ref{item:intro_ideas_result_output} and \cref{sec:numerics_within}). With \( p=s+1/2 \), \cref{fig:sweep_rate_al} \( (\al'=4.5) \) and \cref{fig:sweep_rate_out} \( (\al=4.5) \) illustrate the effects that varying input training data and test data smoothness have on convergence rates, respectively (principles \ref{item:intro_ideas_result_input}~and~\ref{item:intro_ideas_result_shift}). \Cref{fig:sweep_rate_forward_inverse} shows that learning the unbounded ``inverse map'' $-\lap$ $(\text{with } \al =\al'= 4.5) $ is always harder than learning the compact ``forward map'' $(-\lap)^{-1}$ $(\text{with } \al=\al'=2.5) $ as the shift $ z=p-s-1/2 $ in prior regularity is varied (\cref{par:lit_inv_operator}).
	}
	\label{fig:intuition_lap}
\end{figure}

\subsection{Examples}\label{sec:intro_examples}
Although quite a strong assumption, the known diagonalization from \cref{assump:intro_ideas_diagonal} is still realizable in practice. For instance, there may be prior knowledge that the data covariance operator commutes with the true operator (and hence shares the same eigenbasis) or that the true operator obeys known physical principles (e.g., commutes with translation or rotation operators). Regarding the latter, in \cite{portone2021bayesian} the authors infer the eigenvalues of a differential operator closure for an advection-diffusion model from indirect observations. As in \cite{trabs2018bayesian}, the operator could be known up to some uncertain parameter. This is the case for several smoothing forward operators that define commonly studied linear inverse problems, including the severely ill-posed inverse boundary problem for the Helmholtz equation with unknown wavenumber parameter \cite[sect. 5]{agapiou2014bayesian} or the inverse heat equation with unknown scalar diffusivity parameter \cite[sect. 6.1]{trabs2018bayesian}. In both references, the eigenbases are already known. Thus, our learning theory applies to these uncertain operators: taking \( s \) and \( p \) large enough in \cref{eqn:intro_ideas_thm} yields prediction error rates of convergence as close to \( N^{-1} \) as desired.

More concretely, the theory  in this paper may be applied directly to the following examples.
\subsubsection{Blind deconvolution}
Periodic deconvolution on the \( d \)-dimensional torus \( \T^d \) is a linear inverse problem that arises frequently in the imaging sciences. The goal is to recover a periodic signal \( f\colon \T^d\to\C \) from noisy measurements
\[
y= \mu*f + \eta\,,\qw \mu*f\defeq \int_{\T^d}f(\cdot - t)\,\mu(dt) \quad \text{and $\eta$ is noise},
\]
of its convolution with a filter \( \mu \). The filter may be identified with a periodic signal or more generally with a signed measure \cite[sect. 6.2]{trabs2018bayesian}. However, \( \mu \) is sometimes unknown; this leads to \emph{blind} or \emph{semi-blind deconvolution}. One path forward is to first estimate the smoothing operator \( K_{\mu}\colon f\mapsto \mu * f \) from many random \( (f,y) \) pairs under the given model. By the known translation-invariance of the problem, \( K_{\mu} \) is diagonalized in the complex Fourier basis. Inference is then reduced to estimating the Fourier coefficients \( \{\mu_j\} \) of \( \mu\), which are the eigenvalues of \( K_{\mu} \). Since \( \{\mu_j\}\in\cH^s \) for some \( s\in\R \), \cref{thm:intro_ideas_thm} provides a convergence rate.

\subsubsection{Radial EIT}
\emph{Electrical impedance tomography} (EIT) is a non-invasive imaging procedure that is used in medical, industrial, and geophysical applications \cite{mueller2012linear}. Abstractly, EIT concerns the following severely ill-posed nonlinear inverse problem. Let \( \disk\subset \R^2 \) be the unit disk and let \( \sigma\colon \disk\to \R_{>0}\) be the strictly positive electrical conductivity of a medium. With electric potential \( u\colon \disk\to \R \) governed by the elliptic partial differential equation (PDE)
\[
-\grad\cdot(\sigma\grad u) = 0\qin \disk\,,
\]
the goal is to reconstruct the unknown conductivity \( \sigma \) in \( \disk \) from voltage and current boundary measurements of \( u \). These are modeled (to infinite precision) by the linear operators
\[
\dtn\colon u\vert_{\partial \disk}\mapsto \sigma\frac{\partial u}{\partial n}\Big\vert_{\partial \disk} \qor \ntd\colon \sigma\frac{\partial u}{\partial n}\Big\vert_{\partial \disk} \mapsto  u\vert_{\partial \disk}\,,
\]	
where \( \partial/\partial n \) is the outward normal derivative. In practical EIT, either \( \dtn \) or \( \ntd \) must be recovered from finite data. One way to solve this \emph{data completion} step \cite{bui2022bridging} involves making random boundary measurements and employing operator learning \cref{eqn:intro_model}. If \( \sigma \) is \emph{radial}, then \( \dtn \) and \( \ntd \) are diagonalized in the complex Fourier basis over \( \partial \disk= \T^1 \) \cite[sect. 13.1]{mueller2012linear}. In this case, the theory in this paper immediately applies to learn the eigenvalues of both operators.

\subsection{Related work}\label{sec:intro_literature}
A natural setting to apply operator learning is one in which the ambient Hilbert space $H$ comprises real-valued
functions over a domain $ \set{D}\subset\R^d$. For example, there is an emerging body of work focused on learning surrogates for forward, typically nonlinear, solution operators of PDEs
\cite{adcock2020deep,bhattacharya2020model, korolev2021two,li2020neural,lu2019deeponet, nelsen2020random,o2020derivative,schwab2019deep}. In the context
of dynamical systems, there is literature focused on learning the Koopman operator or its generator, both linear operators,
from time series data \cite{brunton2016discovering,giannakis2019data,klus2020data,klus2020eigendecompositions,patel2020physics}. There also is interest in speeding up (Bayesian) inversion techniques with forward surrogates \cite{li2020neural} and in directly learning regularizers for inversion \cite{alberti2021learning,arridge2019solving} (or even entire regularized inverse solution operators \cite{aspri2019data,colbrook2022difficulty,de2019deep}). However, more theory is needed to quantify the difficulty of learning forward versus inverse operators that arise in these contexts. Some sharp theory already exists for nonlinear operator learning. For example, the authors of \cite{caponnetto2007optimal,rastogi2020convergence} establish optimal convergence rates for direct and inverse least squares regression problems with both infinite-dimensional input \emph{and} output spaces under the condition that point evaluation is a Hilbert--Schmidt operator. However, this condition never holds when $ H $ is infinite-dimensional in our linear operator setting \cref{eqn:intro_model}.

We now highlight three subfields that are closely linked to our statistical framework.

\paragraph*{Linear operator learning}
The study of linear function-to-function models within functional data analysis (FDA) \cite{ramsay2005fda} is well-established \cite{crambes2013asymptotics,hormann2015note,reimherr2015functional,wang2020functional1}. Much of this work concerns
the setting $ H=L^2((0,1);\R) $ and linear models based on kernel integral 
operators under colored noise. Operator estimation is then reduced to learning the kernel, usually in a reproducing kernel Hilbert space (RKHS) framework. Linear operator learning has also been considered in machine learning \cite{abernethy2009new}, particularly in the context of conditional expectation operators \cite{mollenhauer2020nonparametric} and conditional mean embeddings \cite{grunewalder2012conditional,klebanov2020rigorous,song2009hilbert}.
The authors of \cite{hormann2015note,reimherr2015functional} study functional linear regression with a spectral operator estimator. This allows them to obtain consistency 
of the prediction error assuming only boundedness of the true 
operator \cite{hormann2015note},
rather than compactness as assumed in much of the FDA literature. Convergence rates are established in \cite{reimherr2015functional}. 
While unbounded operators are not considered in these two works, 
their approaches could likely be modified to handle them.
Relatedly, the authors of \cite{tabaghi2019learning} and \cite{boulle2021learning} share our motivations. The former establishes sample complexities for learning Schatten-class compact operators (motivated by inverse problem solution operators) while the latter for learning compact operators associated to Green's functions of elliptic PDEs (motivated by PDE discovery). Our theory also treats these types of operators but goes further by proving sample complexities for the direct learning of \emph{unbounded operators}, which are of primary interest in these papers (the inverse operator in the former and the partial differential operator in the latter).

\paragraph*{Inverse operator learning}\label{par:lit_inv_operator}
The direct learning of solution operators of inverse problems is currently a popular research area, catalyzed by the success of deep neural networks \cite{arridge2019solving,bubba2020deep,de2019deep,fan2020solving}. However, theoretical analysis in this area is lacking. One difficulty is the interplay between the ill-posedness of the learning and ill-posedness of the inverse problem itself. For a compact operator \( T \), our diagonal theory suggests that learning \( L=T \) under model \cref{eqn:intro_model} is easier than learning the unbounded inverse operator \( L=T^{-1} \) under the same model (\cref{fig:sweep_rate_forward_inverse}). Although less common than the former, the latter setting could arise from noisy differentiation of time series in PDE system identification, for example. One limitation of our theory is that it does not account for \emph{errors-in-covariates} that distinguishes true inverse operator learning, where (a regularized version of) $ T^{-1} $ must be estimated only from noisy forward map samples \cref{eqn:intro_model} with \( L=T \). Total least squares \cite{golub1980analysis} is one solution approach in finite dimensions. The infinite-dimensional setting was considered in \cite{bleyer2013double} but with non-Bayesian methods. Regardless, inverse operator learning in this challenging setting is an important area for future research.

\paragraph*{Bayesian nonparametric statistics}
Although the theoretical analysis of inverse problems with linear operator unknowns is largely absent from the Bayesian nonparametrics literature (see \cref{sec:intro_ideas_compare,sec:intro_ideas_diag}), this literature still has some similarities with \cref{eqn:intro_model,eqn:intro_ideas_ip_operator}.
Many works go beyond \cite{knapik2011bayesian} by deriving posterior contraction rates for problem \cref{eqn:intro_ideas_knapik_ip} without assuming simultaneous diagonalizability of the prior covariance and the forward operator. In \cite{ray2013bayesian}, the author studies linear inverse problems in a non-conjugate setting. However, knowledge of the forward map's SVD is used heavily in the analysis even though the prior is (in one case) represented in a non-SVD basis (one comprised of finite linear combinations of singular vectors). These ideas are generalized in \cite{gugushvili2020bayesian} to priors linked to smoothness scales instead of the SVD. For Gaussian priors not linked to the SVD, new methods were introduced in \cite{monard2019efficient} that yield optimal posterior performance for \( X \)-ray transform inverse problems. These techniques were refined for general linear inverse problems in \cite{giordano2020bernstein}. However, the previous two papers focus on semiparametric inference (i.e., linear functionals) instead of full nonparametric reconstruction (our main interest). The closest work to ours is \cite{trabs2018bayesian}. There, the author studies a linear inverse problem in which the forward map is only known up to an uncertain parameter \( \theta \). Given a noisy observation of \( \theta \) in addition to data of the form \cref{eqn:intro_ideas_knapik_ip}, the author analyzes a Bayesian joint reconstruction procedure. Other papers that use Gaussian priors not linked to the SVD include \cite{agapiou2013posterior,agapiou2021designing,knapik2018general}.
While notable, all of these works mentioned \emph{do not help us extend the results in this paper for \cref{eqn:intro_ideas_ip_diagonal} to non-diagonal linear operator learning \cref{eqn:intro_ideas_ip_matrix} because our framework already avoids the SVD from the start}; see \cref{item:intro_ideas_diag_link} in \cref{sec:intro_ideas_diag}. Removing \cref{assump:intro_ideas_diagonal} while preserving sharp rates will likely require new ideas; see \cref{par:conclusion_beyond}.
Last, although the three papers \cite{abraham2019statistical,bohr2021stability,monard2021consistent} develop powerful new methods, these methods are specific to the particular nonlinear inverse problem studied. In contrast, the aim of this paper is to develop widely applicable theoretical insights into operator learning. We view \cite{abraham2019statistical,bohr2021stability,monard2021consistent} as being more relevant to follow-up work in the area of nonlinear inverse operator learning.

\subsection{Contributions}\label{sec:intro_contribution}
This paper provides a unified framework for the supervised learning of compact, bounded, and unbounded linear operators. The analysis is performed in the ideal situation that the eigenvectors of the true operator are known. Thus, much like the work in \cite{knapik2011bayesian} on Bayesian posterior contraction for linear inverse problems, 
our results give a theoretical roadmap for linear operator learning. Although we do not explicitly learn solution maps of inverse problems from data, our theory provides insight into the difficulty of learning operators defined by both forward and inverse problems. Our primary contributions are now listed:
\begin{enumerate}[label={(C\arabic*)}]
	\item \label{item:contr_ip} we formulate linear operator learning as a nonparametric Bayesian inverse problem with a linear operator as the unknown quantity, generalizing \cite{knapik2011bayesian} to operators;
	\item \label{item:contr_posterior} under a known eigenbasis assumption, in the large sample limit we prove convergence of the full posterior eigenvalues to the truth by deriving in-expectation and high probability upper and lower bounds for the generalization error under distribution shift;
	\item \label{item:contr_mean} we establish analogous convergence rate guarantees for the posterior mean eigenvalues with respect to learning-theoretic notions of excess risk and generalization gap;
	\item \label{item:contr_numeric} we present numerical results for learning compact, bounded, and unbounded operators arising from canonical linear PDEs in a diagonal setting, which directly support the theory, and in a non-diagonal setting, which support conjecture that our theoretical insights remain valid beyond the confines of the theory.
\end{enumerate}

A consequence of these contributions are the theoretical principles
\ref{item:intro_ideas_result_output}~to~\ref{item:intro_ideas_result_shift} (visualized in \cref{fig:intuition_lap}). Although only proved for linear operators, these may still inform state-of-the-art nonlinear operator learning techniques used in practice~\cite{bhattacharya2020model,li2020neural,lu2019deeponet,o2020derivative}. Indeed, the influence of output space smoothness on sample complexity, reflecting \ref{item:intro_ideas_result_output}, has been observed in neural operators \cite{de2022cost,kovachki2021universal,lanthaler2022error}. \Cref{item:intro_ideas_result_input} implies that training on Gaussian random field data with the commonly chosen squared exponential covariance (leading to infinitely smooth samples) is actually statistically disadvantageous. Regarding robustness of models under distribution shift, \ref{item:intro_ideas_result_input}~and~\ref{item:intro_ideas_result_shift} suggest that it may be misleading to only report prediction errors {on test data} with the same smoothness as the training data. Further exploration of these and related issues is crucial to guide the development of operator learning as an emerging field.

\subsection{Outline}\label{sec:intro_outline}
The remainder of the paper is organized as follows. Contribution \ref{item:contr_ip} (summarized in \cref{sec:intro_ideas}) is described in \cref{sec:setup}, where we give a full functional-analytic problem setup and characterize the posterior. Our main theoretical results, \cref{item:contr_posterior,item:contr_mean}, are presented and discussed in
\cref{sec:theory}. Numerical experiments \ref{item:contr_numeric} that illustrate, support, and extend beyond the theory are provided in \cref{sec:numerics}. Concluding remarks follow in \cref{sec:conclusion}. \cref{app:proofs} is devoted to proofs of the main results, with supporting lemmas in \cref{app:lemmas}. Remaining proofs of auxiliary results are located in \cref{app:extra}.

\section{Setup}\label{sec:setup}
After overviewing some notation in \cref{sec:setup_prelim}, we detail our Bayesian inverse problems approach to \cref{eqn:intro_model} in \cref{sec:setup_bayes}. \Cref{sec:setup_statlearn} gives an optimization perspective and defines expected risk and generalization gap in the infinite-dimensional setting.

\subsection{Preliminaries}\label{sec:setup_prelim}
We now detail the conventions used in this paper.
\paragraph*{Linear spaces}
Let $ (H,\ip{\cdot}{\cdot},\norm{\cdot}) $ from \cref{prob:intro_model} be a real, separable, infinite-dimensional Hilbert space. For any self-adjoint positive-definite linear operator $ A $ on $ H $, we define $A^{-1/2}$ by
functional calculus, \( \ip{\cdot}{\cdot}_{A}\defeq \ip{A^{-1/2}\cdot}{A^{-1/2}\cdot}\), and \(\norm{\cdot}_{A}\defeq \norm{A^{-1/2}\cdot\,} \).
The set $\cL(H_1;H_2)$ is the space of bounded linear operators mapping Hilbert spaces $ H_1 $ into $ H_2 $, and when $ H_1=H_2=H $, we write $ \cL(H) $. The separable Hilbert space of Hilbert--Schmidt operators from $ H_1 $ to $ H_2 $ is denoted by $ \HS(H_1;H_2)$ with inner-product $\ip{\cdot}{\cdot}_{\HS(H_1;H_2)} $. When $ H_1=H_2=H $, we write $ (\HS(H),\ip{\cdot}{\cdot}_{\HS}, \norm{\cdot}_{\HS}) $. For any $ a\in H_2$ and $ b\in H_1 $, the map $ a\otimes_{H_1} b \in \HS(H_1;H_2)$ denotes the outer product $ (a\otimes_{H_1} b)c\defeq\ip{b}{c}_{H_1}a $ for any $ c\in H_1 $. We use the shorthand $ a\otimes b\in\HS(H) $ when $ H_1=H_2=H $. For a possibly unbounded linear operator $ T $ on $ H $, we denote its domain by the subspace $ \dom(T)\subseteq H $. The identity map on \( H \) is written as \( \id\in\cL(H) \).

\paragraph*{Probability}
We primarily consider centered Borel probability measures \( \Pi \) on \( H \) with finite second moment \( \E^{h\sim\Pi}\norm{h}^2<\infty \). Such a \( \Pi \) has a covariance operator \( \Cov[\Pi]\defeq \E^{h\sim\Pi}[h\otimes h] \) in \( \cL(H) \) that is symmetric, nonnegative, and trace-class. This leads to the Karhunen--Lo\`eve (KL) expansion \( h=\sum_{j=1}^{\infty}\theta_j\xi_j\psi_j\sim\Pi \) \cite{steinwart2019convergence}. The \( \{\xi_j\} \) are zero mean, unit variance, pairwise uncorrelated real r.v.s on a complete probability space denoted by $ (\varOmega, \cF, \P) $. The $ \{\psi_j\} $ are the eigenvectors of \( \Cov[\Pi] \), extended to form an orthonormal basis of $ H $, and $ \{\theta_j^2\} $ are its nonnegative eigenvalues. If \( \Pi \) is a Gaussian measure, then the \( \{\xi_j\} \) are i.i.d. \( \normal(0,1) \) \cite{stuart2010inverse}. When appropriate, expectations are taken in the sense of Bochner integration. We use \( \E \) with no additional scripts to denote an average over all sources of randomness. We implicitly justify the exchange of expectation and infinite summation with the Fubini--Tonelli theorem.

\paragraph*{Notation}
For real \( p \) and \( q \), we write $ p\mmin q\defeq\min\{p,q\} $ and \( p \mmax q\defeq \max\{p,q\} \).
For two nonnegative real sequences $ \{a_n\} $ and $ \{b_n\} $, we write $ a_n\simeq b_n $ if $ \{a_n/b_n\} $ is bounded away from zero and infinity and $ a_n\lesssim b_n $ if there exists \( C>0 \) such that $ a_n/b_n\leq C $ for all \( n \). We use computer science asymptotic notation. This means that we write $ a_n=O(b_n) $ as $ n\to\infty $ if $ \limsup_{n\to\infty}a_n/b_n<\infty $, $ a_n=\Omega(b_n) $ as $ n\to\infty $ if $ b_n=O(a_n) $, $ a_n=\Theta(b_n) $ as $ n\to\infty $ if both $ a_n=O(b_n) $ and $ a_n=\Omega(b_n) $, and $ a_n=o(b_n) $ as $ n\to\infty $ if $ \lim_{n\to\infty}a_n/b_n=0 $. We sometimes use \( a_n\asymp b_n \) as convenient shorthand for \( a_n=\Theta(b_n) \) and \( a_n\ll b_n \) for \( a_n=o(b_n) \).

\subsection{Bayesian inference}\label{sec:setup_bayes}
In this subsection, we continue the development of operator learning as an inverse problem. We adopt the following conventions. Define \( D_N \) to be the collection of all the data, \( D_N\defeq (X,Y) \). We equip the \( N \)-fold product space $ H^{N} $ with the inner-product $ \ip{U}{V}_{H^N}=\frac{1}{N}\sum_{n=1}^{N}\ip{u_n}{v_n} $ for any $ U=(u_1,\ldots,u_N) $ and $V=(v_1,\ldots, v_N)\in H^N $. This makes $ H^N $ a Hilbert space. For any symmetric positive-definite $ \cC\in\cL(H) $, define $ H_{\cC}\defeq\im(\cC^{1/2})\subseteq H $. Equipped with 
the inner-product $ \ip{\cdot}{\cdot}_{\cC} $, the set $ H_{\cC} $ is a Hilbert space.

\subsubsection{Weighted Hilbert--Schmidt operators}\label{sec:setup_bayes_bochner}
Thus far we have not specified the space to which the self-adjoint operator \( L\colon \dom(L)\subseteq H\to H \) in \hyperref[prob:intro_model]{Main Problem} belongs. Since \( L \) may not be bounded on \( H \), the ideal Hilbert space \( \HS(H) \) is not sufficient. Instead, we consider particular Lebesgue--Bochner spaces. Let \( \nu' \) be a centered Borel probability measure on a sufficiently large space containing \( H \) with bounded covariance \( \Lambda'\defeq\Cov[\nu']\in\cL(H) \). Then \( L_{\nu'}^2(H;H) \) is defined as the set of all Borel measurable maps \( F\colon H\to H \) such that \( \norm{F}_{L^2_{\nu'}(H;H)}\defeq (\E^{x\sim\nu'}\norm{F(x)}^2)^{1/2} \) is finite. Linearity gives additional structure. For any linear $ T \colon  \dom(T)\subseteq H\to H$, the identity $ \ip{v}{Tu}=\tr{Tu\otimes v} $ for all $ u \in \dom(T) $ and $ v\in H $ yields
\begin{equation}\label{eqn:norm_to_trace}
\E^{x\sim\nu'}\norm{Tx}^2=\tr[\big]{T\Lambda'^{1/2}\bigl(T\Lambda'^{1/2}\bigr)^{*}}=\norm{T\Lambda'^{1/2}}_{\HS}^2=\norm{T}_{\HS(H_{\Lambda'};H)}^2\,.
\end{equation}

By \cref{eqn:norm_to_trace}, linear maps with finite \( L_{\nu'}^2 \) Bochner norm can be identified with weighted Hilbert--Schmidt operators. This is useful, as the next fact (proved in \cref{app:extra}) demonstrates.

\begin{fact}[weighted Hilbert--Schmidt spaces]\label{fact:setup_bayes_measure}
	Suppose there is a symmetric positive-definite linear operator $ \cK\in\cL(H) $ that satisfies \( \cK^{-1/2}\in\HS(H_{\Lambda'};H) \), where \( \Lambda'=\Cov[\nu'] \). Then \( \nu'(H_{\cK})=1 \). Additionally, if \( T\in\HS(H_{\cK};H) \), then \( \E^{x\sim\nu'}\norm{Tx}^2<\infty \).
\end{fact}

For \( \cK \) satisfying the hypotheses of \cref{fact:setup_bayes_measure}, the fact  suggests that \( \HS(H_{\cK};H) \) is a natural Hilbert space for \( L \) to belong to. Defining $ \dom(L)\defeq \{h\in H\colon  L h\in H\}$ (the usual domain for many self-adjoint operators), \cref{fact:setup_bayes_measure} also implies that \( \nu'(\dom(L))=1 \). Identifying such a valid \( \cK \) requires some \emph{a priori} knowledge about the unknown \( L \). For example, later in \hyperref[par:item:as_smooth_truth]{subsection 3.1} we show how to choose a \( \cK \) ``smoothing enough'' so that \( L\in\HS(H_{\cK};H) \). For now, to make sense of the remainder of \cref{sec:setup} we \emph{assume} that the following condition holds.

\begin{condition}[existence of \( \cK \)]\label{cond:existence_of_K}
	There exists a symmetric positive-definite linear operator $ \cK\in\cL(H) $ such that \( \{\cK^{-1/2}\Lambda^{1/2}, \cK^{-1/2}\Lambda'^{1/2}\}\subset\HS(H) \) and \( \Ld\in\HS(H_{\cK};H) \).
\end{condition}

Our use of weighted Hilbert--Schmidt spaces is closely related to the notion of $ \Pi $ measurable linear operators for a probability measure \( \Pi \), which is a common way to work with unbounded operators; see \cite{gawarecki2010stochastic,mandelbaum1984linear} and \cite[sect. 3--4]{knapik2011bayesian}. If $ \cK $ is compact, the weighted norm is weak in the sense that $ \HS(H_{\cK};H)\supset\cL(H)\supset\HS(H) $ \cite[sect. 2.2]{gawarecki2010stochastic}. However, if \( L \) is already Hilbert--Schmidt on \( H\), then the choice \( \cK=\id\in\cL(H) \) in \cref{cond:existence_of_K} is valid (if \( \Lambda' \) is trace-class).

\subsubsection{Data model}\label{sec:setup_bayes_model}
Recall the statistical model \( Y=K_X L +\gamma\Xi \) \cref{eqn:intro_ideas_ip_operator} from \cref{sec:intro_ideas_ol_as_ip}. We now give further details about each component in this data model.

\paragraph*{Forward map}
The input data \( X\sim \nu^{\otimes N} \) in \( H^N \) defines the linear forward map \( K_X \). We enforce that the Borel probability measure \( \nu \) has finite second moment. Hence, its covariance $ \Lambda\in\cL(H) $ is symmetric, nonnegative, and trace-class on $ H $. We take \( \Lambda \) to be strictly positive-definite for simplicity. For \( \cK \) as in \cref{cond:existence_of_K} and for any $ Z\in H_{\cK}^N $ (the $N$-fold product of $H_{\cK}$), we
define the forward map $ K_Z\in\cL(\HS(H_{\cK};H);H^{N}) $ by $ T\mapsto K_Z T\defeq (Tz_1,\ldots, Tz_N)$. \Cref{fact:forward_facts}, proved in \cref{app:extra}, addresses the compactness of this map.
\begin{fact}[non-compact]\label{fact:forward_facts}
	If $ Z\in H_{\cK}^{N}\setminus\{0\} $, then $ K_Z\in\cL(\HS(H_{\cK};H);H^{N}) $ is not compact.
\end{fact}

\paragraph*{Noise}
Define \( \pi\defeq \normal(0,\id) \). Since $\id\in\cL(H)$ is not trace-class on $H$, the white noise $\xi\sim\pi$ is not a proper random element in \( H \). It is instead defined as the \( H \)-indexed centered Gaussian process \( \xi\defeq \{\xi_{h}\colon h\in H\} \) with covariance \((h,h')\mapsto \E[\xi_{h}\xi_{h'}]=\ip{h}{h'} \) \cite[sect. 2]{knapik2011bayesian}. For \( \gamma>0 \), the noise is then \( \gamma\Xi \), where \( \Xi\sim\pi^{\otimes N} \) is assumed independent of \( X \) and \( L \). Finally, we interpret \( Y \) in \cref{eqn:intro_ideas_ip_operator} as \( N \) independent stochastic processes \( Y_n\defeq\{\ip{y_n}{h}\colon h\in H\} \) for \( n\in\{1,\ldots, N\} \), such that for \( h\in H \), it holds that \( \ip{y_n}{h}\condbar X, L\sim \normal(\ip{Lx_n}{h}, \gamma^2\norm{h}^2) \). Observing \( Y \) entrywise on the indices \( \{\varphi_j\} \) leads to \cref{eqn:intro_ideas_ip_matrix}. The case of general \( \Cov[\pi]=\Gamma \in\cL(H)\) may be handled by pre-whitening the data \cref{eqn:intro_ideas_ip_operator} \cite[sect. 1]{agapiou2018posterior}; see also the related \cref{cor:theory_expectation_color}.

\paragraph*{Prior}
We assume that $L\sim\mu$ is \emph{a priori} Gaussian, where \( \mu\defeq\normal(0,\Sigma)  \) is conjugate to the likelihood, and independent of \( X \) and \( \Xi \). Since we view the r.v. $ L\colon \dom(L)\subseteq H\to H $ as a densely defined operator, the sense in which $ \mu $ is a proper Gaussian measure requires some care. Specifically, we take $ \Sigma\in\cL(\HS(H_{\cK};H)) $ to be symmetric, positive-definite, and trace-class on $ \HS(H_{\cK};H)\supseteq\HS(H) $, but not necessarily trace-class on $ \HS(H) $. Here $\cK$ ensures that the support of \( \mu \) is large enough to encompass unbounded operators on \( H \).

\paragraph*{Posterior}
Recall that the realized data \( Y \) is given by \cref{eqn:intro_ideas_ip_operator} with \( L=\Ld \) under \cref{assump:intro_ideas_truth}. Since \( \Xi \), $X$, and $L$ are \emph{a priori} independent, the 
posterior for $L$ given
$ Y $ and $ X $, denoted by $ \post $, is the same as that obtained
when $L$ is conditioned on $Y$ with $ X $ fixed, a.s.; see \cite[Thms. 32, 13, and 37]{dashti2017} for more justification. The Bayesian inverse problem \cref{eqn:intro_ideas_ip_operator} is linear and Gaussian. Thus, the posterior is also a Gaussian on \( \HS(H_{\cK};H) \) and is denoted by
\begin{equation}\label{eqn:setup_bayes_posterior_notation}
\post = \normal(\bar{L}^{(N)}, \Sigma^{(N)})\,.
\end{equation}
The posterior mean is $ \bar{L}^{(N)}= \E^{L\sim\post}L\in \HS(H_{\cK};H)$. The posterior covariance operator is $ \Sigma^{(N)}\in\cL(\HS(H_{\cK};H)) $. Explicit formulas for both are known even in this infinite-dimensional setting \cite{knapik2011bayesian,lehtinen1989linear,mandelbaum1984linear}. We link \cref{eqn:setup_bayes_posterior_notation} to our diagonal formulation in the next three subsections. 

\subsubsection{Diagonalization}\label{sec:setup_bayes_diag}
Recall the scalar sequence space model \( y_{jn}=\ip{\varphi_j}{x_n}l_j + \gamma\xi_{jn} \) for \( j\in\N \) and \( n\in\setoneton \) \cref{eqn:intro_ideas_ip_diagonal}.\footnote{In the absence of noise \( \{\gamma \xi_{jn}\} \), determination
of $\{l_j=\ld_j\}$ is trivial: the diagonalizable structure arising from \cref{assump:intro_ideas_diagonal} means that $\{\ld_j\}$ 
may be recovered from a \emph{single} input-output pair, say $ (x_1, \Ld x_1) $.
However, our non-diagonal simulation studies in \cref{sec:numerics_beyond} will 
demonstrate the relevance of our 
theory beyond \cref{assump:intro_ideas_diagonal}. In this setting, determination
of $\{\ld_j\}$ is no longer trivial in the noise-free case.}
This model arises from the matrix sequence problem \cref{eqn:intro_ideas_ip_matrix} under \cref{assump:intro_ideas_diagonal} by noting that \( \Lmat_{jk}=\ip{\varphi_j}{L\phi_k}=l_j\ip{\varphi_j}{\phi_k} \) because \( L \) is self-adjoint. The \( \{\phi_k\} \) are the orthonormal eigenvectors of \( \Lambda=\Cov[\nu] \). For each \( n \), the \( \{x_{kn}=\ip{\phi_k}{x_n}\}_{k\in\N} \) are pairwise uncorrelated r.v.s by KL expansion. If \( L \) and \( \Lambda \) commute, then \( \{\phi_k=\varphi_k\} \) can be taken as the eigenbasis for \( L \). For each \( n \), the scalar model's coefficients \( \{\ip{\varphi_j}{x_n}\} \) are pairwise uncorrelated in this case. However, in general \( L \) and \( \Lambda \) do not commute, so the coefficients are correlated. For \( n\in\setoneton \), it is useful to write these as
\begin{equation}\label{eqn:setup_bayes_gjn}
g_{jn}\defeq \ip{\varphi_j}{x_n}=\sum\nolimits_{k=1}^{\infty}\ip{\varphi_j}{\phi_k}x_{kn}\qa \vartheta_j^2\defeq \Var[g_{j1}]=\ip{\varphi_j}{\Lambda\varphi_j} \qf j\in\N\, .
\end{equation}
Our proofs use some independence-agnostic methods to deal with the dependent, correlated family $ \{g_{jn}\}_{j\in\N} $. Nonetheless, $ \{g_{jn}\}_{n=1}^{N} $ is still i.i.d. for fixed $ j $ and $\E[g_{jn}g_{jn'}]=0$ for $ n\neq n' $.

\subsubsection{Posterior characterization}\label{sec:setup_bayes_posterior}
For two sequences $ \{a_{jn}\}$ and $ \{b_{jn}\} $, we henceforth use the averaging notation $ \avgn{a_j}{b_j}\defeq\frac{1}{N}\sum_{n=1}^{N}a_{jn}b_{jn} $. For \cref{eqn:intro_ideas_ip_diagonal}, we assume a prior $\{l_j\} \sim \priorseq\defeq \bigotimes_{j=1}^{\infty}\normal(0,\sigma_j^2)$.
We will identify $ L\sim\mu $ with $ l\defeq\{l_j\}\sim\priorseq$ in \cref{sec:setup_bayes_testerror}. Under this product prior, \cref{eqn:intro_ideas_ip_diagonal} decouples (i.e., \( \{l_j\}\condbar D_N = \{l_j\condbar D_N\} \)) into an infinite number of random scalar Bayesian inverse problems that are equivalent to the full infinite-dimensional problem \cref{eqn:intro_ideas_ip_operator}. By completing the square \cite[Ex. 6.23]{stuart2010inverse}, we 
obtain the following Gaussian posterior.
\begin{fact}[posterior]\label{fact:posterior_sequence}
	The law of \( \{l_j\}\condbar D_N \) is $\postseq=\bigotimes_{j=1}^{\infty}\normal(\bar{l}_j^{(N)},(\sigma_j^{(N)})^2)$, where
	\begin{equation}\label{eqn:posterior_sequence}
	\bar{l}_j^{(N)}=\dfrac{N\gamma^{-2}\sigma_j^{2}\avgn{y_j}{g_j}}{1+N\gamma^{-2}\sigma_j^2\avgn{g_j}{g_j}}
	\qa
	\bigl(\sigma_j^{(N)}\bigr)^2=\dfrac{\sigma_j^2}{1+N\gamma^{-2}\sigma_j^2\avgn{g_j}{g_j}} \qf j\in\N\,.
	\end{equation}
\end{fact}

\subsubsection{Bayesian test error}\label{sec:setup_bayes_testerror}
The true $ \Ld $ is naturally approximated by the \emph{posterior mean estimator} $ \bar{l}^{(N)}\defeq\{\bar{l}^{(N)}_j\} $ and the \emph{posterior sample estimator} $ {l}^{(N)}\defeq\{{l}^{(N)}_j\}\sim\postseq $. Defining the linear bijection $ B\colon  \{l_j\}\mapsto \sum_{j=1}^{\infty}l_j\varphi_j\otimes\varphi_j $, it follows that the actual posterior $ \post $ \cref{eqn:setup_bayes_posterior_notation} on $ L $ is the pushforward of $ \postseq $ under $ B $, that is, $ L^{(N)}\sim\post=B_{\sharp}\postseq=\normal(\bar{L}^{(N)},\Sigma^{(N)}) $.

Recall the measure \( \nu' \) from \cref{sec:setup_bayes_bochner} that has bounded covariance \( \Lambda'\in\cL(H) \) (e.g., \( \Lambda'=\id \) is allowed). Assume \( \Lambda' \) has an orthonormal eigenbasis \( \{\phi_k'\} \) of \( H \). We now view \( \nu' \) as an arbitrary \emph{test data distribution} that we are interested in predictions on. A useful representation of the weighted norm \cref{eqn:norm_to_trace} is \(T\mapsto \E^{x\sim\nu'}\norm{Tx}^2=\sum_{j,k}\lambda_k(\Lambda')\ip{\varphi_j}{T\phi_k'}^2 \), where \( \{\lambda_k(\Lambda')\} \) denotes the eigenvalues of \( \Lambda' \). In our setting, \( L \) is diagonal in \( \{\varphi_j\} \) which leads to
\begin{equation}\label{eqn:bochner_coord}
\norm{L}^2_{L^2_{\nu'}(H;H)}=\sum\nolimits_{i=1}^{\infty}\vartheta_i'^2l_i^2\,,
\qw
\vartheta_j'^2\defeq \sum\nolimits_{k=1}^{\infty}\lambda_k(\Lambda')\ip{\varphi_j}{\phi_k'}^2=\ip{\varphi_j}{\Lambda'\varphi_j}
\end{equation}
for \( j\in\N \). We can now define a notion of test error (i.e., prediction or ``generalization'' error).
\begin{definition}[test error: posterior]\label{def:test_error}
	The \emph{test error of the posterior sample estimator} is
	\begin{equation}\label{eqn:error_bochner_intro}
	\E^{D_N}\E^{L^{(N)}\sim\post}\norm[\big]{\Ld - L^{(N)}}_{L^2_{\nu'}(H;H)}^2 = \E^{D_N}\E^{l^{(N)}\sim\postseq}\sum\nolimits_{j=1}^{\infty}\vartheta_j'^2 \abs[\big]{\ld_j - l_j^{(N)}}^2 \, .
	\end{equation}
\end{definition}
The outer expectation is with respect to the
data, and the inner expectation is with respect to the Bayesian
posterior. The definition of test error for the posterior mean is similar.
\begin{definition}[test error: mean]\label{def:test_error2}
	The \emph{test error of the posterior mean estimator} is
	\begin{equation}\label{eqn:error_bochner_intro2}
	\E^{D_N}\norm[\big]{\Ld - \bar{L}^{(N)}}_{L^2_{\nu'}(H;H)}^2 = \E^{D_N}\sum\nolimits_{j=1}^{\infty}\vartheta_j'^2 \abs[\big]{\ld_j - \bar{l}_j^{(N)}}^2\, .
	\end{equation}
\end{definition}
We say that \cref{eqn:error_bochner_intro} 
or \cref{eqn:error_bochner_intro2} tests \emph{in-distribution} 
if $ \nu'=\nu $ and \emph{out-of-distribution} or \emph{under distribution shift} otherwise. If \( \Lambda'=\id \), then the \( L^2_{\nu'} \) Bochner norm equals the familiar un-weighted \( \HS(H) \) norm.
In \cref{sec:theory}, we study the \( N\to\infty \) asymptotics of \cref{eqn:error_bochner_intro,eqn:error_bochner_intro2}.

\subsection{Statistical learning}\label{sec:setup_statlearn}
We briefly adopt a statistical learning theory perspective to complement the Bayesian approach of \cref{sec:setup_bayes}. Let $ \cP $ denote the joint distribution on \( (x,y) \) implied by \( y=\Ld x+\gamma\xi \), where \( x\sim\nu \) and \( \xi\sim\pi =\normal(0,\id)\) independently. The data in \cref{eqn:intro_model} is then $ (x_n, y_n)\sim \cP$ i.i.d., $ n\in\setoneton $. Since regression is our focus, it is natural to work with the square loss function on $ H $. Then $ \E^{(x,y)\sim\cP}\frac{1}{2}\norm{y-Lx}^2 $ and $\frac{1}{N}\sum_{n=1}^{N}\frac{1}{2}\norm{y_n-Lx_n}^2 $ define the expected risk and empirical risk for \( L \), respectively. However, these expressions are not well-defined because infinite-dimensional \( H \) implies $ \norm{y}=\norm{\xi}=\infty $ a.s. \cite[Rem. 3.8]{stuart2010inverse}. Inspired by the negative log likelihood of $ \post $ as in \cite{alberti2021learning,nickl2020convergence}, we re-define the risks as follows.
\begin{definition}[expected risk]
	Given \( L \), the \emph{expected risk} \emph{(}or \emph{population risk}\emph{)} is 
	\begin{equation}\label{eqn:risk_expect}
	\cR_{\infty}(L)\defeq\E^{(x,y)\sim\cP}\bigl[\tfrac{1}{2}\norm{Lx}^2-\ip{y}{Lx}\bigr]\,.
	\end{equation}
\end{definition}
\begin{definition}[empirical risk]
	Given \( L \), the \emph{empirical risk} is
	\begin{equation}\label{eqn:risk_emp}
	\cR_{N}(L)\defeq\frac{1}{N}\sum\nolimits_{n=1}^{N}\Bigl[\tfrac{1}{2}\norm{Lx_n}^2-\ip{y_n}{Lx_n}\Bigr]= \tfrac{1}{2}\norm{K_XL}_{H^N}^2-\ip{Y}{K_XL}_{H^N}\,,
	\end{equation}
	and the \emph{regularized empirical risk} is
	\begin{equation}\label{eqn:risk_emp-reg}
	\cR_{N,W}(L)\defeq \cR_{N}(L) + \tfrac{1}{2N}\norm{W^{-1/2}L}_{\HS(H_{\cK};H)}^2 \, ,
	\end{equation}
	where $ W\in\cL(\HS(H_{\cK};H)) $ is symmetric positive-definite and $ \cK $ is as in \cref{cond:existence_of_K}.
\end{definition}
\Cref{eqn:risk_expect,eqn:risk_emp} \emph{are well-defined} because the ``infinite constants'' $ \frac{1}{2}\norm{y}^2 $ and \( \frac{1}{2}\norm{y_n}^2 \) from the original risk expressions are subtracted away and the linear cross terms \( \ip{y}{Lx} \) and \( \ip{y_n}{Lx_n} \), viewed as actions under stochastic processes (see \cref{sec:setup_bayes_model}), are finite a.s.\,.

The role of risk is to quantify the accuracy of a hypothesis \( L \). By the independence of \( x \) and \( \xi \) plus the stochastic process definition of \( \pi \) in \cref{sec:setup_bayes_model}, \( \E^{(x,y)\sim\cP}\ip{y}{Lx}=\E^{x\sim\nu}\ip{\Ld x}{Lx} \) so that \(\cR_{\infty}(L)=\tfrac{1}{2}\E^{x\sim\nu}\norm{\Ld x - Lx}^2-\tfrac{1}{2}\E^{x\sim\nu}\norm{\Ld x}^2\). Thus, the infimum of \( \cR_{\infty} \) is achieved at the \emph{regression function} \cite{caponnetto2007optimal} $\E[y\condbar x=\cdot]=\Ld\in \HS(H_{\cK};H)$. Minimizers of the empirical risk over the RKHS \emph{hypothesis class} \( \sL=\im(W^{1/2}) \) are point estimates of the true \( \Ld \) (but we do not require \( \Ld\in\sL \)). Our focus is the minimizer $ \hat{L}^{(N,W)} $ of the convex functional \cref{eqn:risk_emp-reg} over \( \sL \). It may be identified as the posterior mean $ \bar{L}^{(N)} $ from \cref{eqn:setup_bayes_posterior_notation} whenever $ W $ equals the prior covariance \( \Sigma \) \cite{dashti2013map}. We enforce this and write $ \hat{L}^{(N, W)}\equiv\bar{L}^{(N)} $. To quantify the performance of $ \bar{L}^{(N)} $, we employ the following notions of error from statistical learning.
\begin{definition}[excess risk]\label{def:excess_risk}
	The \emph{excess risk} of the posterior mean is defined by
	\begin{equation}\label{eqn:excess_risk}
	\cE_N\defeq 2\cR_{\infty}(\bar{L}^{(N)}) - 2\cR_{\infty}(\Ld)=\E^{x\sim\nu}\norm{\Ld x-\bar{L}^{(N)}x}^2\,.
	\end{equation}
\end{definition}

The excess risk is always nonnegative and provides a notion of consistency for $ \bar{L}^{(N)} $. In \cref{sec:theory_excess_risk}, we control \cref{eqn:excess_risk} either in expectation, \( \E^{D_N}\cE_N \), or with high probability over the input training samples, \( \E^{Y\condbar X}\cE_N \). The last expectation is over the noise only, under \cref{eqn:intro_ideas_ip_operator}.

Next, we define the generalization gap. It can take any sign and, as the difference between test and training errors, controls the amount of ``overfitting'' that $\bar{L}^{(N)}$ can exhibit.

\begin{definition}[generalization gap]\label{def:gen_gap}
	The \emph{generalization gap} of the posterior mean is
	\begin{equation}\label{eqn:gen_gap}
	\cG_N\defeq\cR_{\infty}(\bar{L}^{(N)})- \cR_{N}(\bar{L}^{(N)})\,.
	\end{equation}
\end{definition}

\Cref{eqn:gen_gap} may be written in terms of $ \Ld $ instead of $ y $ (see \cref{eqn:j1j2j3} in the proof of \cref{thm:gen_gap_expect}). In \cref{sec:theory_gen_gap}, we bound the \emph{expected generalization gap} \( \EGG \).

\section{Convergence rates}\label{sec:theory}
We are now ready to study the sample complexity of the posterior estimator \cref{eqn:posterior_sequence} with respect to the notions of error defined in \cref{sec:setup_bayes_testerror,sec:setup_statlearn}. In \cref{sec:theory_assumptions}, we list and interpret our main assumptions. In \cref{sec:theory_expectation}, under fourth moment conditions we establish asymptotic convergence rates of both the posterior sample and mean estimators and related lower bounds. Posterior contraction is discussed in \cref{sec:theory_contraction}. 
Analogous high probability results are developed in \cref{sec:theory_probability} for subgaussian design. Last, both upper and lower bounds are established in expectation for the excess risk and generalization gap in \cref{sec:theory_excess_risk,sec:theory_gen_gap}. We collect all of the proofs in \cref{app:proofs}.

\subsection{Main assumptions}\label{sec:theory_assumptions}
In the setting of the sequence model \cref{eqn:intro_ideas_ip_diagonal}, our convergence theory for diagonal linear operator learning is primarily developed under five assumptions.
\begin{assumption}[eigenvalue learning assumptions]\label{assump:theory_assumptions_main}
	The following conditions hold true.
	\begin{enumerate}[label={(A\arabic*)}]
		\item \label{item:as_true_diag} \emph{(\sfit{diagonal true operator})} \Cref{assump:intro_ideas_diagonal,assump:intro_ideas_truth} hold, so that \( \Ld=\sum_{j=1}^{\infty}\ld_j\varphi_j\otimes\varphi_j \).
		
		\item \label{item:as_smooth_truth} \emph{(\sfit{smoothness of true operator})} The true eigenvalues satisfy \( \ld\defeq \{\ld_j\}\in\cH^{s} \) for some \( s\in\R \).
		
		\item \label{item:as_smooth_prior} \emph{(\sfit{smoothness of prior})} The prior variance sequence \( \{\sigma_j^2\} \) in \( \priorseq= \bigotimes_{j=1}^{\infty}\normal(0,\sigma_j^2) \) satisfies 
		\begin{equation}\label{eqn:theory_assumptions_prior}
		\sigma_j^2=\Theta(j^{-2p}) \qas j\to\infty  \quad\text{for some}\quad  p\in\R\,.
		\end{equation}
		
		\item \label{item:as_smooth_data} \emph{(\sfit{smoothness of data})} The trace-class covariance operator \( \Lambda\in\cL(H) \) of the input training data distribution \( \nu \) satisfies 
		\begin{equation}\label{eqn:theory_assumptions_train}
		\vartheta_j^2=\ip{\varphi_j}{\Lambda\varphi_j}=\Theta(j^{-2\al}) \qas j\to\infty  \quad\text{for some}\quad \al>1/2\,.
		\end{equation}
		The input test data distribution \( \nu' \) is a centered Borel probability measure with a bounded covariance operator \( \Lambda'\in\cL(H) \) that satisfies 
		\begin{equation}\label{eqn:theory_assumptions_test}
		\vartheta_j'^2=\ip{\varphi_j}{\Lambda'\varphi_j}=\Theta(j^{-2\al'}) \qas j\to\infty  \quad\text{for some}\quad \al'\geq 0\,.
		\end{equation}
		
		\item \label{item:as_exponent} \emph{(\sfit{smoothness range})} It holds that \( (\al\mmin \al') + s >0 \) and \( (\al\mmin \al') + (p-1/2) >0 \).
	\end{enumerate}
\end{assumption}

These assumptions are interpreted as follows.
\subparagraph*{\Cref{item:as_true_diag}}
The diagonalization allows us to identify \( \Ld \) with its eigenvalues \( \ld \). The domain $\dom(\Ld)\defeq \{h\in H \colon \norm{\Ld h}^2=\sum_{j=1}^{\infty} \abs{\ld_j}^2\ip{\varphi_j}{h}^2<\infty \}$ ensures that \( \Ld \) is self-adjoint on \( H \).

\subparagraph*{\Cref{item:as_smooth_truth}}\label{par:item:as_smooth_truth}
The regularity condition \( \ld\in\cH^s \) implicitly determines the sense in which the series expansion for \( \Ld \) in \ref{item:as_true_diag} converges. If \( s\geq 0 \), then \( \Ld\in\HS(H) \). Otherwise, there exists \( \cK\in\cL(H) \) such that \( \Ld\in\HS(H_{\cK};H) \). For example, define \( \cK_{s'}\defeq\sum_j\kappa_j^2\varphi_j\otimes\varphi_j \) with \( \kappa_j^2=j^{2s'} \). Then \( \norm{\Ld}_{\HS(H_{\cK_{s'}};H)}=\norm{\ld}_{\cH^{s'}} \), so \( \Ld\) converges in \(\HS(H_{\cK_{s'}};H) \) for any \( s'\leq s <0 \).

\subparagraph*{\Cref{item:as_smooth_prior}}
The exponent \( p\in\R \) in \cref{eqn:theory_assumptions_prior} adjusts the regularity of prior draws \( l\sim\priorseq \): \( l\in\cH^{s'} \) a.s. for every \( s'<p-1/2 \). The choice \( p=s+1/2 \) thus gives the closest match to the
true regularity of \( \ld\in\cH^s \). Relating back to \cref{sec:setup_bayes_model}, the full prior is \( \mu=B_{\sharp}\priorseq=\normal(0,\Sigma) \). With, e.g., \( \cK=\cK_{s'} \) as above, \( \Sigma \) then satisfies $ \Sigma\varphi_i\otimes\varphi_j={\kappa}_j^2\sigma_j^2\delta_{ij}\varphi_i\otimes\varphi_j $ for all \( i,j \).

\subparagraph*{\Cref{item:as_smooth_data}}
\Cref{eqn:theory_assumptions_train,eqn:theory_assumptions_test} reflect algebraic spectral decay of the input data covariance operators with respect to the eigenbasis \( \{\varphi_j\} \) of \( \Ld \). This provides a weak link between the data distributions and the prior; see \cref{item:intro_ideas_diag_link}. Although sharp bounds such as \cref{eqn:theory_assumptions_train,eqn:theory_assumptions_test} may be difficult to verify when \( \Lambda \) or \( \Lambda' \) is not diagonalized in \( \{\varphi_j\} \), \cref{fig:assump_ex_data_decay} provides strong numerical evidence that exact power law decay can still exist in this setting.

\subparagraph*{\Cref{item:as_exponent}}
The first inequality in \ref{item:as_exponent} ensures that \( \Ld \) has finite \( L^2_{\nu} \) and \( L^2_{\nu'} \) Bochner norms \cref{eqn:bochner_coord}. In particular, $ \nu(\dom(\Ld))=\nu'(\dom(\Ld))=1 $.\footnote{Notice that  we \emph{do not} invoke the \( \cK \)-weighted Hilbert--Schmidt formulation from \cref{sec:setup_bayes_bochner,sec:setup_bayes_model} in \cref{assump:theory_assumptions_main}. Such abstraction is unnecessary for our straightforward diagonal approach (\cref{item:as_true_diag}). In particular, the scalar sequence space model \cref{eqn:intro_ideas_ip_diagonal} is well-defined without reference to any \( \cK \). However, work going beyond diagonal operators may need to use \( \HS(H_{\cK};H) \) spaces, with \( \cK \) satisfying \cref{cond:existence_of_K}.} The second inequality ensures that the prior covariance \( \Sigma \) is trace-class on both $ \HS(H_{\Lambda};H) $ and $ \HS(H_{\Lambda'};H) $. This means \( L\sim\mu=\normal(0,\Sigma) \) has finite \( L^2_{\nu} \) and \( L^2_{\nu'} \) Bochner norms a.s.\,. It follows that the latter two assertions also hold for the posterior \( \post=\normal(\bar{L}^{(N)},\Sigma^{(N)}) \), a.s. with respect to \( D_N \).

\begin{figure}[htbp]%
	\centering
	\subfloat[Decay of \( \ip{\varphi_j}{\Lambda(\tilde{\al})\varphi_j} \) as \( j\to\infty \)]{
		\includegraphics[width=0.3425\textwidth]{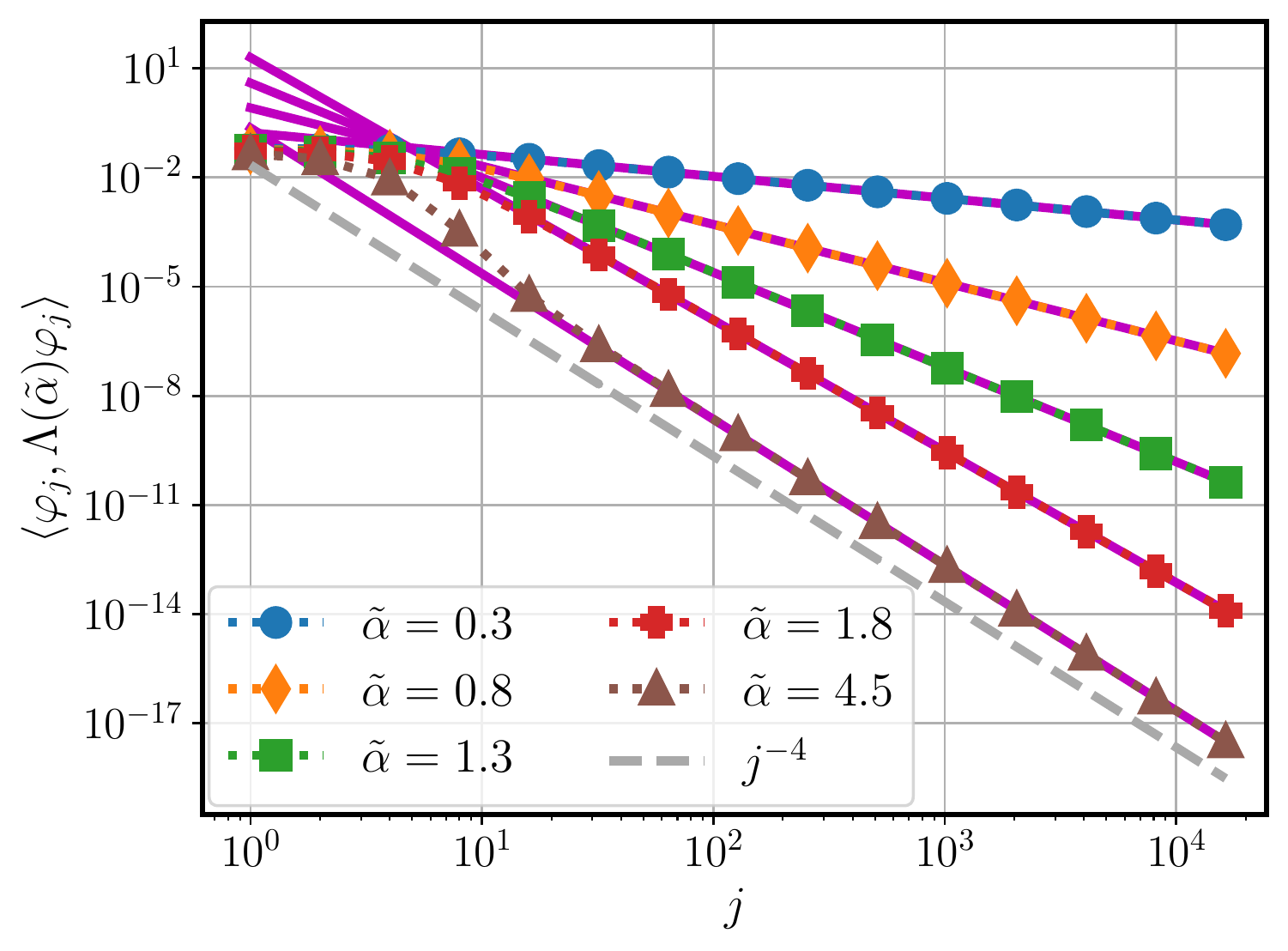}
		\label{fig:cov_decay_volterra_j}}\hspace{0.10\textwidth}%
	\subfloat[Fit to power law as \( \tilde{\al} \) varies]{
		\includegraphics[width=0.32\textwidth]{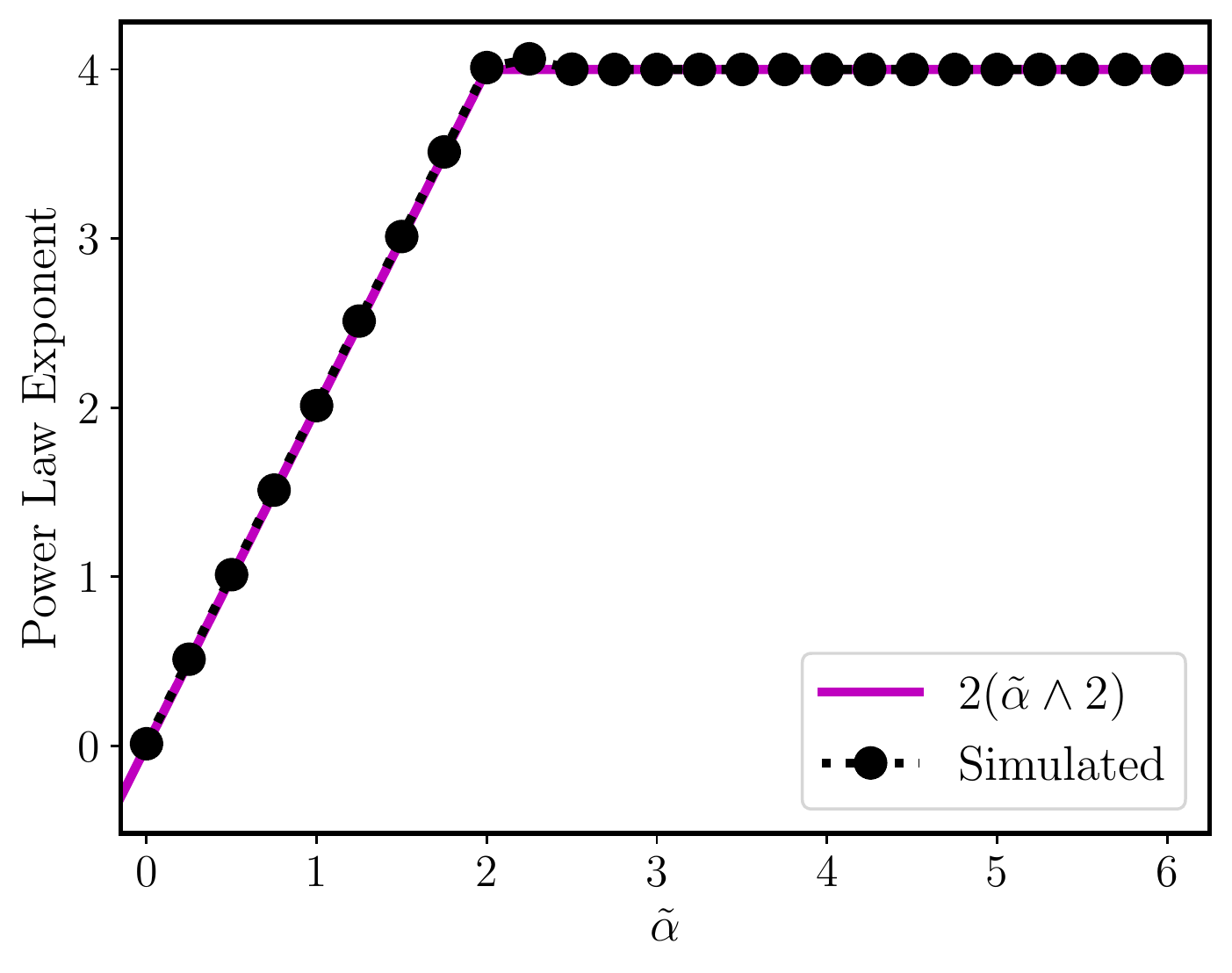}
		\label{fig:cov_decay_volterra_al}}
	\caption{Example of exact power law spectral decay \cref{eqn:theory_assumptions_train} when \( \Lambda \) is not diagonalized by \( \{\varphi_j\} \). Here we choose \( \Lambda=\Lambda(\tilde{\al}) \) such that its eigenpairs \( \{(\lambda_k^2, \phi_k)\} \) satisfy \( \lambda_k^2=15^{2\tilde{\al} - 1}((k\pi)^2+225)^{-\tilde{\al}}=\Theta(k^{-2\tilde{\al}}) \) with \( \tilde{\al}\in\R \) and \( z\mapsto \phi_k(z)=\sqrt{2}\sin(k\pi z) \). We choose output basis \( z\mapsto \varphi_j(z)=\sqrt{2}\cos((j-\frac{1}{2})\pi z) \). Both \( \{\phi_k\} \) and \( \{\varphi_j\} \) are orthonormal bases of \( H=L^2((0,1);\R) \). One can show that \( \vartheta_j^2=\ip{\varphi_j}{\Lambda(\tilde{\al})\varphi_j}=\sum_{k=1}^{\infty}64\pi^{-2}\lambda_k^2 k^2(4(j(j-1)-k^2)+1)^{-2} \). For select \( j\leq 2^{21} \), we sum the first \( 2^{21} \) terms of this series to approximately compute the \( \{\vartheta_j^2\} \). \Cref{fig:cov_decay_volterra_j} shows that \( \vartheta_j^2 \) decays asymptotically as a power law (with magenta lines being linear least square fits) for various \( \tilde{\al} \) (saturating near \( 2\tilde{\al}=4 \)). \Cref{fig:cov_decay_volterra_al} suggests that \( \Lambda \) satisfies assumption~\ref{item:as_smooth_data} with \( \al=\tilde{\al}\mmin 2 \).
	}
	\label{fig:assump_ex_data_decay}
\end{figure}

\subsection{Expectation bounds}\label{sec:theory_expectation}
To develop error bounds in expectation, we only require mild polynomial moment conditions on the input training data measure \( \nu \).
\begin{assumption}[expectation: training data]\label{assump:theory_expectation_train}
	The training data distribution \( \nu \) is a centered Borel probability measure on \( H \) with KL expansion \( x=\sum_{k=1}^\infty \lambda_k \zeta_k\phi_k \sim\nu\). The eigenvalues \( \{\lambda_k^2\} \) of \( \Cov[\nu]=\Lambda \) are ordered to be nonincreasing, and the zero mean and unit variance r.v.s \( \{\zeta_k\} \) are independent, have finite fourth moments, and satisfy \( \E[\zeta_j^4]=O(1)\) as \( j\to\infty \). In particular, \( \E^{x\sim\nu}\norm{x}^4<\infty \). Last, the r.v.s \( \{\avgn{g_j}{g_j}\}_{j,N\in\N} \), defined in \cref{sec:setup_bayes_posterior} as \( \avgn{g_j}{g_j}=\frac{1}{N}\sum_{n=1}^{N}\ip{\varphi_j}{x_n}^2 \), satisfy \( 	\limsup_{N\to\infty} \E[(\avgn{g_j}{g_j})^{-4}]\lesssim\ip{\varphi_j}{\Lambda\varphi_j}^{-4} \) for all \( j\in\N \).
\end{assumption}

Henceforth, it is useful to define the parametrized sequences \( \{J_N\}_{N\in\N} \) and \( \{\rho_N\}_{N\in\N} \)  by
\begin{equation}\label{eqn:theory_expectation_rho_J}
J_N(\al,p)\defeq \floor[\Big]{N^{\frac{1}{2(\al + p)}}}
\qa
\rho_N(\al,\al',p)\defeq
\begin{cases}
N^{-\bigl(1-\frac{\al+1/2 -\al'}{\al+p}\bigr)}\, , &\textit{if }\,\al'<\al+1/2\,,\\
N^{-1}\log N\, , &\textit{if }\,\al'=\al+1/2\,,\\
N^{-1}\, , &\textit{if }\,\al'>\al+1/2\,,
\end{cases}
\end{equation}
respectively, for \( N\in\N \). Notice that \( J_N\to\infty \) (if \( \al+p>0 \)) and \( \rho_N=\Omega(N^{-1}) \) as \( N\to\infty \). Our main result gives asymptotic convergence rates of the test errors from \cref{sec:setup_bayes_testerror}.

\begin{theorem}[expectation: upper bound]\label{thm:upper-expect}
	Let the ground truth \( \Ld \), prior \( \mu \) on \( L \), training data distribution \( \nu \), and test data distribution \( \nu' \) satisfy \cref{assump:theory_assumptions_main,assump:theory_expectation_train}.  Let \( \rho_N=\rho_N(\al,\al',p) \) in \cref{eqn:theory_expectation_rho_J} with \( \al\), \(\al' \), and \( p \) as in \cref{assump:theory_assumptions_main}. Denote by \( \post \) the posterior distribution \cref{eqn:setup_bayes_posterior_notation} for \( L \) arising from the observed data \( D_N= (X,Y) \) in \cref{eqn:intro_ideas_ip_operator}. Then
	\begin{equation}\label{eqn:upper-expect1}
	\E^{D_N}\E^{L^{(N)}\sim\post}\norm[\big]{\Ld - L^{(N)}}_{L^2_{\nu'}(H;H)}^2
	=
	O\bigl(\rho_N\bigr)+
	o\bigl(N^{-\left(\frac{\al'+s}{\al+p}\right)}\bigr)\qas N\to\infty\,,
	\end{equation}
	where the constants in this upper bound depend on $\Ld$ or, equivalently, on $\ld$. Furthermore,
	\begin{equation}\label{eqn:upper-expect1-add}
     \sup_{\norm{\ld}_{\cH^{s}}\lesssim 1}
	\E^{D_N}\E^{L^{(N)}\sim\post}\norm[\big]{\Ld - L^{(N)}}_{L^2_{\nu'}(H;H)}^2
     =
     O\bigl(\rho_N + N^{-\left(\frac{\al'+s}{\al+p}\right)}\bigr)\qas N\to\infty\,.
	\end{equation}
	Both assertions also hold for the test error \cref{eqn:error_bochner_intro2} of the posterior mean \( \bar{L}^{(N)} \).
\end{theorem}

\cref{thm:upper-expect} has the same implications as \cref{thm:intro_ideas_thm}, namely, principles \ref{item:intro_ideas_result_output}~to~\ref{item:intro_ideas_result_shift}. The effect of distribution shift \ref{item:intro_ideas_result_shift} in \cref{eqn:upper-expect1} is apparent: increasing \( \al' \) always improves the sample complexity (until the rate \( N^{-1} \) is achieved). We note that the three smoothness cases in the \( \rho_N \) term from \cref{eqn:upper-expect1} are similar to those in functional linear regression \cite{cai2006prediction}. 
The ``matching'' prior smoothness choice $ p= s+1/2 $ leads to asymptotically balanced contributions from both error terms in \cref{eqn:upper-expect1-add}. The rate is then $ N^{-(2\al'+2s)/(1+2\al+2s)} $ if $ \al'<\al+1/2 $ (which for $ \al'=\al $ is minimax optimal \cite{cavalier2008nonparametric,knapik2018general,knapik2011bayesian}) or $ N^{-1} $ (up to logarithms) if $ \al'\geq\al+1/2 $. Principles \ref{item:intro_ideas_result_output} and \ref{item:intro_ideas_result_input} are evident: as $ s $ \emph{decreases} ($ \Ld $ becomes ``less compact'' and possibly unbounded) and \( \al \) \emph{increases} (the \( \{x_n\} \) become smoother), the rates \emph{degrade}. \Cref{fig:intuition_lap} visualizes these rates in various settings. Last, we note that the rate of convergence in \cref{eqn:upper-expect1} can be strictly faster when \( \Ld \) is fixed as opposed to when \( \Ld \) is varying for the worst case error \cref{eqn:upper-expect1-add}. Our interest is mainly in individual bounds (i.e., fixed \( \Ld \)) because these are more useful in practice.

Next, we provide a lower bound corresponding to a given $\Ld$, equivalently, $\ld$.
\begin{theorem}[expectation: lower bound]\label{thm:lower-expect}
	Let the hypotheses of \cref{thm:upper-expect} be satisfied. Let \( J_N=J_N(\al,p) \) in \cref{eqn:theory_expectation_rho_J} with \( \al\) and \( p \) as in \cref{assump:theory_assumptions_main}. Then for any positive sequence \( \{\tau_n\} \) such that $ \tau_{n}\to 0 $ and \( n\tau_n\to\infty \) as \( n\to\infty \), the posterior mean test error satisfies
	\begin{equation}\label{eqn:lower-expect}
	\E^{D_N}\norm[\big]{\Ld - \bar{L}^{(N)}}_{L^2_{\nu'}(H;H)}^2
	=
	\Omega\Bigl(
	\tau_N\rho_N
	+
	\sum\nolimits_{j>J_N}j^{-2\al'}\abs{\ld_{j}}^2
	\Bigr)
	\qas N\to\infty\,.
	\end{equation}
	The same assertion holds for the test error \cref{eqn:error_bochner_intro} of the full posterior \( \post \), but without $ \{\tau_n\} $.
\end{theorem}

The tail series term in \cref{eqn:lower-expect} is closely related to the lower bound in \cite[Thm. 3.4]{lanthaler2022error} for nonlinear operator learning because both involve the spectral decay of the covariance operator of the pushforward measure \( \Ld_{\sharp}\nu' \). Since this tail term is order $ N^{-(\al'+s)/(\al+p)}o(1) $ by \cref{eqn:knapik_1p12} in \cref{lem:knapik_1}, the lower bound \cref{eqn:lower-expect} ``matches'' the corresponding terms in the individual upper bound \cref{eqn:upper-expect1} up to \( o(1) \) factors. But without further conditions on \( \ld \) (and hence knowledge about the \( o(1) \) factors), the bounds are not guaranteed to be sharp in the over-smoothing prior regime $ p>s+1/2 $. The rates do match (up to \( \tau_N \) in \cref{eqn:lower-expect} for \( \bar{L}^{(N)} \), but \( \tau_N \) is under control) for under-smoothing priors with $ p\leq s+1/2 $ because the \( \rho_N \) terms in both \cref{thm:upper-expect,thm:lower-expect} dominate. The \( \{\tau_n\} \) factor is likely an artifact of our proof technique. By choosing \( \ld \) such that \( \abs{\ld_j}=J_N^{-s}\delta_{j-1,J_N} \) in \cref{eqn:lower-expect}, \cref{thm:lower-expect} also implies that \cref{eqn:upper-expect1-add} is truly sharp.

So far, we have assumed that \( \ld\in\cH^{s} \) for some \( s \). However, this does not preclude the possibility that \( \ld\in\cH^{s'} \) for another $s'>s$. For example, the previous theorems account for operators with analytic spectral smoothness (see \cref{sec:intro_examples}): \( \ld_j\asymp \exp(-c_1j^{c_2}) \) for \( c_1\) and  \(c_2>0 \) (here \( \ld\in \cH^{s} \) for every \( s\in\R \)). Nevertheless, many scientific problems are naturally distinguished by \emph{regularly varying} eigenvalues. These behave like a power law up to a slowly varying function \( S\colon\R_{\geq 0}\to\R_{\geq 0} \) \cite{bingham1989regular} (this means that \( S(\lambda x)/S(x)\to 1 \) as \( x\to\infty \) for every \( \lambda>0 \); examples include logarithms or functions with positive limit). The following sharp convergence result concerns posterior sample estimates of regularly varying true eigenvalues. Similar may be proved for the posterior mean, but the upper and lower bounds must be considered separately as in \cref{thm:upper-expect,thm:lower-expect}. The implications are the same.

\begin{theorem}[asymptotically sharp bound for regularly varying eigenvalues]\label{thm:theory_expectation_rvsharp}
	Let the hypotheses of \cref{thm:upper-expect} be satisfied, but instead of \ref{item:as_smooth_truth}, let \( \Ld \) be such that $\abs{\ld_j} =\Theta(j^{-1/2-s}S(j))$ as \( j\to\infty \) for some slowly varying function \( S \) at infinity. Let \( J_N=J_N(\al,p) \) in \cref{eqn:theory_expectation_rho_J}. Then
	\begin{equation}\label{eqn:theory_expectation_rvsharp}
	\E^{D_N}\E^{L^{(N)}\sim\post}\norm[\big]{\Ld - L^{(N)}}_{L^2_{\nu'}(H;H)}^2
	= \Theta\bigl(\rho_N + N^{-\left(\frac{\al'+s}{\al+p}\right)}S^2(J_N)\bigr) \qas N\to\infty \,.
	\end{equation}
\end{theorem}

Although we have thus far restricted our attention to the Gaussian white noise model \cref{eqn:intro_ideas_ip_operator}, the next corollary shows that our theory remains valid for \emph{smoother} Gaussian noise.
\begin{corollary}[colored noise]\label{cor:theory_expectation_color}
	Suppose that the Gaussian distribution of the $\{\xi_n\}$ determining the data \( Y \) in \cref{eqn:intro_ideas_ip_operator} is not necessarily white, but is instead given by \( \pi=\normal(0,\Gamma) \), where \( \Gamma\in\cL(H) \) is symmetric positive-definite with eigenbasis \( \{\varphi_j\} \) shared with \( \Ld \) and eigenvalues \( \lambda_j(\Gamma)=\Theta(j^{-2\beta}) \) as \( j\to\infty \) for some \( \beta\geq 0 \). Let \( \postseq \) be given by \cref{eqn:posterior_sequence} except with each \( \gamma^2 \) replaced by \( \gamma^2\lambda_j(\Gamma) \). Let the hypotheses of \cref{thm:upper-expect,thm:lower-expect,thm:theory_expectation_rvsharp} hold, respectively, except let \( \rho_N=\rho_N(\al-\beta,\al',p) \), \( J_N=J_N(\al-\beta,p) \), and instead of \ref{item:as_exponent}, let \( \min\{\al-\beta,\al'\} + \min\{p-1/2, s\} >0 \). Then the results of \cref{thm:upper-expect,thm:lower-expect,thm:theory_expectation_rvsharp} remain valid, respectively, if in each display \cref{eqn:upper-expect1,eqn:upper-expect1-add,eqn:lower-expect,eqn:theory_expectation_rvsharp} every instance of \( \al \) is replaced by \( \al-\beta \).
\end{corollary}

The corollary follows from the hypothesis that \( \Gamma \) and \( \Ld \) commute. Indeed, pre-whitening the new output data gives \( \Gamma^{-1/2}y_n=\Ld\Gamma^{-1/2}x_n+\gamma\normal(0,\id) \), which our existing theory can handle. This result implies that larger \( \beta \) (smoother noise) improves convergence rates because the input data smoothness has effectively been reduced from \( \al \) to \( \al-\beta \) (see principle \ref{item:intro_ideas_result_input}).

\subsection{Posterior contraction}\label{sec:theory_contraction}
The performance of Bayesian procedures is often quantified by the rate of contraction of the posterior around the true data-generating parameter as $ N\to\infty $. In the setting of operator learning, we follow \cite{agapiou2013posterior,agapiou2014bayesian,knapik2018general,knapik2011bayesian} and consider finding a positive sequence $ \ep_N\to 0 $ such that for any positive sequence $ M_N\to\infty $, it holds that
\begin{equation}\label{eqn:posterior_contraction}
\E^{D_N}\post\bigl(\{ L\colon  \norm{\Ld - L}_{L_{\nu'}^{2}(H;H)}\geq M_N\ep_N \} \bigr)\to 0 \qas N\to\infty\, .
\end{equation}
We say that $ \ep_N $ is a \emph{contraction rate} of the posterior $ \post $ with respect to the $ L_{\nu'}^2(H;H) $ Bochner norm. By Chebyshev's inequality, \( (M_n\ep_N)^{-2} \) times the posterior test error \cref{eqn:error_bochner_intro}
is an upper bound for the left hand side of \cref{eqn:posterior_contraction}.
Thus, the limit in \cref{eqn:posterior_contraction} holds true if \cref{eqn:error_bochner_intro} is $ O(\ep_N^2) $ as $ N\to\infty $. The next corollary is then a consequence of \cref{thm:upper-expect}.
\begin{corollary}[posterior contraction]\label{cor:spc}
	Let the hypotheses of \cref{thm:upper-expect} be satisfied. Then any sequence $ \{\ep_N\}_{N\in\N} $ such that $ \ep_N^2 $ is of the order of the right hand side of \cref{eqn:upper-expect1} as $ N\to\infty $ is a contraction rate of $ \post $ with respect to the $ L_{\nu'}^2(H;H) $ Bochner norm.
\end{corollary}

We deduce that the inverse problem \cref{eqn:intro_ideas_ip_operator} for linear operator learning is \emph{moderately ill-posed} \cite[sect. 4]{agapiou2018posterior} under \cref{assump:theory_assumptions_main,assump:theory_expectation_train} because $ \ep_N^2 $ follows a power law \cref{eqn:upper-expect1}. Since $ \post $ is Gaussian, \cref{eqn:error_bochner_intro} admits a decomposition into three terms: the squared estimation bias, estimation variance, and posterior spread (i.e., the trace of $ \Sigma^{(N)} $) \cite[sect. 1.1]{agapiou2018posterior}. Inspection of the proof of \cref{thm:upper-expect} shows that the second term on the right of \cref{eqn:upper-expect1} is the contribution from the squared estimation bias, while the first is from both the estimation variance and posterior spread \cref{eqn:proofs_iall}. Interpretations are similar for the remaining theorems.

\subsection{High probability bounds}\label{sec:theory_probability}
A stronger assumption on the input data distribution is needed to obtain concentration bounds. It includes Gaussian measures as a special case.
\begin{assumption}[high probability: training data]\label{assump:theory_probability_train}
	The training data distribution \( \nu \) is a centered Borel probability measure on \( H \) with KL expansion \( x=\sum_{k=1}^\infty \lambda_k \zeta_k\phi_k \sim\nu\). The eigenvalues \( \{\lambda_k^2\} \) of \( \Cov[\nu]=\Lambda \) are ordered to be nonincreasing, and the zero mean and unit variance r.v.s \( \{\zeta_k\} \) are independent \( \sigma_{\nu}^2 \)-subgaussian for some absolute constant \( \sigma_{\nu}\geq 1 \). In particular, \( \nu \) is a strict subgaussian measure with trace-class covariance operator proxy \( \Lambda\).\footnote{A centered real-valued r.v. \( Z \) is \( \sigma^2 \)-subgaussian, denoted by \( Z\in\SG{\sigma^2} \), if \( \E\exp(tZ)\leq \exp(\sigma^2t^2/2) \) for all \( t\in\R \) \cite{wainwright2019high}. On a Hilbert space \( (H,\ip{\cdot}{\cdot}) \), a centered \( H \)-valued r.v. \( x \) is subgaussian with respect to trace-class covariance operator proxy \( Q\in\cL(H) \), denoted by \( x\in\SG{Q} \), if there exists \( q\geq 0 \) such that \( \E\exp(\ip{h}{x})\leq \exp(q^2\ip{h}{Qh}/2) \) for all \( h\in H \) \cite{antonini1997subgaussian}. It is strictly subgaussian if \( Q\preccurlyeq c \E[x\otimes x] \) for some \( c>0 \).
	}
\end{assumption}

The next result holds
with high probability over the subgaussian design $ X\sim \nu^{\otimes N} $.
\begin{theorem}[high probability: upper and lower bounds]\label{thm:upper-lower-prob}
	Let the ground truth \( \Ld \), prior \( \mu \) on \( L \), training data distribution \( \nu \), and test data distribution \( \nu' \) satisfy \cref{assump:theory_assumptions_main,assump:theory_probability_train}.  Let \( J_N=J_N(\al,p) \) and \( \rho_N=\rho_N(\al,\al',p) \) in \cref{eqn:theory_expectation_rho_J} with \( \al\), \(\al' \), and \( p \) as in \cref{assump:theory_assumptions_main}. Denote by \( \post \) the posterior distribution \cref{eqn:setup_bayes_posterior_notation} for \( L \) arising from the observed data \( D_N= (X,Y) \) in \cref{eqn:intro_ideas_ip_operator}. There is an absolute constant \( c_1>0 \) for which the following holds. Suppose $ \delta\in(0,1\mmin c_1\sigma_{\nu}^2) $. Define $ \Ndm\defeq (1-\delta)N $ for \( N\in\N \). Then there exist two constants $ c_2=c_2(\delta)>0 $ and $ c_3\in(0,1/(2c_1^2\sigma_{\nu}^4)) $ such that, as $ N\to\infty $, it holds that
	\begin{equation}\label{eqn:upper-prob1}
	\E^{Y\condbar X}\E^{L^{(N)}\sim \post}\norm[\big]{\Ld-L^{(N)}}^{2}_{L_{\nu'}^{2}(H;H)}
	=
	O\bigl((\tfrac{2\delta}{1-\delta}\one_{\{\al'<\al+1/2\}}+1)\rho_{\Ndm}\bigr)
	+
	o\bigl(\Ndm^{-\left(\frac{\al'+s}{\al+p}\right)}\bigr)
	\end{equation}
	with probability at least $ 1-c_2\exp(-c_3N\delta^2) $ over $ X\sim\nu^{\otimes N} $ and
	\begin{equation}\label{eqn:lower-prob}
	\E^{Y\condbar X}\E^{L^{(N)}\sim \post}\norm[\big]{\Ld-L^{(N)}}^{2}_{L_{\nu'}^{2}(H;H)}
	=
	\Omega\Bigl(
	\rho_N
	+
	\sum\nolimits_{j>J_N}j^{-2\al'}\abs{\ld_{j}}^2
	\Bigr)
	\end{equation}
	with probability at least $ {1-c_2\exp(-c_3N\delta^2)} $ over $ X\sim\nu^{\otimes N} $. Both assertions remain valid if the inner expectations are removed and \( L^{(N)} \) is replaced by the posterior mean \( \bar{L}^{(N)} \), and, for the lower bound only, the term \( \rho_N \) is multiplied by $ 1-\delta $ on the right hand side of \cref{eqn:lower-prob}.
\end{theorem}

We explicitly see that the probability of failure for \cref{thm:upper-lower-prob} is exponentially small in the sample size. The implications of this theorem are the same as those of \cref{thm:upper-expect}. The corresponding lower bounds are analogous to the in-expectation results from \cref{thm:lower-expect}.

\subsection{Excess risk}\label{sec:theory_excess_risk}
In the previous two subsections, we bounded the test error \cref{eqn:error_bochner_intro} from above and below. It follows that corresponding bounds for the excess risk \( \cE_N \) \cref{eqn:excess_risk} may be obtained by specializing to the in-distribution case \( \nu'=\nu \) for the posterior mean (so $ \al'=\al $).
\begin{corollary}[expected excess risk: upper and lower bounds]\label{cor:excess_risk_expect}
	Let the hypotheses of \cref{thm:upper-expect,thm:lower-expect} be satisfied. Then the expected excess risk \( \E^{D_N}\cE_N \) satisfies the bounds
	\begin{equation}\label{eqn:excess_risk_expect_upper}
	\E^{D_N}\E^{x\sim\nu}\norm{\Ld x-\bar{L}^{(N)}x}^2
	=
	O\bigl(N^{-\left(\frac{\al+p-1/2}{\al+p}\right)}\bigr)
	+
	o\bigl(N^{-\left(\frac{\al+s}{\al+ p}\right)}\bigr)
	\qas N\to\infty\,,
	\end{equation}
	and for any positive sequence \( \{\tau_n\} \) such that $ \tau_{n}\to 0 $ and \( n\tau_n\to\infty \) as \( n\to\infty \), it holds that
	\begin{equation}\label{eqn:excess_risk_expect_lower}
	\E^{D_N}\E^{x\sim\nu}\norm{\Ld x-\bar{L}^{(N)}x}^2
	=
	\Omega\Bigl(
	\tau_N N^{-\left(\frac{\al+p-1/2}{\al+p}\right)}
	+
	\sum\nolimits_{j>J_N}j^{-2\al}\abs{\ld_{j}}^2
	\Bigr)
	\ \ \text{as}\ \ N\to\infty\,.
	\end{equation}
\end{corollary}

\Cref{cor:excess_risk_expect} is proved as a consequence of \cref{thm:upper-expect,thm:lower-expect}. A similar result may be established for $ \E^{Y\condbar X}\cE_{N} $ by using \cref{thm:upper-lower-prob}. We omit the details for brevity. It is also interesting that \emph{fast rates} for the excess risk (i.e., faster than $ N^{-1/2} $ \cite{mathieu2021excess}) are attained by the posterior mean eigenvalue estimator in certain regimes. The usual statistical learning techniques based on bounding suprema of empirical processes typically yield slow $ N^{-1/2} $ rates or worse \cite{tabaghi2019learning}. Our results are sharper because we use explicit diagonal calculations.

\subsection{Generalization gap}\label{sec:theory_gen_gap}
Last, we estimate the generalization gap \cref{eqn:gen_gap} in $ L^1_{\P}(\varOmega;\R) $.
\begin{theorem}[expected generalization gap: upper and lower bounds]\label{thm:gen_gap_expect}
	Let the hypotheses of \cref{thm:upper-expect} be satisfied. Then for \( \cG_N \) as in \cref{eqn:gen_gap}, it holds that
	\begin{equation}\label{eqn:thm_gen_gap-expect}
	\EGG=
	O\bigl(N^{-\left(\frac{1}{2}\mmin\frac{\al+p-1/2}{\al+p}\right)}\bigr)\qas N\to\infty\, .
	\end{equation}
	Additionally, for any positive sequence \( \{\tau_n\} \) such that $ \tau_{n}\to 0 $ and \( n^{1/2}\tau_n\to\infty \) as \( n\to\infty \),
	\begin{equation}\label{eqn:thm_gen_gap-expect-lower}
	\EGG=
	\Omega\bigl(\tau_NN^{-\left(\frac{\al+p-1/2}{\al+p}\right)}\bigr)
	\qas N\to\infty
	\end{equation}
	if $ (\al+s)/(\al+p)\geq 2 $. Otherwise, the previous assertion \cref{eqn:thm_gen_gap-expect-lower} remains valid provided that $ p<1+\al+2s $ 
	and $\tau_n\gg n^{-1/2} \mmax n^{-(1+\al+2s-p)/(2\al+2p)} $ as $ n\to\infty $.
\end{theorem}

We see that the expected generalization gap decays at least as fast as the standard Monte Carlo rate $ N^{-1/2} $ if $\alpha+p \ge 1$. Otherwise, it decays at a slower rate that is arbitrarily slow as $\alpha+p$ approaches $1/2$ from above. The lower bound only matches the latter contribution.

\section{Numerical studies}\label{sec:numerics}
We now instantiate our operator learning framework numerically, both according to the theory (\cref{sec:numerics_within}) and beyond (\cref{sec:numerics_beyond}). For clarity, we only implement the posterior mean estimator $ \bar{L}^{(N)} $.
Our conceptually infinite-dimensional problem must be carefully discretized to avoid obscuring the theoretical infinite-dimensional behavior \cite[sect. 1.2]{agapiou2014analysis}. We use spectral truncation \cite{agapiou2014analysis,agapiou2021designing} to finite-dimensionalize infinite sequence spaces. For $ v=\{v_j\}\in\R^{\infty} $, its truncation is $ v^{(J)}\defeq\{v_j\}_{j\leq J}\in\R^{J} $ for $ J\in\N $ ``Fourier'' modes. We use the relative expected squared $ L_{\nu'}^2 $ Bochner norm as a numerical error metric, given by
\begin{equation}\label{eqn:relative_error}
\E^{D_N}\E^{x\sim\nu'}\norm{\Ld x-\bar{L}^{(N)}x}^{2}/\E^{x\sim\nu'}\norm{\Ld x}^{2}
=
\E^{D_N}\sum\nolimits_{j=1}^{\infty}\vartheta_j'^{2}\abs[\big]{\ld_j-\bar{l}_{j}^{(N)}}^{2}/\sum\nolimits_{k=1}^{\infty}\vartheta_k'^2\abs{\ld_k}^{2}\, .
\end{equation}

\subsection{Within the theory}\label{sec:numerics_within}
We now confirm the theoretical results of this paper with simulation studies. 
Define $ A\colon \dom(A)\subset H\to H $ by $ h\mapsto Ah\defeq -\lap h $ with domain $ \dom(A)\defeq H_0^1(I;\R)\cap H^{2}(I;\R) $, where $ I\defeq (0,1) $, $ H\defeq L^2(I;\R) $, and $ \lap $ is the Laplacian (i.e., second derivative). We consider truths $ \Ld=A,\id$, and $A^{-1} $ corresponding to unbounded, bounded, and compact self-adjoint operators on $ H $, respectively. The map $ A $ is diagonalized in the orthonormal basis $ \{\varphi_j\} $ of $ H $ given by $ z\mapsto \varphi_j(z)=\sqrt{2}\sin(j\pi z) $. This is the output space basis used henceforth. Then $ \Ld=A,\id$, and $A^{-1} $ have eigenvalue sequences $ \ld=\{(j\pi)^2\}, \{1\}$, and $ \{(j\pi)^{-2}\}\in\cH^{s} $ for any $ s<s^{\star} $, where $ s^{\star}=-5/2, -1/2$, and $ 3/2 $, respectively. These eigenvalues are regularly varying (with \( S\equiv 1 \)) as in \cref{thm:theory_expectation_rvsharp}.

We work in the Gaussian setting of \cref{thm:intro_ideas_thm}. We choose Mat\'ern-like covariances
\begin{equation}\label{eqn:covariance_matern}
\Lambda = \tau_1^{2\al-1}(A + \tau_{1}^{2}\id)^{-\al}\qa \Lambda' = \tau_2^{2\al'-1}(A + \tau_{2}^{2}\id)^{-\al'}
\end{equation}
for $ \nu $ and $ \nu' $. Here $ \{\tau_i\}_{i=1,2} $ are inverse length 
scales. Draws from $\nu$ (resp. $\nu'$) are in $\cH^{s'}$
for all $s'<\al-1/2$ (resp. $s'<\al'-1/2$). Notice that $ \Ld $, \( \Lambda \), and \( \Lambda' \) are simultaneously diagonalizable in $ \{\phi_j\equiv\varphi_j\} $.  The eigenvalues are $ \lambda_j(\Lambda)=\vartheta_j^2=\tau_1^{2\al-1}((j\pi)^2 + \tau_{1}^{2})^{-\al}\asymp j^{-2\al} $ and similarly for $ \lambda_j(\Lambda')=\vartheta_j'^2\asymp j^{-2\al'} $. These satisfy assumption~\ref{item:as_smooth_data}. We directly define the prior covariance $ \Sigma $ in sequence space according to assumption~\ref{item:as_smooth_prior}, choosing $ \sigma_j^2\defeq\tau_3^{2p-1}((j\pi)^2 + \tau_{3}^{2})^{-p}\asymp j^{-2p} $ for \( \tau_3>0 \). We enforce assumption~\ref{item:as_exponent} for the values of $ \al, \al'$, and $p $.

An independent random dataset $ D_N $ (as in \cref{eqn:intro_ideas_ip_diagonal}) is generated for each sample size $ N\in\N $ to construct $ \bar{l}^{(N)} $. For each \( N \), this is repeated $ 250 $, $ 500 $, or $ 1000 $ times for $ \Ld=A, \id$, and $ A^{-1} $, respectively, to approximate the outer expectation in \cref{eqn:relative_error} by sample averages. Convergence rates are produced by linear least square fits to the logarithm of computed errors. We fix the noise scale to be $ \gamma = 10^{-1}, 10^{-3}$, and $ 10^{-5} $ for $ \Ld  = A, \id$, and $ A^{-1} $, respectively.

\subsubsection{In-distribution}\label{sec:numerics_within_in}
We set $ \al=\al'=4.5 $ (in-distribution), $ \tau_1=\tau_2=15 $, \( \tau_3=1 \), and define the prior smoothness $ p=p(\Ld)=1/2+s^{\star}(\Ld)+z $, where $ z=-0.75, 0$, or $ 0.75 $ is a fixed shift to replicate rough, matching, or smooth priors, respectively. Sequences are discretized by keeping up to $ J=2^{16}=65,536 $ Fourier modes. The sample size is $ N\in\{2^{4}, 2^{5},\ldots, 2^{14}\} $. \Cref{tab:ratesPT} empirically verifies our sharp theoretical predictions from \cref{thm:theory_expectation_rvsharp} for \( \ER \). The convergence as \( N \) increases is visualized in \cref{fig:scaled_rates_wt} for the smooth prior case.

\begin{table}[tbhp]
	\centering
	\footnotesize
	\caption{Matching test measure. Theoretical v.s. experimental (in parentheses) convergence rate exponents $ r $ in $ O(N^{-r}) $ of the relative expected squared $ L_{\nu}^{2}(H;H) $ in-distribution error (i.e., the scaled excess risk $ \ER $).
	}
	\label{tab:ratesPT}
	\renewcommand{\arraystretch}{1.2}
	\begin{tabular}{l@{\hspace{0.25mm}}lccc}
		\toprule
		\multicolumn{1}{l}{$ \Ld $} &
		\{\emph{Operator Class}\} &
		\multicolumn{1}{c}{Rough Prior} &
		\multicolumn{1}{c}{Matching Prior} &
		\multicolumn{1}{c}{Smooth Prior} \\
		\midrule
		
		$A$ &\{\emph{Unbounded}\} &  0.714 (0.714) & 0.800 (0.809) & 0.615 (0.616)\\
		
		$\id$ &\{\emph{Bounded}\} &  0.867 (0.865) & 0.889 (0.889) & 0.762 (0.762)\\
		
		$A^{-1}$ &\{\emph{Compact}\} &  0.913 (0.913) & 0.923 (0.920) & 0.828 (0.830)\\
		\bottomrule
	\end{tabular}
\end{table}

Moving on to study the rates of convergence of $ \ER $ and $ \EGG $ for unbounded $ \Ld=A $ in more detail, we now use \emph{$ N $-dependent} spectral truncation. For each $ N $, we only take Fourier modes from the set $ \{j\in\N\colon  j\leq cJ_N\} $, where $ c>0 $ is a tunable constant and $ J_N\defeq N^{1/(2\al+2p)}\ll N $. This approach is justified because it is more stable numerically and the results in \cref{sec:theory} remain valid with this $ N $-dependent truncation. Contributions from the tail set $ \{j\in\N\colon  j> cJ_N\} $ are of equal order or negligible, asymptotically, relative to those from the truncated set (\cref{app:proofs}). \Cref{fig:sweep_risk_and_gap} shows results with $ N $ up to $ 2^{21} $ and $ c $ such that $ cJ_{2^{21}}\approx2^{14}$ (maximal truncation level). The influence of discretization manifests itself through $ \gamma $. For $ \ER $, the under-smoothing prior region ($ z<0 $) is relatively insensitive to $ \gamma $ and the rate exponents closely match \cref{eqn:excess_risk_expect_upper}. But in the over-smoothing prior region $ z>0 $ for large $ \gamma $, the rates begin to deviate from the theory because large constants mask the theoretical asymptotic behavior in this finite sample regime. Similarly, for finite \( N \), the noise scale can alter the correct behavior of the competing terms in the bound \cref{eqn:thm_gen_gap-expect} for $ \EGG $. For small $ \gamma $,  terms $ O(N^{-1/2}) $ have large hidden constants that obscure terms \( \gg N^{-1/2} \) for small $ z<0 $ (\cref{fig:sweep_gap_lowest}). For large $ \gamma $, this behavior is reversed (\cref{fig:sweep_gap_high}).

\begin{figure}[htbp]%
	\centering
	\subfloat[$ \ER $, $ \gamma^2=10^{-6} $]{\includegraphics[width=0.25\textwidth]{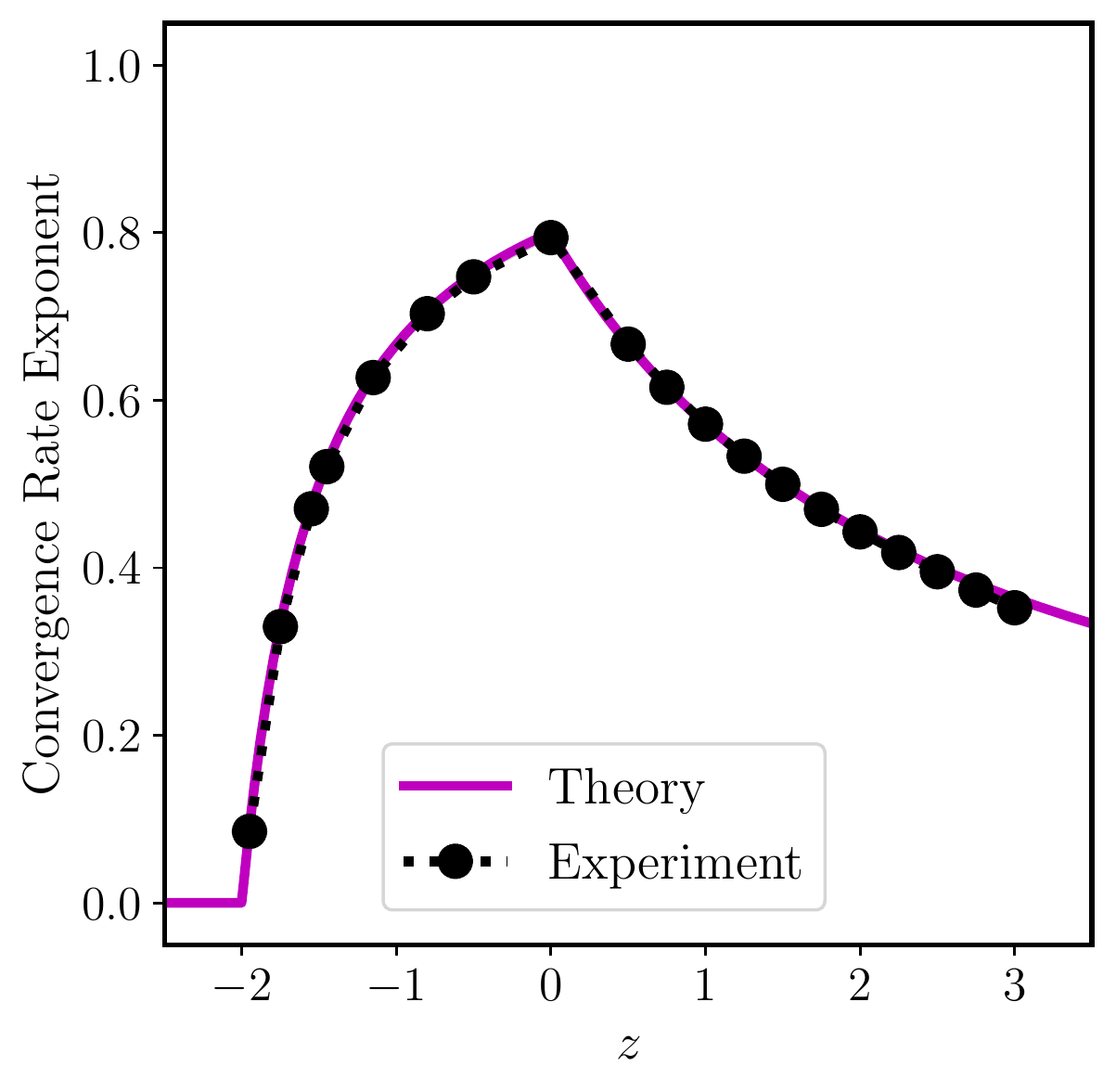}\label{fig:sweep_risk_lowest}}\hfill%
	\subfloat[$ \ER $, $ \gamma^2=25^2 $]{\includegraphics[width=0.237\textwidth]{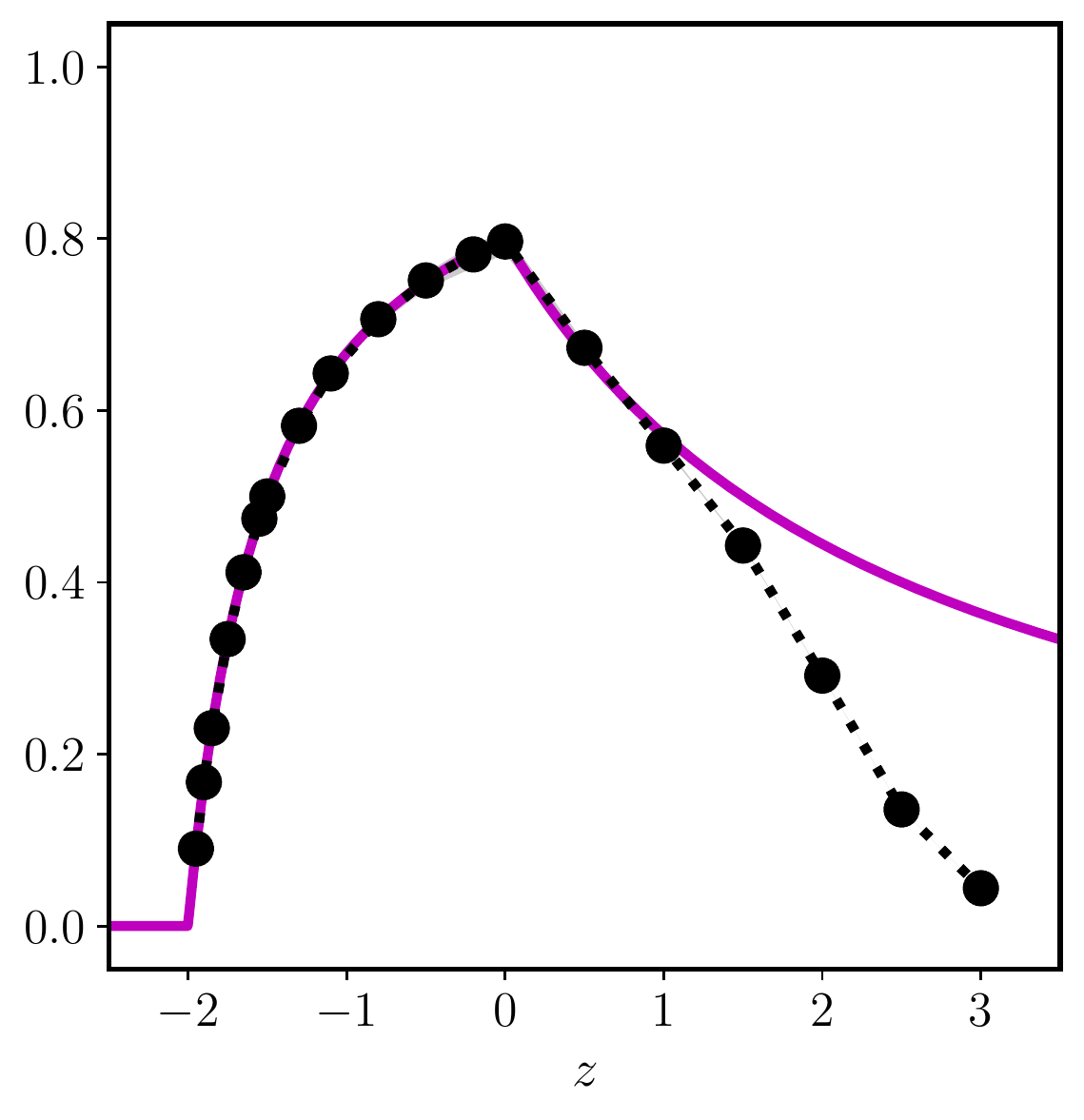}\label{fig:sweep_risk_high}}\hfill%
	\subfloat[$ \EGG $, $ \gamma^2=10^{-6} $]{\includegraphics[width=0.237\textwidth]{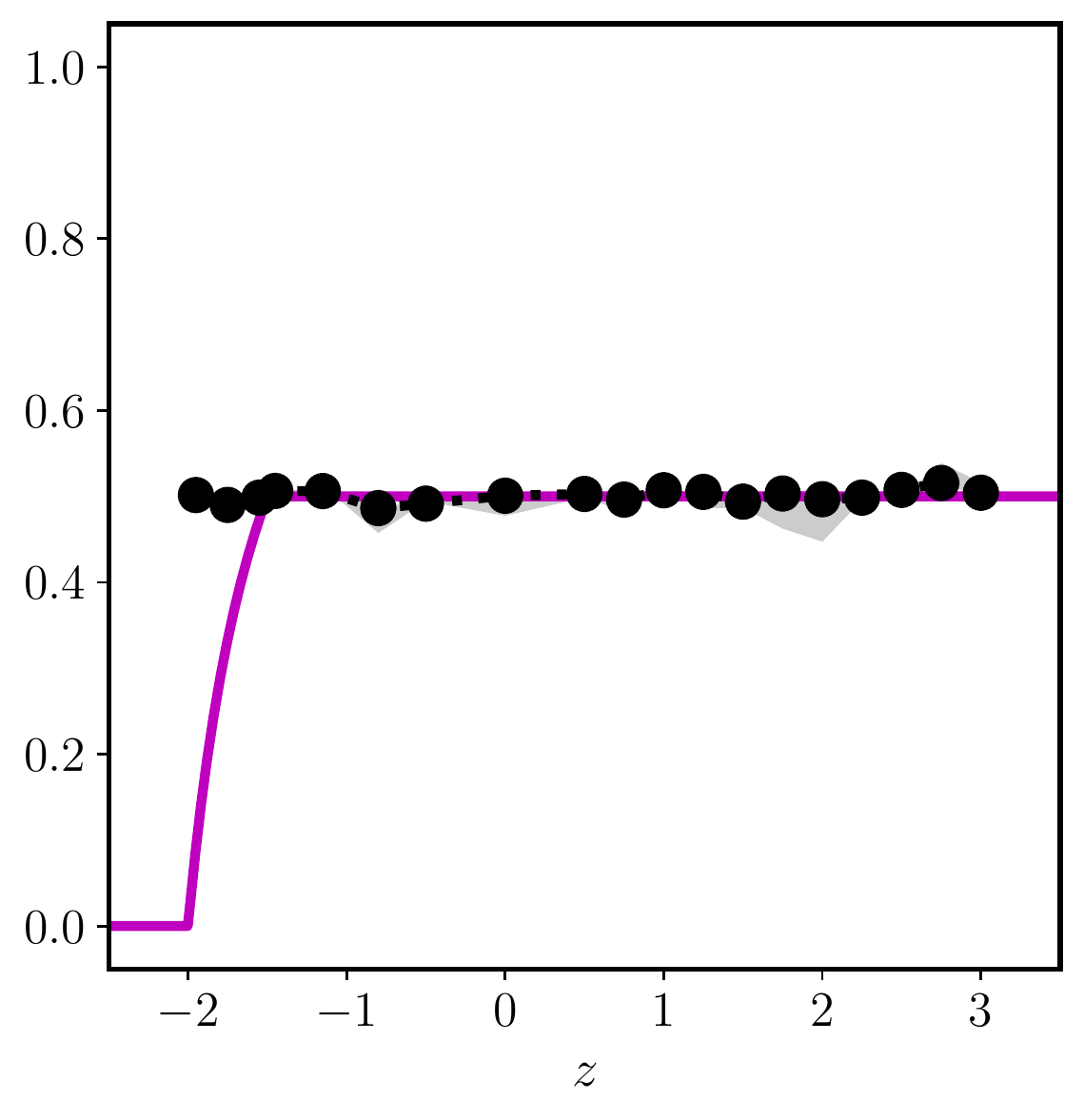}\label{fig:sweep_gap_lowest}}\hfill%
	\subfloat[$ \EGG $, $ \gamma^2=25^2 $]{\includegraphics[width=0.237\textwidth]{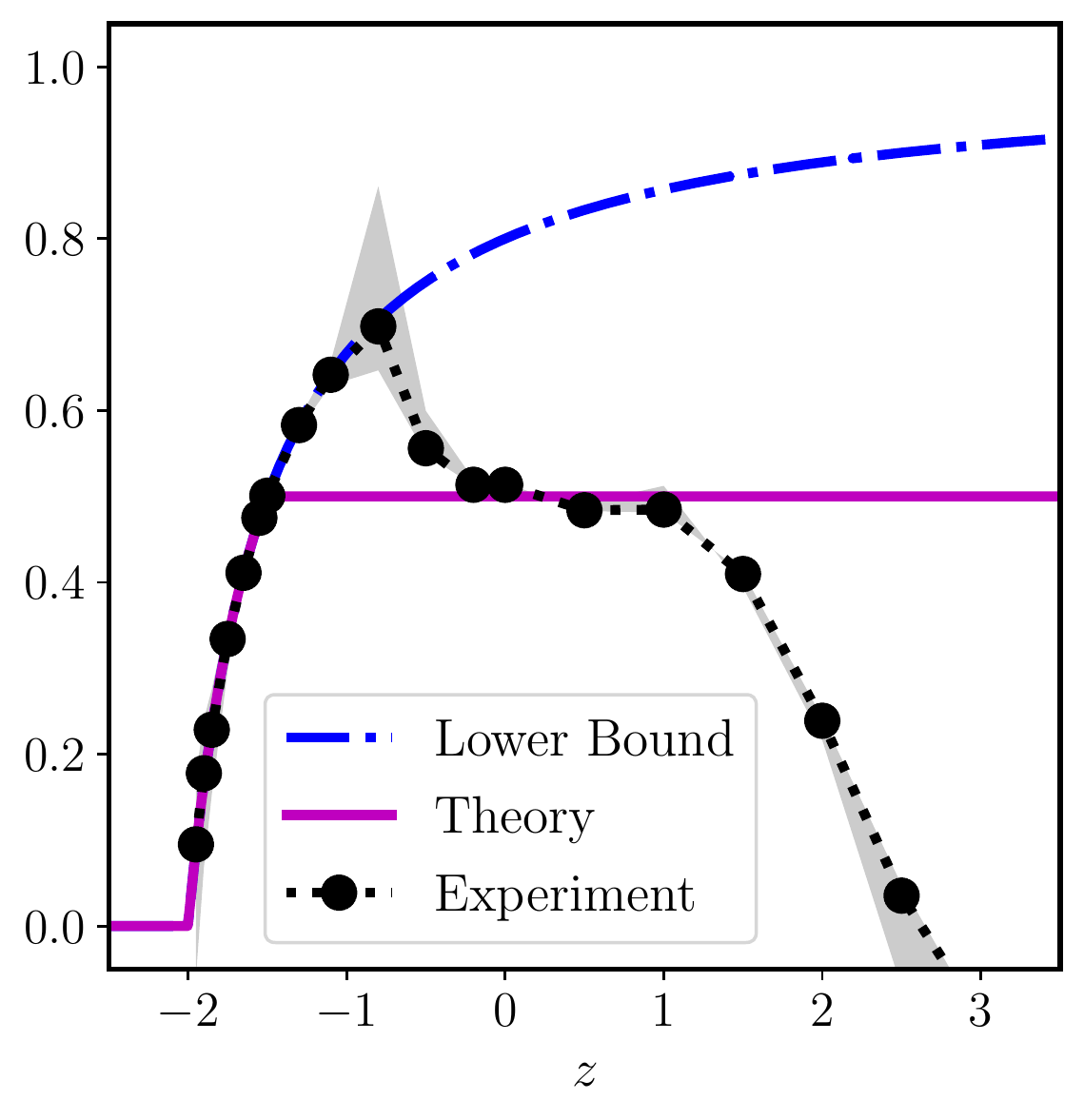}\label{fig:sweep_gap_high}}
	\caption{The numerical influence of data noise variance \( \gamma^2 \) for $ \Ld=A $. For two distinct $ \gamma^2 $ values, \cref{fig:sweep_risk_lowest,fig:sweep_risk_high} show convergence rate exponents for $ \E^{D_N}\cE_N $ v.s. $ z $, with $ z=p+2 $ being the prior smoothness shift parameter, while \cref{fig:sweep_gap_lowest,fig:sweep_gap_high} display rates for $ \EGG $ v.s. $ z $. Throughout, the solid magenta ``Theory'' curves denote the theoretical upper bound rate exponents, and the shaded regions denote one standard deviation from the mean rate exponent computed from 250 repetitions of the numerical experiment.
	}
	\label{fig:sweep_risk_and_gap}		
\end{figure}

\subsubsection{Out-of-distribution}\label{sec:numerics_within_out}
We now vary $ \al' $ to simulate distribution shift. With \( J=2^{16} \) and \( \tau_2=15 \), our results in \cref{tab:ratesPT-rs} show near perfect agreement with \cref{thm:theory_expectation_rvsharp} for out-of-distribution regimes on both sides of the boundary case \( \al'=\al+1/2 \). In the matching prior setting $ (z=0) $, \cref{fig:match_L1_alldistr,fig:match_L2_alldistr,fig:match_L3_alldistr} show the decay of the test error \cref{eqn:relative_error} with \( N \). The magenta lines are least square fits and the shaded regions denote one standard deviation from the mean with respect to resampling \( D_N \). The excellent numerical fits verify our assertions.

\begin{table}[htbp]
	\centering
	\footnotesize
	\caption{Distribution shift. Theoretical v.s. experimental (in parentheses) convergence rate exponents $ r $ in $ O(N^{-r}) $ of the relative expected squared $ L_{\nu}^{2}(H;H) $ out-of-distribution error \cref{eqn:relative_error} for rougher and smoother test measures.
	}
	\label{tab:ratesPT-rs}
	\renewcommand{\arraystretch}{1.2}
	\begin{tabular}{lcccccc}
		\toprule
		&
		\multicolumn{3}{c}{\emph{Rougher Test Measure}: $ \al'=4<\al=4.5 $} &
		\multicolumn{3}{c}{\emph{Smoother Test Measure}: $ \al'=5.25>\al=4.5 $}\\
		\cmidrule(lr){2-4}
		\cmidrule(lr){5-7}
		
		\multicolumn{1}{l}{\( \Ld \)} &
		\multicolumn{1}{c}{Rough Prior} &
		\multicolumn{1}{c}{Matching Prior} &
		\multicolumn{1}{c}{Smooth Prior} &
		\multicolumn{1}{c}{Rough Prior} &
		\multicolumn{1}{c}{Matching Prior} &
		\multicolumn{1}{c}{Smooth Prior} \\
		\midrule
		
		$A$ &  0.429 (0.428) & 0.600 (0.607) & 0.462 (0.462) &  1.000 (0.992) & 1.000 (0.996) & 0.846 (0.849)\\

		$\id$ &  0.733 (0.734) & 0.778 (0.788) & 0.667 (0.667) &  1.000 (0.986) & 1.000 (0.979) & 0.905 (0.905)\\

		$A^{-1}$ &  0.826 (0.837) & 0.846 (0.861) & 0.759 (0.764) &  1.000 (0.981) & 1.000 (0.975) & 0.931 (0.926)\\
		\bottomrule
	\end{tabular}
\end{table}

\begin{figure}[htbp]%
	\centering
	\subfloat[$ \ER $ v.s. $ N $]{
	\includegraphics[width=0.209\textwidth]{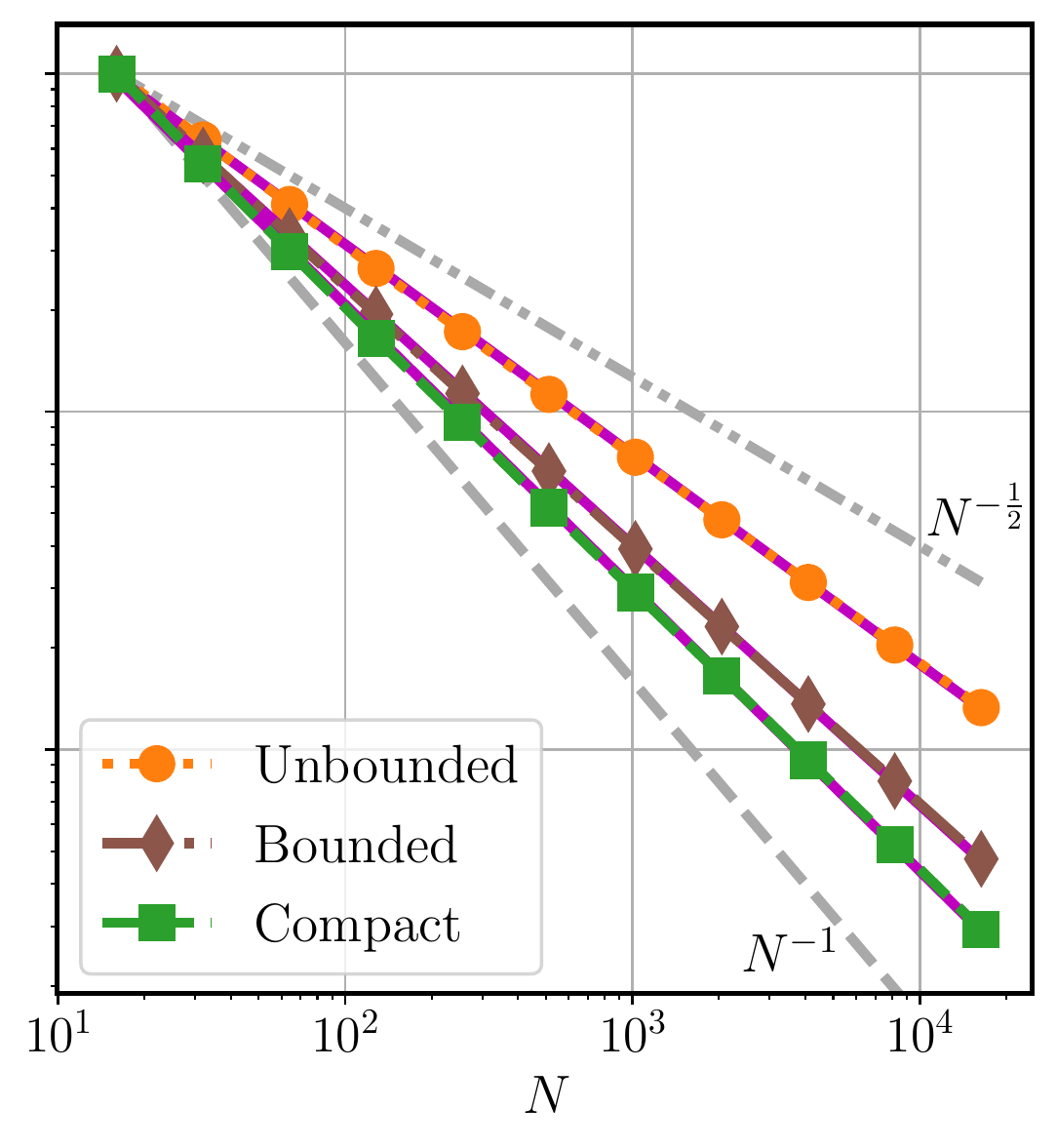}\label{fig:scaled_rates_wt}}\hfill%
	\subfloat[$ \Ld=A $, $ \gamma=10^{-1} $]{	\includegraphics[width=0.239\textwidth]{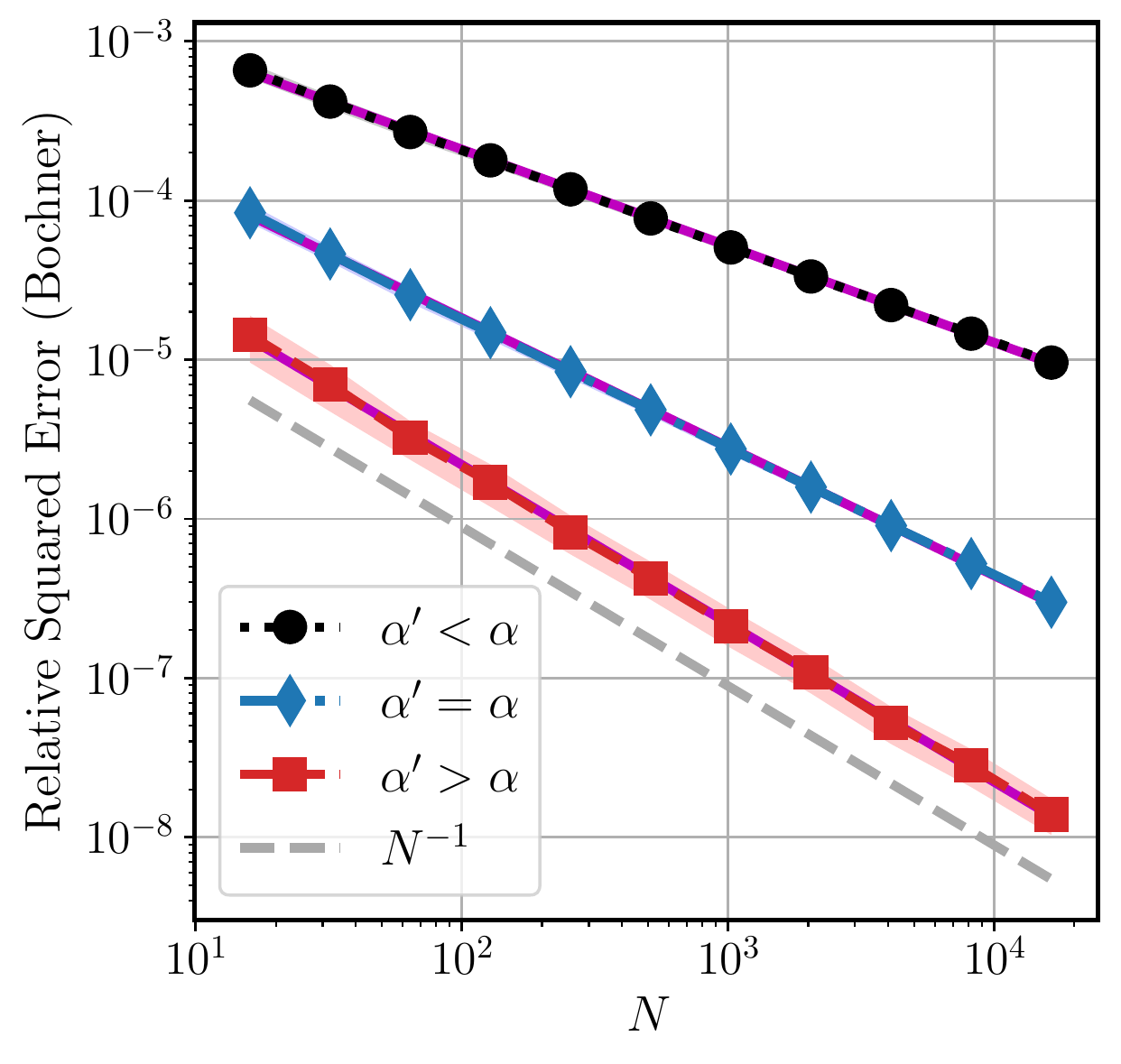}\label{fig:match_L1_alldistr}}\hfill%
	\subfloat[$ \Ld=\id $, $ \gamma=10^{-3} $]{	\includegraphics[width=0.225\textwidth]{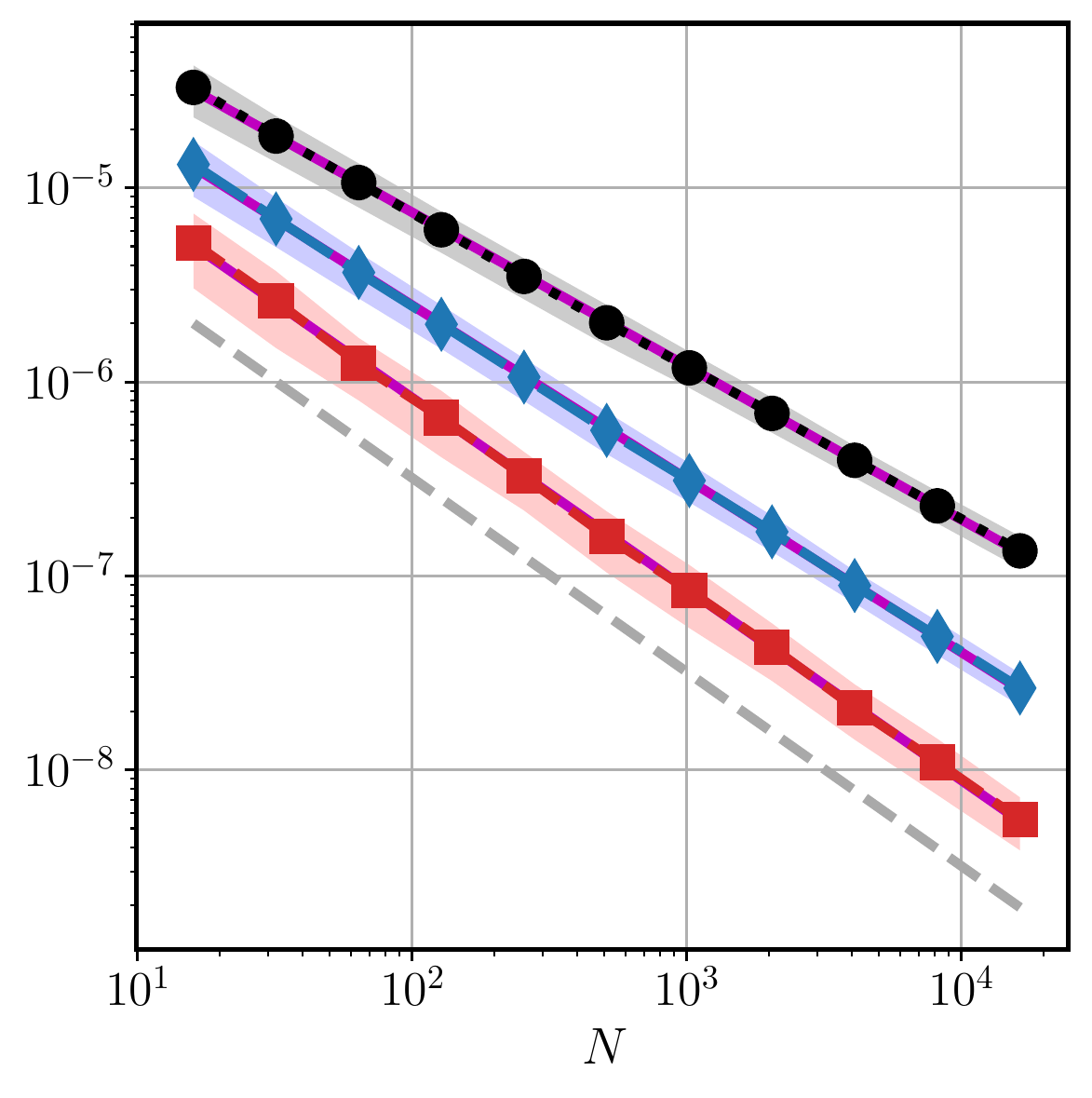}\label{fig:match_L2_alldistr}}\hfill%
	\subfloat[$ \Ld=A^{-1}$, $ \gamma=10^{-5} $]{	\includegraphics[width=0.225\textwidth]{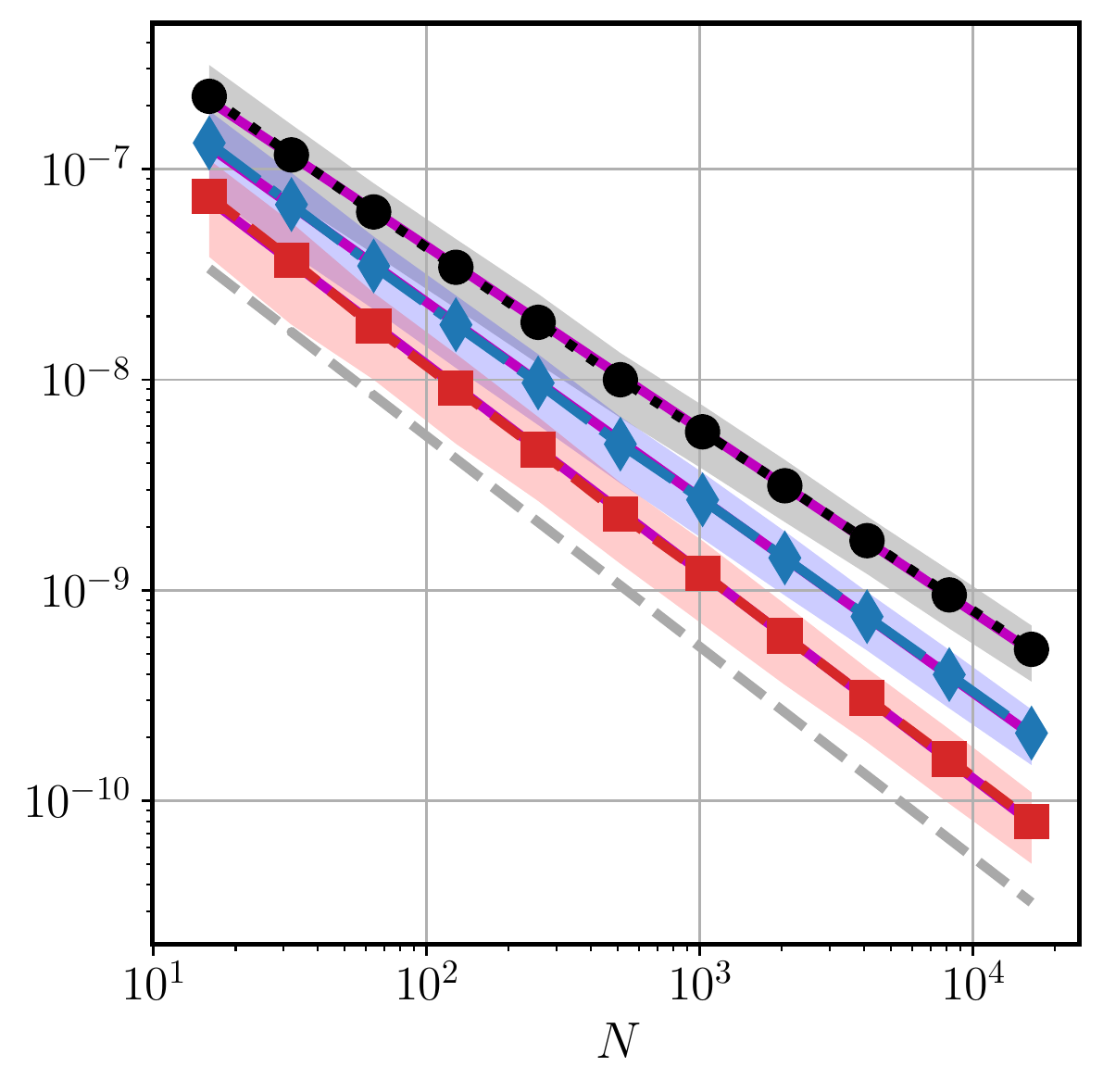}\label{fig:match_L3_alldistr}}
	\caption{Within the theory. \Cref{fig:scaled_rates_wt} (corresponding to \cref{tab:ratesPT} column four) shows that convergence improves with increased operator smoothing (the logarithmic vertical axis is rescaled to ease comparison of the slopes). \Cref{fig:match_L1_alldistr,fig:match_L2_alldistr,fig:match_L3_alldistr} are such that \( z=0 \) (matching \( p=s^{\star}+1/2 \)) and the test measures $ \nu' $ are either equal to ($ \al'=\al $), rougher than ($ \al'<\al $), or smoother than ($ \al'>\al $) the training measure $ \nu $. For fixed $ \Ld $, the same \( \bar{L}^{(N)} \) achieves smaller relative error \cref{eqn:relative_error} as $ \al' $ increases, that is, when testing against smoother input functions. In all cases, the observed rates closely match the theoretical ones (see \cref{tab:ratesPT,tab:ratesPT-rs}).}
	\label{fig:match_alldistr}
\end{figure}

\subsection{Beyond the theory}\label{sec:numerics_beyond}
In this subsection, we consider truths $ \Ld\in\HS(H_{\cK};H) $ (with \( \cK \) satisfying  \cref{cond:existence_of_K}) that are \emph{not necessarily diagonalized} by $ \{\varphi_j\} $. So, the infinite matrix $ \Ldmat\defeq \{\Ldmat_{jk}\} $ from \cref{eqn:intro_ideas_ip_matrix} must be estimated instead of $ \ld $. Recall that \( \Lambda \) has eigenpairs \( \{(\lambda_k^2, \phi_k)\} \). By \cref{fact:setup_bayes_measure}, $ \Ld\in\HS(H_{\Lambda};H) $ so the expansion $ \Ld=\sum_{i,j}(\lambda_j\Ldmat_{ij})\varphi_i\otimes_{H_{\Lambda}}(\lambda_j\phi_j)=\sum_{i,j}\Ldmat_{ij}\varphi_i\otimes\phi_j $ always exists and is unique. Yet, we have no theory for posterior estimators of $ \Ldmat $. To derive the posterior mean, we notice that the inverse problem for $ \mathsf{L}\condbar D_N $ decouples along rows of $ \mathsf{L}=\{\mathsf{L}_{jk}\} $, which are denoted by $ \mathsf{L}_{j:} $ for $ j\in\N $. We assume a Gaussian prior $ \mathsf{L}_{j:}\sim\normal(0,\Sigma_j) $, where $ \Sigma_j=\diag(\{\sigma_{jk}^2\}_{k\in\N}) $ is diagonal for simplicity. Thus $ (\mathsf{L}_{j:})_k=\mathsf{L}_{jk}\sim\normal(0,\sigma_{jk}^2) $. By deriving the normal equations, we obtain for $ j,k$, and $\ell\in\N $ the posterior mean
\begin{equation}\label{eqn:outside_theory_posterior_mean}
\bar{\mathsf{L}}_{j:}^{(N)}=	\bigl(\mathsf{A}^{(N)}+\tfrac{\gamma^2}{N}\Sigma_j^{-1}\bigr)^{-1}\mathsf{b}_j^{(N)}\,,\qw  \mathsf{A}_{\ell k}^{(N)}\defeq \avgn{x_{\ell}}{x_k}\qa \bigl(\mathsf{b}_j^{(N)}\bigr)_{\ell}\defeq \avgn{y_j}{x_{\ell}}\,.
\end{equation}

We use the same covariances \cref{eqn:covariance_matern} diagonalized in Fourier sine input basis $ \{\phi_j\} $, but now use Volterra cosine output basis $ \{\varphi_j\} $ as in \cref{fig:assump_ex_data_decay}, where $ z\mapsto \varphi_j(z)\defeq \sqrt{2}\cos((j-\frac{1}{2})\pi z) $. Define the \emph{divergence form elliptic operator} $ A_a\colon \dom(A_a)\subset H\to H $ by $ h\mapsto A_ah\defeq -\grad\cdot(a\grad h) $, where $ \dom(A_a)=\dom(A) $ as before and $ z\mapsto a(z)\defeq \exp(-3z) $ is smooth. We learn (via $ \bar{\mathsf{L}}^{(N)} $) unbounded, bounded, and compact self-adjoint operators $\Ld=A_a,\id$, and $A_a^{-1} $, respectively. For each of the three $ \Ld $, we pick prior variance sequences $ \sigma_{jk}^2=\sigma_{jk}^{2}(\Ld) $  given by
\begin{equation}\label{eqn:beyond_prior}
\sigma_{jk}^2(\Ld)\defeq
\begin{cases}
(jk)^{-(z-2)}\bigl(\tfrac{1+(k/j)^2}{1+(j-k)^2}\bigr)^2, & \text{if }\, \Ld=A_a\,,\\[1mm]
(jk)^{-z}\bigl(\tfrac{k+k/j}{1+j+(j-k)^2}\bigr)^2, & \text{if }\, \Ld=\id\,,\\[1mm]
(jk)^{-(z+2)}\bigl(\tfrac{1+j/k}{1+(j-k)^2}\bigr)^2, & \text{if }\, \Ld=A_a^{-1}.
\end{cases}
\end{equation}
These priors ensure that $ \mathsf{L} $ matches the exact asymptotic behavior (as $ j\to\infty$, $k\to\infty$, and $j=k\to\infty $) of $ \Ldmat $ when $ z=0 $. Our simulation setup follows \cref{sec:numerics_within}, except now with $ J=2^{12} $, $ N $ up to $ 2^{14} $, and only $ 100 $ Monte Carlo repetitions. Although $ A_a $ is not diagonal in $ \varphi_j\neq\phi_j $ (each $ \Ldmat $ is dense) and the posterior mean estimator is now a doubly-indexed sequence, our results in \cref{fig:match_alldistr_ot} support the same conclusions previously asserted.

\begin{figure}[htbp]%
	\centering
	\subfloat[$ \ER $ v.s. $ N $]{
	\includegraphics[width=0.209\textwidth]{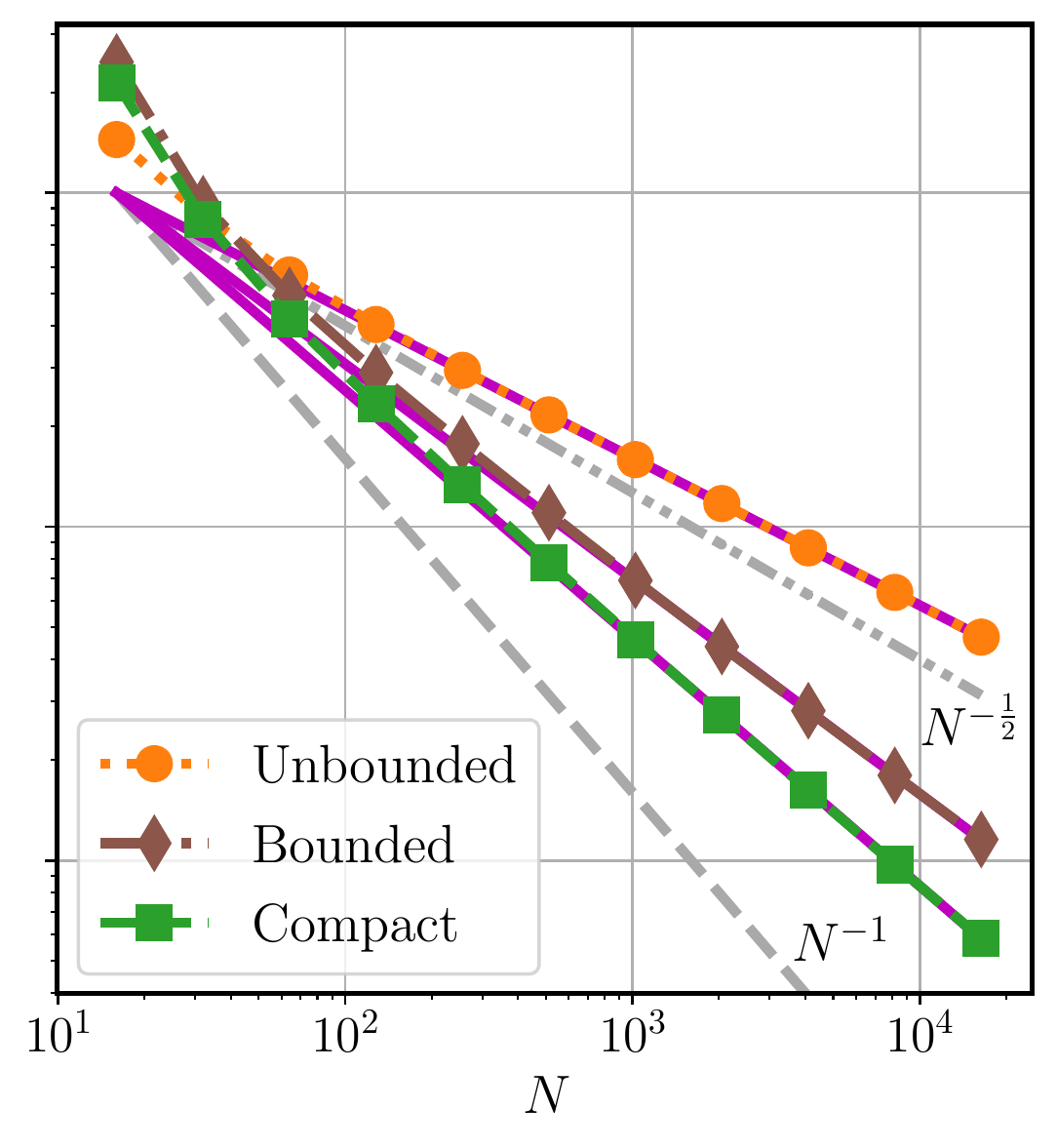}\label{fig:scaled_rates_ot}}\hfill%
	\subfloat[$ \Ld=A_a $, $ \gamma=10^{-1} $]{	\includegraphics[width=0.239\textwidth]{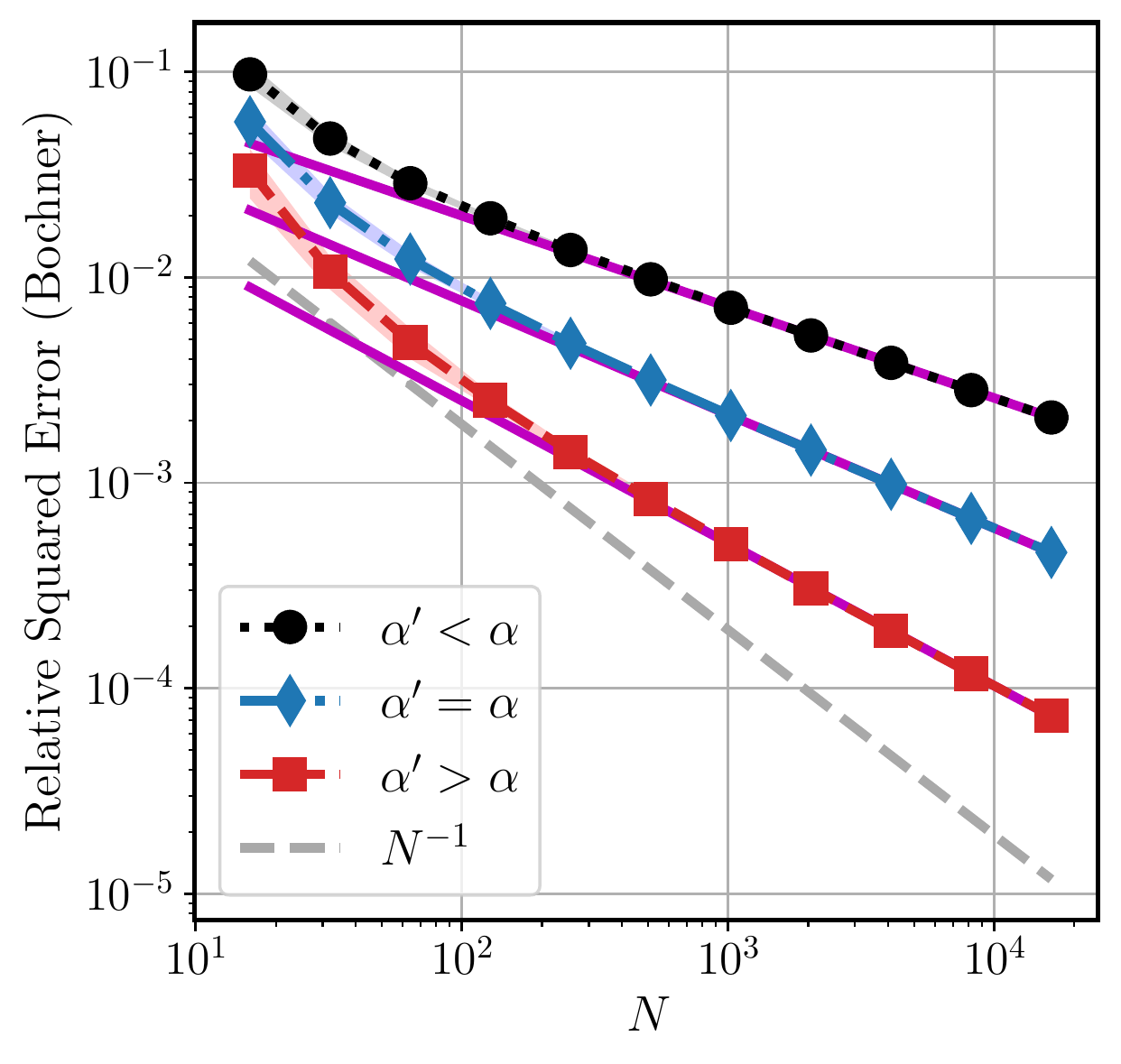}\label{fig:match_L1_alldistr_ot}}\hfill%
	\subfloat[$ \Ld=\id $, $ \gamma=10^{-3} $]{	\includegraphics[width=0.225\textwidth]{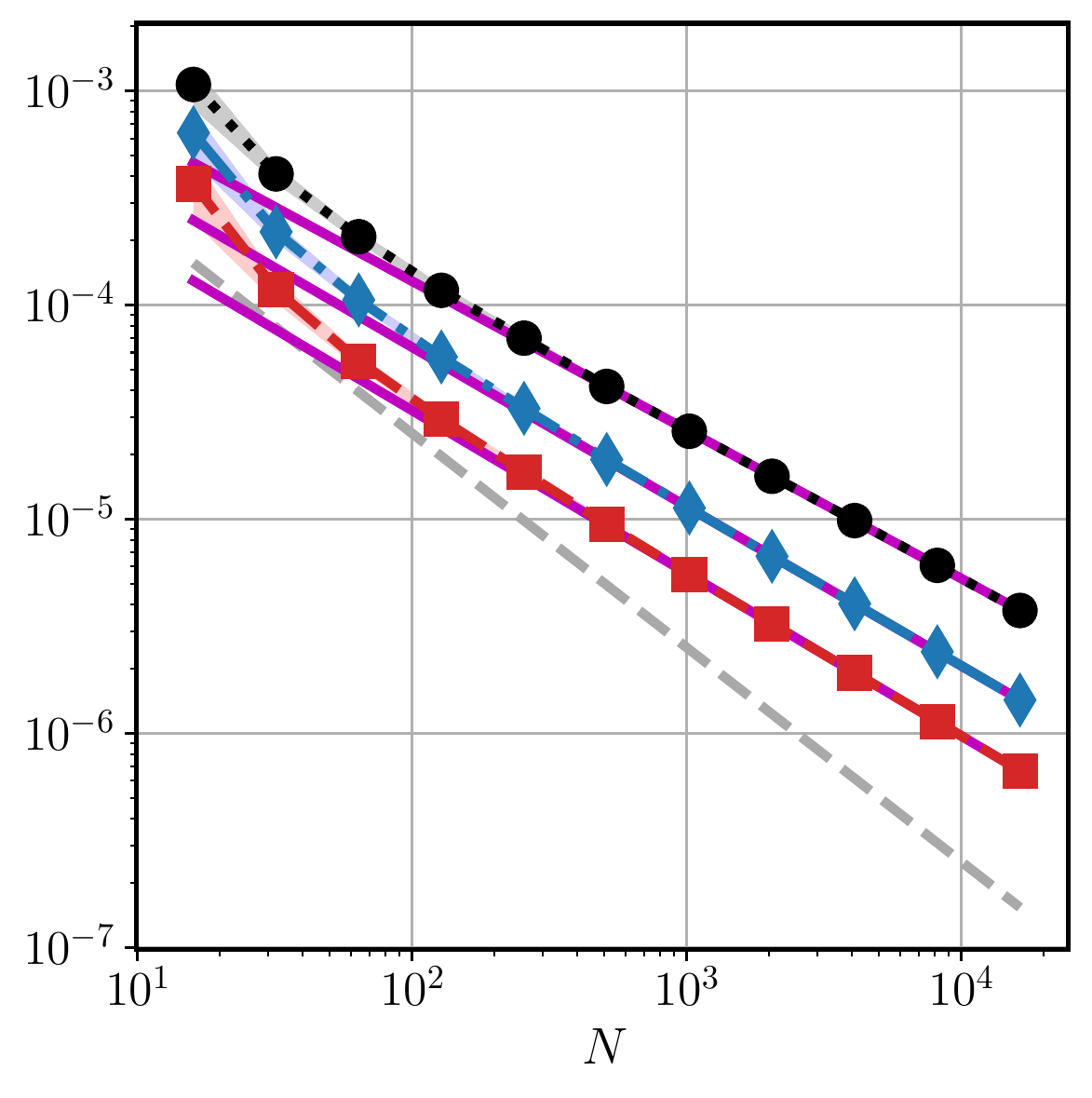}\label{fig:match_L2_alldistr_ot}}\hfill%
	\subfloat[$ \Ld=A_a^{-1}$, $ \gamma=10^{-5} $]{	\includegraphics[width=0.225\textwidth]{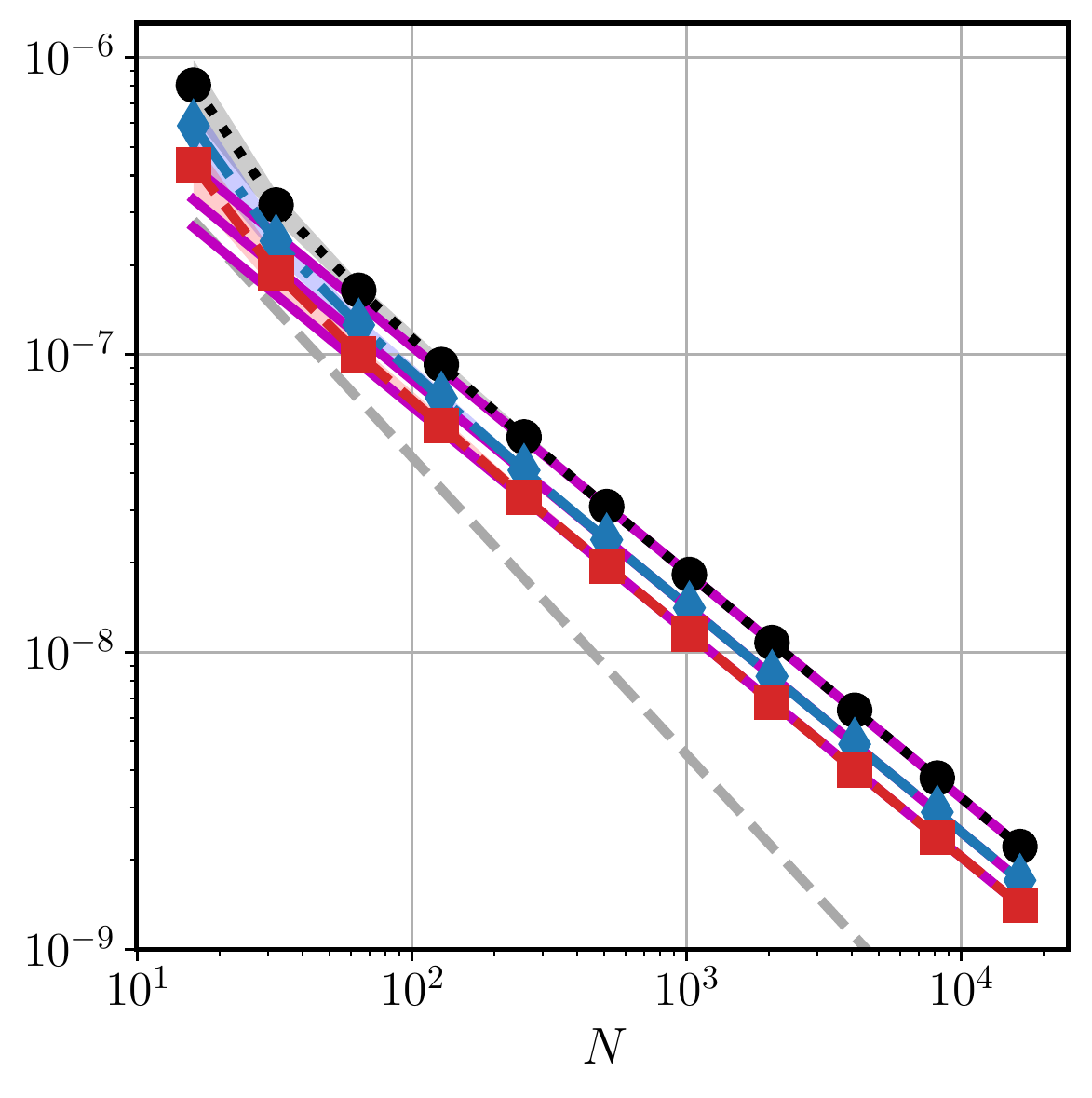}\label{fig:match_L3_alldistr_ot}}
	\caption{Beyond the theory. Analogous to \cref{fig:match_alldistr} except with the non-diagonal elliptic operator $ A_a $.}
	\label{fig:match_alldistr_ot}
\end{figure}

\section{Conclusion}\label{sec:conclusion}
This paper concerns the supervised learning of linear operators between Hilbert spaces. Learning is framed as a Bayesian inverse problem with a linear operator as the unknown quantity. Working in the best-case scenario of known eigenvectors, the analysis establishes convergence rates in the infinite data limit. The main results reveal useful theoretical insights about operator learning, including what types of operators are harder to learn than others, what types of training data lead to reduced sample complexity, and how distribution shift affects error. The work opens up the following directions for future research.

\paragraph*{Extensions in diagonal setting}
One immediate extension of our diagonal approach involves generalizing it from self-adjoint operators with known eigenvectors to non-self-adjoint operators with known singular vectors. Another involves taking the simultaneous large data and small noise limit. Although our approach requires Gaussian conjugacy, Gaussian priors are not suitable for all problems. Recent work using non-conjugate priors may prove useful in our setting \cite{gugushvili2020bayesian,knapik2018general,ray2013bayesian,trabs2018bayesian}. To exploit the Bayesian posterior beyond just theoretical contraction performance, exploration of uncertainty quantification via credible sets is also of interest.

\paragraph{Beyond diagonal operators}\label{par:conclusion_beyond}
In the linear setting, it is desirable to remove the known eigenbasis assumption but retain rates of convergence. The proof of \cref{fact:forward_facts} in \cref{app:extra} implies that the SVD of the random forward map \( K_X \) in \cref{eqn:intro_ideas_ip_operator} is determined by functional PCA of \( X \). Thus, the SVD approach in \cref{sec:intro_ideas_compare} and \cite{knapik2011bayesian} could be used to recover the doubly-indexed infinite matrix coordinates of the true operator in the (random) SVD basis. Another approach is to directly study the non-diagonal problem \cref{eqn:intro_ideas_ip_matrix} as in \cref{sec:numerics_beyond}. Nonlinear operators also deserve attention, as the experimental results in \cite{de2022cost} demonstrate.  Central to their statistical analysis will be the modern architectures (beyond kernel methods \cite{caponnetto2007optimal,rastogi2020convergence}) that parametrize the unknown operators and their inherent problem-dependent structure.

\appendix
\section{Proofs of main results}\label{app:proofs}
In this appendix, we provide proofs of the theorems from the main body of the paper, in order of appearance. We begin with \cref{thm:intro_ideas_thm}.

\begin{proof}[Proof of \cref{thm:intro_ideas_thm}]
	\Cref{thm:intro_ideas_thm} is a special case of \cref{thm:upper-expect} in the case \( \al'<\al+1/2 \) in \( \rho_N(\al,\al',p) \) \cref{eqn:theory_expectation_rho_J}. It remains to show that the Gaussian measure \( \nu=\normal(0,\Lambda) \) satisfies \cref{assump:theory_expectation_train}. The KL expansion coefficients certainly satisfy the fourth moment condition. The final condition on \( \{\avgn{g_j}{g_j}\} \) is verified by \cref{lem:chisq} because \( \{g_{jn}\}_{n=1}^N\sim \normal(0,\vartheta_j^2)^{\otimes N} \).
\end{proof}

\subsection[Proofs for subsection 3.2]{Proofs for \cref{sec:theory_expectation}}\label{app:proofs_expectation}
Under \cref{assump:theory_assumptions_main,assump:theory_expectation_train}, we calculate from \cref{eqn:posterior_sequence,eqn:bochner_coord} that $ \E^{Y\condbar X}\E^{L^{(N)}\sim\post}\norm{\Ld -L^{(N)}}^{2}_{L_{\nu'}^{2}(H;H)}=\cI_1+\cI_2+\cI_3 $ for \( N\in\N \), where
\begin{subequations}\label{eqn:proofs_iall}
	\begin{align}
	\cI_1=\sum_{j=1}^{\infty}\dfrac{\vartheta_j'^{2}\abs{\ld_j}^{2}}{(1+N\gamma^{-2}\sigma_j^2\avgn{g_j}{g_j})^2}&\, ,\quad 
	\cI_2=\sum_{j=1}^{\infty}\dfrac{N\vartheta_j'^{2}\gamma^{-2}\sigma_j^{4}\avgn{g_j}{g_j}}{(1+N\gamma^{-2}\sigma_j^2\avgn{g_j}{g_j})^2} \label{eqn:i12} \, ,\qa\\
	\cI_3=\sum_{j=1}^{\infty}&\dfrac{\vartheta_j'^{2}\sigma_j^2}{1+N\gamma^{-2}\sigma_j^2 \avgn{g_j}{g_j}} \label{eqn:i3}\, .
	\end{align}
\end{subequations}
This is the test error averaged only over the posterior and noise distributions, keeping the random design \( X \) fixed. The posterior mean test error \cref{eqn:error_bochner_intro2} is given by \( \E[\cI_1+\cI_2] \) only. Recall from \cref{assump:theory_assumptions_main} that $\vartheta_j'^2$ \cref{eqn:theory_assumptions_test} decays as $j^{-2\alpha'}$ (determining the test distribution $\nu'$) and
$\sigma_j^2$ \cref{eqn:theory_assumptions_prior} decays (or grows) as $j^{-2p}$ (determining the prior on $L$). The three series depend on $X=\{x_n\}$ through the correlated r.v.s $\{g_{jn}=\ip{\varphi_j}{x_n}\}$ \cref{eqn:setup_bayes_gjn}. These are mean zero with variance \( \vartheta_j^2 \) \cref{eqn:theory_assumptions_train} decaying as \( j^{-2\al} \). The truth is $\ld\in\cH^s $, as in \cref{item:as_smooth_truth}. All of the following proofs involve estimating the three random series \cref{eqn:proofs_iall}, which converge $ \P $-a.s. by \cref{item:as_exponent} in \cref{assump:theory_assumptions_main} (and by \cref{lem:as_series} for \( \cI_2 \)).
For convenience, we set \( u\defeq 2(\al+p)>1 \) and write $ \avgn{g_j}{g_j} \eqdef \vartheta_j^2 Z_j^{(N)}$. Thus, \( \E Z_j^{(N)} = 1 \). We also set \( \gamma\equiv 1 \) without loss of generality.

\begin{proof}[Proof of \cref{thm:upper-expect}]	
	We split each of the three series \cref{eqn:proofs_iall} into sums over two disjoint index sets $ \{j\in\N\colon  j\leq N^{1/u}\} $ and $ \{j\in\N\colon  j>N^{1/u}\} $. We denote such sums by $ \cI_i^{\leq} $ and $ \cI_i^{>} $, respectively, for each $ i\in\{1,2,3\} $. We must estimate their expectations over \( X\sim\nu^{\otimes N} \) to prove the assertion \cref{eqn:upper-expect1}. Notice that $ Nj^{-u}\simeq 1+Nj^{-u} $ whenever $ j\leq N^{1/u} $.
	
	%%%% I_2		
	Beginning with $\E \cI_2$, its partial sum $ \E \cI_2^{\leq} $ satisfies
	\[
	\E \sum_{j\leq N^{1/u}}\dfrac{N\vartheta_j'^{2}\sigma_j^{4}\avgn{g_j}{g_j}}{(1+N\sigma_j^2\avgn{g_j}{g_j})^2}
	\leq \sum_{j \leq N^{1/u}}\dfrac{\vartheta_j'^{2}\E\bigl[ (\avgn{g_j}{g_j})^{-1} \bigr]}{N} 
	\lesssim \sum_{j \leq N^{1/u}} \dfrac{\vartheta_j'^2\vartheta_j^{-2}}{N}\asymp \sum_{j \leq N^{1/u}}\frac{j^{-2(\al'+p)}}{1+Nj^{-u}}
	\]
	as \( N\to\infty \). We used \cref{assump:theory_expectation_train} and Lyapunov's inequality to bound the negative moment. By applying \cref{eqn:knapik_2p2} in \Cref{lem:knapik_2} (with $ t=2(\al'+p) $, $ v=1 $, and condition $ t>1 $ satisfied by \cref{item:as_exponent}) to the last sum, we deduce that $ \E \cI_2^{\leq} =O(\rho_N)$. The tail series satisfies
	\[
	\textstyle
	\E\cI_2^{>}\leq \sum_{j> N^{1/u}}N\vartheta_j'^{2}\sigma_j^{4}\E \bigl[\vartheta_j^2 Z_j^{(N)}]
	\asymp 
	N \sum_{j> N^{1/u}}j^{-2(\al'+\al+2p)}
	\asymp N^{-(1-{(\al +1/2 -\al')}/{(\al+p)})}
	\]
	as $ N\to\infty $ by \cref{eqn:knapik_2p1} in \Cref{lem:knapik_2} (applied with $ t=2(\al'+\al+2p)>1 $ by \cref{item:as_exponent}). This is always the same order as, or negligible compared to, the upper bound on $ \E\cI_2^{\leq} $.
	
	%%%% I_3
	By the same argument used for \( \E\cI_2^{\leq} \) (bounding its denominator by one and using \cref{assump:theory_expectation_train} plus Lyapunov's inequality), we deduce that $ \E\cI_3^{\leq} = O(\rho_N) $ also. The tail \( \E\cI_3^{>} \) is bounded above by \( \sum_{j> N^{1/u}}\vartheta_j'^{2}\sigma_j^2\asymp \sum_{j> N^{1/u}}j^{-2(\al'+p)} \). This sum is the same order as the bound on \( \E\cI_2^{>} \) by \cref{eqn:knapik_2p1} in \Cref{lem:knapik_2} (with $ t=2(\al'+p)>1 $ by \cref{item:as_exponent}).
	
	%%%% I_1
	Last, again by \cref{assump:theory_expectation_train} and Lyapunov's inequality, $ \E\cI_1^{\leq} $ is bounded above by
	\begin{equation}\label{eqn:proofs_expectation_i1upper}
	\sum\nolimits_{j\leq N^{\frac{1}{u}}}\dfrac{\vartheta_j'^2\abs{\ld_j}^2\E\bigl[(\avgn{g_j}{g_j})^{-2}\bigr]}{(N\sigma_j^2)^2}
	\lesssim
	\sum\nolimits_{j\leq N^{\frac{1}{u}}}\dfrac{\vartheta_j'^{2}\abs{\ld_j}^2(\vartheta_j^2)^{-2}}{(N\sigma_j^2)^2}
	\asymp
	\sum\nolimits_{j\leq N^{\frac{1}{u}}}\frac{j^{-2\al'}\abs{\ld_j}^2}{(1+Nj^{-u})^2}
	\end{equation}
	as \( N\to\infty \). Application of \cref{eqn:knapik_1p2} in \Cref{lem:knapik_1} (with \( \xi=\ld \), $ t=2\al' $, $ q=s $, $ v=2 $, and $ t\geq-2q $ satisfied by \cref{item:as_exponent}) shows that this last sum is \( o(N^{-(\al'+s)/(\al+p)}) \) if \( (\al'+s)/(\al+p)<2 \) or \( \Theta(N^{-2}) \) otherwise. The tail sum matches this bound in the first case and is strictly smaller otherwise because \( \E \cI_1^{>} \leq \sum_{j>N^{1/u}} \vartheta_j'^2\abs{\ld_j}^2\asymp\sum_{j>N^{1/u}}j^{-2\al'}\abs{\ld_j}^2 \) (apply \cref{eqn:knapik_1p12} in \Cref{lem:knapik_1} with \( \xi=\ld \), $ t=2\al' $ and $ q=s $). All together, we deduce that $ \E\cI_2 $ and $ \E\cI_3 $ have the same upper bound \( \rho_N\gg N^{-2} \). This implies \cref{eqn:upper-expect1}. The uniform bound over \( \norm{\ld}_{\cH^s}\lesssim 1 \) follows from the first assertion in \cite[Lem. 8.1]{knapik2011bayesian} (this turns the little-\( o \) into a big-\( O \) as claimed). The final assertion follows because the posterior mean test error only corresponds to $ \cI_1 $ and $ \cI_2 $.
\end{proof}

\begin{proof}[Proof of \cref{thm:lower-expect}]
	The proof proceeds by developing lower bounds on each of the three series \cref{eqn:proofs_iall}, using the same disjoint index sets approach in the proof of \cref{thm:upper-expect}. For $ \E \cI_3 $, since $ r\mapsto (1+ar)^{-1} $ is convex on $ [0,\infty) $ for all $ a\geq 0 $, Jensen's inequality yields
	\begin{equation}\label{eqn:proof_I3_lower}
	\E\sum_{j=1}^{\infty} \dfrac{\vartheta_j'^{2}\sigma_j^2}{1+N\sigma_j^2\avgn{g_j}{g_j}}
	\geq \sum_{j=1}^{\infty} \dfrac{\vartheta_j'^{2}\sigma_j^2}{1+N\sigma_j^2\E[\vartheta_j^2 Z_j^{(N)} ]} 
	\asymp \sum_{j=1}^{\infty}\dfrac{j^{-2(\al'+p)}}{1+Nj^{-u}} \qas N\to\infty\, .
	\end{equation}
	The last sum is \( \Theta(\rho_N) \) by \cref{eqn:knapik_2p2} in \cref{lem:knapik_2} (with $ t=2(\al'+p) >1$ by \ref{item:as_exponent} and $ v=1 $).
	
	Next, \( \E\cI_2 \geq \E\cI_2^{\leq} \) by nonnegativity. For any positive $ \tau_N\to 0 $, define the events \( A_j^{(N)}\defeq \bigl\{\omega\in\varOmega\colon  Z_j^{(N)}(\omega)\geq \tau_N\bigr\} \) for every \( j\) and \( N\in\N\). The law of total expectation yields
	\begin{align*}
	\E\cI_2^{\leq}
	=&\sum\nolimits_{j\leq N^{1/u}} N \vartheta_j'^2\sigma_j^4\vartheta_j^2\E\Biggl[\dfrac{Z_j^{(N)}}{(1+N\sigma_j^2\vartheta_j^2 Z_j^{(N)})^2} \lllcondbar A_j^{(N)}\Biggr]\P\bigl(A_j^{(N)}\bigr) \\
	& + \sum\nolimits_{j\leq N^{1/u}}N\vartheta_j'^2\sigma_j^4\vartheta_j^2\E\Biggl[\dfrac{Z_j^{(N)}}{(1+N\sigma_j^2\vartheta_j^2 Z_j^{(N)})^2}\lllcondbar \bigl(A_j^{(N)}\bigr)^{\comp}\Biggr]\P\bigl(A_j^{(N)}\bigr)^{\comp}\, .
	\end{align*}
	The second term in the above display is nonnegative, so we obtain
	\begin{align*}
	\E\cI_2^{\leq}&\geq \sum\nolimits_{j\leq N^{1/u}}\tau_N N\vartheta_j'^2\sigma_j^4\vartheta_j^2\E\bigl[(1+N\sigma_j^2\vartheta_j^2 Z_j^{(N)})^{-2}\lcondbar A_j^{(N)}\bigr]\P\bigl(A_j^{(N)}\bigr)\\
	&\geq \sum_{j\leq N^{1/u}}\dfrac{\tau_N N\vartheta_j'^2\sigma_j^4\vartheta_j^2\P\bigl(A_j^{(N)}\bigr)}{\bigl(1+N\sigma_j^2\vartheta_j^2\E\bigl[Z_j^{(N)}\lcondbar A_j^{(N)}\bigr]\bigr)^2}
	= \sum_{j\leq N^{1/u}}\dfrac{\tau_N N\vartheta_j'^2\sigma_j^4\vartheta_j^2\P\bigl(A_j^{(N)}\bigr)^3}{\bigl(\P\bigl(A_j^{(N)}\bigr)+N\sigma_j^2\vartheta_j^2\E\bigl[\one_{A_j^{(N)}}Z_j^{(N)}\bigr]\bigr)^2}\,.
	\end{align*}
	We applied conditional Jensen's inequality to yield the second inequality because $ r\mapsto (1+ar)^{-2} $ is convex on $[0,\infty) $ for any $ a\geq 0$. Next, Markov's inequality plus \cref{assump:theory_expectation_train} gives
	\[
	\textstyle\sup_{j\geq 1}\P\bigl(A_j^{(N)}\bigr)^{\comp}=\sup_{j\geq 1}\P\bigl\{(Z_j^{(N)})^{-1}>\tau_{N}^{-1}\bigr\}\leq \sup_{j\geq 1}\tau_N\E\bigl[(Z_j^{(N)})^{-1}\bigr]\to 0\qas N\to\infty\, .
	\]
	This implies $ \inf_{j\geq 1}\P(A_j^{(N)})\to 1 $ as \( N\to\infty \). 
	Using this and the facts $ \E[\one_{A}Z_j^{(N)}] \leq \E[Z_j^{(N)}]=1$ and $ \P(A)\leq 1 $ for any $ A\in\cF $ and applying $ 1+Nj^{-u}\simeq Nj^{-u} $ for $ j\leq N^{1/u} $ twice yields
	\begin{equation*}
		\E\cI_2^{\leq}\gtrsim \sum\nolimits_{j\leq N^{1/u}}\dfrac{\tau_N N\vartheta_j'^2\sigma_j^4\vartheta_j^2\P\bigl(A_j^{(N)}\bigr)^3}{(1+N\sigma_j^2\vartheta_j^2)^2}
		\gtrsim \tau_N\sum\nolimits_{j\leq N^{1/u}}\dfrac{ j^{-2(\al'+p)}}{1+Nj^{-u}} \qas N\to \infty\,.
	\end{equation*}
	Comparing to \cref{eqn:proof_I3_lower}, we deduce \( \E\cI_2=\Omega(\tau_N\rho_N) \). This is negligible relative to $ \E\cI_3=\Omega(\rho_N) $.
	
	Last, by Jensen's inequality, we lower bound \( \E\cI_1 \) by the rightmost sum in \cref{eqn:proofs_expectation_i1upper}, which is always $ \Omega(N^{-2}) $ (if $ \ld\neq 0 $), plus the tail of the same sum, which gives the second term in \cref{eqn:lower-expect} by \cref{eqn:knapik_1p12} in \cref{lem:knapik_1} (with \(\xi=\ld, t=2\al',q=s \), and \( v=2 \)). The $ \Omega(N^{-2}) $ contribution from the first sum is dominated by both \( \rho_N =\Omega(N^{-1})\) and \( \tau_N\rho_N \) if \( \tau_N\gg N^{-1} \). Therefore, the posterior sample test error \cref{eqn:error_bochner_intro} enjoys the asserted rate, while the posterior mean test error $ \E[\cI_1+\cI_2] $ only admits the bound \cref{eqn:lower-expect} with the \( \tau_N \) factor as claimed.
\end{proof}

\begin{proof}[Proof of \cref{thm:theory_expectation_rvsharp}]
	The \( \rho_N \) term (corresponding to \( \cI_2 \) and \( \cI_3 \) in  \cref{eqn:proofs_iall}) in the assertion \cref{eqn:theory_expectation_rvsharp} follows from \cref{thm:upper-expect,thm:lower-expect} for the posterior sample estimator. It remains to obtain the second term in \cref{eqn:theory_expectation_rvsharp}. Following the argument from the proof of \cref{thm:upper-expect}, \( \E\cI_1^{\leq} \) is asymptotically bounded above by the last sum in \cref{eqn:proofs_expectation_i1upper}. Now given \( \abs{\ld_j}\asymp j^{-1/2-s}S(j) \), by the full version of \cite[Lem. 8.2]{knapik2011bayesian} (applied with \( \xi=\ld \), \( t=-2\al' \), \( v=2 \), \( q=s \), \( \cS=S \), and \( t>-2q \) by \cref{item:as_exponent}), this sum has exact order the second term in \cref{eqn:theory_expectation_rvsharp} if \( (\al'+s)/(\al+p)< 2 \) and is negligible relative to \( \rho_N \) otherwise. The tail \( \E\cI_1^{>} \) is always bounded above by the second term in \cref{eqn:theory_expectation_rvsharp} by the proof of \cite[Lem. 8.2]{knapik2011bayesian}. After an application of Jensen's inequality, the argument leading to a matching lower bound for \( \E\cI_1 \) is the same as the one above.
\end{proof}

\subsection[Proof for subsection 3.4]{Proof for \cref{sec:theory_probability}}\label{app:proofs_probability}
We follow \cref{app:proofs_expectation} by letting \( u\defeq 2(\al+p)>1 \), $ \avgn{g_j}{g_j} \eqdef \vartheta_j^2 Z_j^{(N)}$, and \( \gamma\equiv 1 \), but instead of \cref{assump:theory_expectation_train} we now enforce \cref{assump:theory_probability_train}, which defines our \( \Lambda \)-subgaussian data. This yields \( g_{jn}=\ip{\varphi_j}{x_n}\in\SG{\sigma_{\nu}^2\vartheta_j^2} \).
Henceforth, let $ \SE{v^2}{a} $ denote the set of real subexponential (SE) r.v.s with parameters $ (v,a)\in \R_{\geq 0}^2$. The inclusion $ X\in \SE{v^2}{a} $ is characterized by the moment generating function (MGF) bound $ \E\exp(\theta (X-\E X))\leq \exp(v^2\theta^2/2) $ for all $ \abs{\theta}<1/a $ \cite{wainwright2019high}. Using \cite[Lem. 2.7.6]{vershynin2018high} gives \( g_{jn}^2/\vartheta_j^2\in\SE{c^2\sigma_{\nu}^4}{c\sigma_{\nu}^2} \) for an absolute constant \( c>0 \). By independence, \( Z_j^{(N)}\in\SE{c^2\sigma_{\nu}^4/N}{c\sigma_{\nu}^2/N} \) \cite[sect. 2.1.3]{wainwright2019high}. The following proof relies on SE concentration from \cref{app:lemmas}.

\begin{proof}[Proof of \cref{thm:upper-lower-prob}]
	We prove the upper and lower concentration bounds separately.
	\paragraph*{Upper bound}
	Fix $ \delta\in(0,1\mmin c\sigma_{\nu}^2) $ and define $ \Ndm\defeq (1-\delta)N $. We follow the disjoint index sets approach from \cref{thm:upper-expect}, except now we sum over $ \{j\in\N\colon  j\leq \Ndm^{1/u}\} $ and $ \{j\in\N\colon  j>\Ndm^{1/u}\} $. Denote these sums by \( \cI_i^{\leq,\delta} \) and \( \cI_i^{>,\delta} \), respectively \cref{eqn:proofs_iall}. We first bound
	\[
	\cI_1^{\leq,\delta}
	\leq 
	\sum\nolimits_{j\leq \Ndm^{1/u}}\dfrac{\vartheta_j'^2\abs{\ld_j}^2}{(1+\Ndm\sigma_j^2\vartheta_j^2)^2}
	\asymp
	\sum\nolimits_{j\leq \Ndm^{1/u}}\dfrac{j^{-2\al'}\abs{\ld_j}^2}{(1+\Ndm j^{-u})^2}\qas N\to\infty
	\]
	with probability (w.p.) at least $ 1-\Ndm^{1/u}\exp(-N\delta^2/(2c^2\sigma_{\nu}^4)) $ by \cref{lem:subexp_bound_dep} (with \( n=N\), $X_j^{(N)}=Z_j^{(N)}$, $v=a=c\sigma_{\nu}^2$, $J=\floor{\Ndm^{1/u}}$, and the lower tail only). The remaining bounds for $ \cI_1 $ (including the almost sure bound for $ \cI_1^{>,\delta} $) are the same as those in the proof of \cref{thm:upper-expect}, except with \( N \) replaced by \( \Ndm \). This gives the second term in \cref{eqn:upper-prob1}.
	
	Following the arguments in the proof of \cref{thm:upper-expect} for \( \E\cI_2^{\leq} \) and by a similar application of \cref{lem:subexp_bound_dep}, we deduce that $ \cI_2^{\leq,\delta} = O(\rho_{\Ndm}) $ w.p. at least $ 1-\Ndm^{1/u}\exp(-N\delta^2/(2c^2\sigma_{\nu}^4)) $. For the infinite tail series \( \cI_2^{>,\delta} \), bounding its denominator by one yields
	\[
	\textstyle
	\cI_2^{>,\delta}\leq \sum\nolimits_{j> \Ndm^{1/u}}N\vartheta_j'^2\vartheta_j^2\sigma_j^4Z_j^{(N)}
	\lesssim
	N(1+\delta)\sum\nolimits_{j> \Ndm^{1/u}}j^{-2(\al'+\al+2p)}\qas N\to\infty
	\]
	w.p. at least $ 1-\exp(-N\delta^2/(2c^2\sigma_{\nu}^4)) $. The second inequality is from \cref{lem:fullseriesbound_dep} (with $n=N$, $X_j^{(N)}=Z_j^{(N)}$, $v=a=c\sigma_{\nu}^2$, $\{w_j=\vartheta_j'^2\vartheta_j^2\sigma_j^4\}$, and the upper tail only), where $ \{\vartheta_j'^2\vartheta_j^2\sigma_j^4\} $ is in \( \ell^1 \) because $ \al'+\al+2p>1 $ by \cref{item:as_exponent}. We deduce \( \cI_2^{>,\delta}=O((1+\delta)/(1-\delta)\Ndm^{-(1-{(\al +1/2 -\al')}/{(\al+p)})}) \) by the same argument used for \( \E\cI_2^{>} \) in the proof of \cref{thm:upper-expect}.
	
	Along similar lines as the proof of \cref{thm:upper-expect}, the posterior covariance term $ \cI_3^{\leq,\delta} $ has the same order as $ \cI_2^{\leq,\delta} $ with the same probability (by \cref{lem:subexp_bound_dep}). The tail $ \cI_3^{>,\delta} $ is bounded above a.s. by the first case in \( \rho_{\Ndm} \) \cref{eqn:theory_expectation_rho_J} as \( N\to\infty \). Since $ a+b(1+\delta)/(1-\delta)\lesssim (1+\delta)/(1-\delta)$ for any $ a,b>0 $, we deduce \( \cI_2+\cI_3 \) has order the first term in \cref{eqn:upper-prob1}. The asserted total probability follows by combining the individual event probabilities with the union bound and the fact that there exists $ c_2(\delta)>0 $ and $ 0<c_3<c'\defeq1/(2c^2\sigma_{\nu}^4) $ such that $ \sup_{n\gtrsim 1}n^{1/u}\exp(-(c'-c_3)n\delta^2)<c_2(\delta) $. The assertion about the upper bound for \( \bar{L}^{(N)} \) follows by ignoring $ \cI_3 $.
	
	\paragraph*{Lower bound}
	Since $ 1+\delta\in(1,2) $ is bounded, we do not track this factor in what follows. The proof proceeds by splitting all series at the critical index $ J_N=\floor{N^{1/u}} $ (since $ (1+\delta)N\simeq N $) as in \cref{thm:upper-expect}. By nonnegativity, we lower bound the error \cref{eqn:proofs_iall} by $ \cI_1^{>}+\cI_2^{\leq} + \cI_3^{\leq} $. The tail term \( \cI_1^{>} \) is bounded below by the second term in \cref{eqn:lower-prob} with high probability by \cref{lem:subexp_bound_dep} and \cref{eqn:knapik_1p12} in \cref{lem:knapik_1}.
	The remaining calculations showing that \( \cI_2^{\leq} \) and \( \cI_3^{\leq} \) are \( \Omega(\rho_N) \) with high probability follow directly from \cref{lem:subexp_bound_dep} and \cref{eqn:knapik_2p2} in \cref{lem:knapik_2} and are omitted. For \( \bar{L}^{(N)} \), the only variance contribution is from $ \cI_2^{\leq} $. Its lower bound has the small pre-factor $ 1-\delta $ as asserted. Combining the individual event probabilities as was done for the upper bound completes the proof of \cref{thm:upper-lower-prob}.
\end{proof}

\subsection[Proof for subsection 3.6]{Proof for \cref{sec:theory_gen_gap}}\label{app:proofs_gen_gap}
This subsection proves \cref{thm:gen_gap_expect} by bounding the generalization gap \( \cG_N\) \cref{eqn:gen_gap}, which only involves in-distribution notions of error. We work in the setting of \cref{app:proofs_expectation}, letting \( u\defeq 2(\al+p)>1 \) and \( \gamma\equiv 1 \) and enforcing \cref{assump:theory_assumptions_main,assump:theory_expectation_train}. Then, the \( L^1_{\P}(\varOmega;\R) \) norm of \( \cG_N \) satisfies $ \EGG=\E^{D_N}\abs{\cJ_1+\cJ_2+\cJ_3} $, where
\begin{equation}\label{eqn:j1j2j3}
	\textstyle
	\frac{1}{2}\sum_{j=1}^{\infty}(\vartheta_j^2-\avgn{g_j}{g_j})\abs[\big]{\bar{l}_{j}^{(N)}-\ld_j}^2,\quad
	\frac{1}{2}\sum_{j=1}^{\infty}(\avgn{g_j}{g_j}-\vartheta_j^2)\abs{\ld_j}^2,\qa
	\sum_{j=1}^{\infty}\avgn{g_j}{\xi_j}\bar{l}_j^{(N)}
\end{equation}
define \( \cJ_1 \), \( \cJ_2 \), and \( \cJ_3 \), respectively. In \cref{eqn:j1j2j3}, the r.v.s \( \{\xi_{jn}\} \) from \cref{eqn:intro_ideas_ip_diagonal} are i.i.d. $\normal(0,1)$.
Using the explicit form \cref{eqn:posterior_sequence} of the posterior mean $ \{\bar{l}_{j}^{(N)}\} $, we find that $ \EGG $ equals
\begin{equation}\label{eqn:gen_gap_expect}
\E\,\abs[\Bigg]{\dfrac{1}{2}\sum_{j=1}^{\infty}(\vartheta_j^2-\avgn{g_j}{g_j})\dfrac{\abs{\ld_j}^2+N\sigma_j^4\avgn{g_j}{g_j}}{(1+N\sigma_j^2\avgn{g_j}{g_j})^2}
	+ \cJ_2
	+ \sum_{j=1}^{\infty}\dfrac{(\avgn{g_j}{\xi_j})^2 +\ld_j\avgn{g_j}{g_j}\avgn{g_j}{\xi_j}}{N^{-1}\sigma_j^{-2}+\avgn{g_j}{g_j}} }\, .
\end{equation}
The following proof and \cref{lem:as_series} imply the convergence of \cref{eqn:j1j2j3} $ \P $-a.s. and \cref{eqn:gen_gap_expect}.

\begin{proof}[Proof of \cref{thm:gen_gap_expect}]
	We prove the upper and lower bounds on \( \EGG \) separately.
	\paragraph*{Upper bound}
	By the triangle inequality, \cref{eqn:gen_gap_expect} is bounded above by five terms $ G_i $ for $ i\in\{1,\ldots,5\} $. Here, $ \{G_1, G_2\} $ corresponds to $ \E\abs{\cJ_1} $, $ G_3 $ to $ \E\abs{\cJ_2} $, and $ \{G_4,G_5\} $ to $ \E\abs{\cJ_3} $. 
	
	By triangle and Jensen's inequality, $ G_3=\E\abs{\cJ_2} \leq \tfrac{1}{2}\sum_{j=1}^{\infty}\abs{\ld_j}^2(\Var[\avgn{g_j}{g_j}])^{1/2}$. Independence of \( \{x_n\} \) yields \( \Var[\avgn{g_j}{g_j}] \leq \frac{1}{N}\E^{x\sim\nu}\ip{\varphi_j}{x}^4\). Using \cref{eqn:setup_bayes_gjn,assump:theory_expectation_train} (\( \{\zeta_j\} \) are zero mean, unit variance, and independent), \( \E^{x\sim\nu}\ip{\varphi_j}{x}^4\simeq \sum_k c_{jk}^4\E\zeta_k^4+\sum_{k'\neq k}c_{jk}^2c_{jk'}^2 \), where \( c_{jk}\defeq\ip{\Lambda^{1/2}\varphi_j}{\phi_k} \). The second term is bounded above by a constant times \( (\sum_kc_{jk}^2)^2=\vartheta_j^4 \) and so is the first term (using \( \limsup_{j\to\infty}\E\zeta_j^4<\infty \) and \( \ell^2\subset \ell^4 \)). Thus, \( G_3 \lesssim \norm{\Ld}_{L^2_{\nu}(H;H)}^2\, N^{-1/2}\).
	
	Using the disjoint index sets approach from the proof of \cref{thm:upper-expect}, $G_1^{\leq}$ is bounded above by \( \frac{1}{2}\sum_{j\leq N^{1/u}}(N\sigma_j^2)^{-2}\abs{\ld_j}^2\E[\abs{\vartheta_j^2-\avgn{g_j}{g_j}}(\avgn{g_j}{g_j})^{-2}] \). By the Cauchy--Schwarz inequality and \cref{assump:theory_expectation_train}, the expectation on the right is bounded above by \( (\Var[\avgn{g_j}{g_j}])^{1/2}\vartheta_j^{-4} \) for sufficiently large \( N \). It follows that \( G_1^{\leq} \) is of the order \( N^{-1/2} \) times the rightmost sum in \cref{eqn:proofs_expectation_i1upper} with \( \al'=\al \), which all together is \( o(N^{-1/2}) \). This contribution is negligible relative to \( G_3 \). A similar argument shows that the tail sum $ G_1^{>} $ is never bigger than $ G_1^{\leq} $.
	
	The other term associated with $ \cJ_1 $, which is \( G_2 \), satisfies $ G_2^{\leq}\leq \frac{1}{2}\sum_{j\leq N^{1/u}}N^{-1}\E[\abs{\vartheta_j^2-\avgn{g_j}{g_j}}(\avgn{g_j}{g_j})^{-1}] =O(N^{-1/2}N^{-(1-1/u)})=o(N^{-1/2})$ (since \( u>1 \)) by an argument similar to the one used for \( G_1 \). The Cauchy--Schwarz inequality and the variance bound used for \( G_1 \) yields \( G_2^{>} \leq \frac{1}{2}\sum_{j> N^{1/u}}N\sigma_j^4\E[\abs{\vartheta_j^2-\avgn{g_j}{g_j}}\,\avgn{g_j}{g_j}]\lesssim N^{-1/2}\sum_{j>N^{1/u}}N\sigma_j^4\vartheta_j^4\). The last sum is asymptotic to \( N^{-1/2}\sum_{j>N^{1/u}}Nj^{-2u}\) as \( N\to\infty \), which is the same order as \( G_2^{\leq} \) by \cref{eqn:knapik_2p1} in \cref{lem:knapik_2} (with $ t=2u>1 $). Thus, \( G_2 \) is also negligible relative to \( G_3 \).
	
	Moving on to $ G_4 $ from $ \E\abs{\cJ_3} $, we first average out the noise $ \{\xi_{jn}\} $ to obtain
	\[
	G_4=\E\sum_{j=1}^{\infty}\dfrac{(\avgn{g_j}{\xi_j})^2}{N^{-1}\sigma_j^{-2}+\avgn{g_j}{g_j}}
	= \E^{X}\sum_{j=1}^{\infty}\dfrac{N\sigma_j^2\E^{Y\condbar X}\bigl[(\avgn{g_j}{\xi_j})^2\bigr]}{1+N\sigma_j^2\avgn{g_j}{g_j}}
	=\E^{X}\sum_{j=1}^{\infty}\dfrac{\sigma_j^2\avgn{g_j}{g_j}}{1+N\sigma_j^2\avgn{g_j}{g_j}}\,.
	\]
	Since the map $ r\mapsto r(1+ar)^{-1} $ is concave on $[0,\infty) $ for all $ a\geq 0 $, Jensen's inequality yields \( G_4\lesssim \sum_{j=1}^{\infty}j^{-u}/(1+Nj^{-u})
	=
	O(N^{-(\al+p-1/2)/(\al+p)}) \) as \( N\to\infty \) by \cref{eqn:knapik_2p2} in \cref{lem:knapik_2} (with $ t=u>1 $ and $ v=1 $, satisfying the first case).
	
	Last, Jensen's inequality applied to the entire series \( G_5 \) from \( \E\abs{\cJ_3} \) yields
	\[
	G_5
	\leq
	\biggl(\E^{X}\E^{Y\condbar X}\abs[\Bigg]{\sum_{j=1}^{\infty}\dfrac{N\sigma_j^{2}\ld_j\avgn{g_j}{g_j}\avgn{g_j}{\xi_j}}{1+N\sigma_j^{2}\avgn{g_j}{g_j}} }^{2}\ \biggr)^{1/2}
	=
	\biggl(\E^X\sum_{j=1}^{\infty}\dfrac{N\abs{\ld_j}^2\sigma_j^4(\avgn{g_j}{g_j})^3}{(1+N\sigma_j^{2}\avgn{g_j}{g_j})^2}\biggr)^{1/2}
	\]
	because \( 	\E^{Y\condbar X}[(\avgn{g_j}{\xi_j})(\avgn{g_{j'}}{\xi_{j'}})]=\tfrac{1}{N^2}\sum_{n,n'\leq N}g_{jn}g_{j'n'}\E[\xi_{jn}\xi_{j'n'}]=\tfrac{\delta_{jj'}}{N}\bigl(\tfrac{1}{N}\sum_{n=1}^{N}g_{jn}g_{j'n}\bigr) \)
	for any $ j$ and $ j'\in\N $. Thus, \( G_5\leq (\sum_{j=1}^{\infty}N^{-1}\abs{\ld_j}^2\vartheta_j^2)^{1/2} = \norm{\Ld}_{L^2_{\nu}(H;H)}^2\, N^{-1/2}\). Comparing each \( \{G_i\}_{i=1,\ldots, 5} \), we conclude that $ \EGG=O(N^{-1/2} + G_4) $ as \( N\to\infty \) as asserted.

	\paragraph*{Lower bound}
	By the triangle inequality, 
	\(
	\EGG\geq \E\abs{\cJ_3+\cJ_2}-\E\abs{\cJ_1}\geq \abs{\E \cJ_3+\E \cJ_2}-\E\abs{\cJ_1} = \abs{\E \cJ_3}-\E\abs{\cJ_1}\, .
	\)
	We first develop a lower bound on $ \abs{\E \cJ_3} $, which equals \( G_4 \) by the zero mean property of the \( \{\xi_{jn}\} \). By an argument similar to the one used to lower bound \( \E\cI_2 \) in the proof of \cref{thm:lower-expect}, $ \abs{\E \cJ_3}\gtrsim\tau_N \sum_{j\leq N^{1/u}}j^{-u}/(1+Nj^{-u})
	=
	\Omega(\tau_NN^{-(1-1/u)}) $ as \( N\to\infty \) for any positive $ \tau_N\to 0 $. This is the asserted lower bound in \cref{eqn:thm_gen_gap-expect-lower}. To conclude the proof, we claim that the upper bounds previously developed for $ \E\abs{\cJ_1} $ (i.e., for \( G_1 \) and \( G_2 \)) are asymptotically negligible relative to \( \tau_NN^{-(1-1/u)} \) under the hypotheses. 
	Enforcing $ \tau_N \gg N^{-1/2} $ ensures that this is true for the \( G_2 \) bound. By \cref{eqn:proofs_expectation_i1upper}, if \( (\al+s)/(\al+p)\geq 2 \), then the \( G_1 \) contribution is \( N^{-1/2}N^{-2}\ll \tau_NN^{-(1-1/u)} \). Otherwise, \( G_1 \) is strictly smaller than \( N^{-1/2}N^{-(\al+s)/(\al+p)} \). This term is negligible relative to the \( \abs{\E \cJ_3} \) contribution if \( \tau_N\gg N^{-(1+\al+2s-p)/(2\al+2p)}\to 0 \), which requires $ p<1+\al+2s $ as assumed in the hypotheses.
\end{proof}

\section{Supporting lemmas}\label{app:lemmas}
Our first two results, which are variations of \cite[Lems. 8.1--8.2]{knapik2011bayesian}, develop sharp asymptotics for certain series that arise from \( \postseq \) in \cref{eqn:posterior_sequence}.
\begin{lemma}[series asymptotics: Sobolev regularity]\label{lem:knapik_1}
	Let $ q\in\R $, $ t\geq -2q $, $ u>0 $, and $ v\geq 0 $. Then for every $ \xi\in\cH^{q}(\N;\R) $, it holds that
	\begin{equation}\label{eqn:knapik_1p12}
	\sum\nolimits_{j> N^{1/u}}\dfrac{j^{-t}\xi_j^2}{\left(1+Nj^{-u}\right)^{v}}
	\simeq
	\sum\nolimits_{j>N^{1/u}}j^{-t}\xi_j^2
	\leq
	N^{-\left(\frac{t+2q}{u}\right)}
	\Bigl(\sum\nolimits_{j> N^{1/u}}j^{2q}\xi_j^2\Bigr)
	\end{equation}
	for all $ N\in\N $. Additionally, for every fixed $ \xi\in\cH^{q}(\N;\R) $, it holds that
	\begin{equation}\label{eqn:knapik_1p2}
	\sum_{j\leq N^{1/u}}\dfrac{j^{-t}\xi_j^2}{\left(1+Nj^{-u}\right)^{v}} = 
	\begin{cases}
	o\bigl(N^{-\left(\frac{t+2q}{u}\right)}\bigr)\, , & \text{if }\, (t+2q)/u<v\,,\\
	N^{-v}\,\norm{\xi}_{\cH^{(uv-t)/2}}^2\,\bigl(1+o(1)\bigr) \, , & \text{if }\,(t+2q)/u\geq v
	\end{cases}
	\end{equation}
	as $ N\to\infty $. The previous assertion \cref{eqn:knapik_1p2} remains valid for the full infinite series.
\end{lemma}
\begin{proof}
	The claims follow from \cite[Lem. 8.1, p. 2653]{knapik2011bayesian} and its proof therein.
\end{proof}

\begin{lemma}[series asymptotics: sharp]\label{lem:knapik_2}
	Let $ t>1 $, $ u>0 $, and $ v\geq 0 $. Then as $ N\to\infty $,
	\begin{subequations}
		\begin{align}
			\sum\nolimits_{j>N^{1/u}}\dfrac{j^{-t}}{\left(1+Nj^{-u}\right)^{v}}\simeq\sum\nolimits_{j>N^{1/u}}j^{-t} &=\Theta\bigl(N^{-\left(\frac{t-1}{u}\right)}\bigr)\qa \label{eqn:knapik_2p1}\\
			\sum_{j=1}^{\infty}\dfrac{j^{-t}}{\left(1+Nj^{-u}\right)^{v}}\asymp
			\sum_{j\leq N^{1/u}}\dfrac{j^{-t}}{\left(1+Nj^{-u}\right)^{v}}
			&=
			\begin{cases}
			\Theta(N^{-\left(\frac{t-1}{u}\right)})\, , & \text{if }\, (t-1)/u<v\,,\\
			\Theta(N^{-v}\log N)\, , &\text{if }\,(t-1)/u=v\,,\\
			\Theta(N^{-v})\, , &\text{if }\,(t-1)/u>v\, .
			\end{cases}\label{eqn:knapik_2p2}
		\end{align}
	\end{subequations}
\end{lemma}
\begin{proof}
	The claims follow from \cite[pp. 2654--2655]{knapik2011bayesian}. Choose the slowly varying function used there to be identically constant, $ q=-1/2 $, and use the fact that $ \sum_{j=1}^{J}1/j \asymp \log J $.
\end{proof}

The next lemma justifies the a.s. convergence of various random series in our proofs.
\begin{lemma}[almost sure convergence of series]\label{lem:as_series}
	Let $ \{X_j\}_{j\geq 1} $ be a sequence of (possibly dependent) real r.v.s. If $ \sum_{j=1}^{\infty}\E\abs{X_j}<\infty $, then $ \sum_{j=1}^{J}X_j \xrightarrow{\mathrm{a.s.}} \sum_{j=1}^{\infty}X_j $ as $ J\to\infty $.
\end{lemma}
\begin{proof}
	An application of monotone convergence shows that $ \sum_{j}\abs{X_j} $ converges a.s.\,.
\end{proof}

We now turn to some useful concentration inequalities for subexponential r.v.s.
\begin{lemma}[subexponential: union]\label{lem:subexp_bound_dep}
	For \( n\in\N \), let $ \{X_j^{(n)}\}_{j\geq 1} $ be a (possibly dependent) family of unit mean \( \SE{v^2/n}{a/n} \) r.v.s. Fix $\delta \in (0,\min\{1,v^2/a\})$ and \( J\in\N \). Then with probability at 
	least $ 1-2J\exp(-n\delta^2/(2v^2)) $, it holds that $ (1-\delta)\leq X_j^{(n)} \leq (1+\delta)$ for all $j\leq J $.
\end{lemma}
\begin{proof}
	The result follows from application of the union bound to \cite[Prop. 2.9]{wainwright2019high}.
\end{proof}

To develop tighter concentration for subexponential series, we need the next two lemmas.
\begin{lemma}[subexponential: closure under addition]\label{lem:subexp_add}
	Let $ J\in\N $. If $ \{X_j\}_{j=1,\ldots, J} $ are (possibly dependent) real-valued r.v.s such that $ X_j\in \SE{v_j^2}{a_j} $ for every \( j\in\{1,\ldots,J\} \), then
	\begin{equation}\label{eqn:subexp_add}
	\textstyle
	\sum_{j=1}^{J} X_j \in \SE[\big]{(\sum_{j=1}^{J}v_j)^2}{(\sum_{j=1}^{J}v_j)\max_{1\leq i\leq J}\frac{a_i}{v_i}}\, .
	\end{equation}
\end{lemma}
\begin{proof}
	Defining the centered r.v. \( Y_j\defeq X_j-\E X_j \) for each \( j\in\{1,\ldots,J\} \), we estimate
	\begin{align*}
	\textstyle
	\E\exp(\theta\sum_{j=1}^{J}Y_j)	=\E\prod_{j=1}^{J}
	&\textstyle
	\exp(\theta Y_j)\leq\prod_{j=1}^{J}\bigl(\E\exp(\theta Y_j p_j)\bigr)^{1/p_j}\\
	&\textstyle
	\leq \prod_{j=1}^{J}\bigl(\exp(v_j^2\theta^2p_j^2/2)\bigr)^{1/p_j}
	=\exp\bigl((\sum_{j=1}^{J}v_j)^2\theta^2/2\bigr)\, .
	\end{align*}
	We used the generalized H\"older's inequality to yield the first inequality with $ \sum_{i=1}^{J}1/p_i=1 $ and $ p_i\defeq v_i^{-1}\sum_{j=1}^{J}v_j $. The SE MGF bound applied for each \( j\in\{1,\ldots,J\} \) yields the second inequality, which is valid for all $ \abs{\theta}<\min_{i\leq J}(p_ia_i)^{-1}=(\max_{i\leq J}p_ia_i)^{-1} $ as asserted.
\end{proof}

\begin{lemma}[subexponential: series]\label{lem:fullseriesbound_dep}
	For \( n\in\N \), let $ \{X_j^{(n)}\}_{j\geq 1} $ be a (possibly dependent) family of nonnegative unit mean \( \SE{v^2/n}{a/n} \) r.v.s. Let $w \in\ell^{1}(\N;\R) $ be nonnegative. Fix $\delta \in (0,\min\{1,v^2/a\})$. Then with probability at least $ 1-2\exp(-n\delta^2/(2v^2)) $, it holds that
	\begin{equation}\label{eqn:full_series_concentration_dep}
	(1-\delta)\textstyle\sum_{j=1}^{\infty} w_j \leq \sum_{j=1}^{\infty} w_j X_j^{(n)}\leq (1+\delta)\sum_{j=1}^{\infty} w_j\, .
	\end{equation}
\end{lemma}
\begin{proof}
	For any $  J \in \N $, define $Y_J\defeq \sum_{j\leq J} w_j X_j^{(n)} $. It follows from \cref{lem:subexp_add} that $ Y_J\in\SE{\frac{v^2}{n}\norm{\{w_j\}_{j\leq J}}_1^2}{\frac{a}{n}\norm{\{w_j\}_{j\leq J}}_1} $. Since
	\(
	\sum_{j=1}^{\infty}\E\abs{w_jX_j^{(n)}}
	=\sum_{j=1}^{\infty}w_j
	< \infty
	\)
	holds by hypothesis, we deduce that $ Y_J\to Y_{\infty} $ as $ J\to\infty $ $ \P $-a.s. by monotone convergence (\cref{lem:as_series}). Fatou's lemma applied to the \( Y_J \) SE MGF bound yields $ Y_{\infty}\in\SE{\frac{v^2}{n}\norm{w}_{\ell^1}^2}{\frac{a}{n}\norm{w}_{\ell^1}} $. Thus, the fact that $ \E Y_{\infty}=\norm{w}_{\ell^1}$ and the SE tail bound (\cref{lem:subexp_bound_dep}) establish that
	$ \P\bigl\{\abs{Y_{\infty}-\E Y_{\infty}}\leq \E Y_{\infty}\, \delta\bigr\}\geq 1 - 2 \exp(-n \delta^2/(2v^2)) $ for all $\delta \in (0,\min\{1,v^2/a\})$ as asserted.
\end{proof}

Our last result, specific to Gaussian design, is used in the proof of \cref{thm:intro_ideas_thm}.
\begin{lemma}[chi-square moments]\label{lem:chisq}
	Let $ \ZZ\sim \chi^2(n) $ be a chi-square r.v. with $ n \in\N $ degrees of freedom. Then for any $ q>-n/2 $, \( 	\E[\ZZ^{q}]=2^{q}\frac{\Gamma(q+n/2)}{\Gamma(n/2)} \),
	where $ \Gamma $ is Euler's gamma function. 
\end{lemma}
\begin{proof}
	A direct calculation with the PDF of $ \chi^2(n) $ yields the moment in closed form.
\end{proof}

\section{Proofs of auxiliary results}\label{app:extra}
We prove the facts asserted in \cref{sec:setup_bayes}.
\begin{proof}[Proof of \cref{fact:setup_bayes_measure}]
	By \cref{eqn:norm_to_trace}, \( \cK^{-1/2}\in\HS(H_{\Lambda'};H) \) if and only if \( \E^{x\sim\nu'}\norm{x}^2_{\cK}<\infty \). Hence, \( \nu'(H_{\cK})=1 \) as claimed. For the second claim, for any orthonormal basis \( \{\psi_j\} \) of \( H \) we compute
	\[
	\textstyle\norm{T\Lambda'^{1/2}}_{\HS}^2=\sum_{i,j}\ip{\psi_i}{T(\cK^{1/2}\cK^{-1/2})\Lambda'^{1/2}\psi_j}^2=\sum_{i,j}\ip{(T\cK^{1/2})^{*}\psi_i}{\cK^{-1/2}\Lambda'^{1/2}\psi_j}^2\, .
	\]
	Applying the Cauchy--Schwarz inequality to the rightmost equality yields the upper bound \( \norm{(T\cK^{1/2})^*}_{\HS}^2\norm{\cK^{-1/2}\Lambda'^{1/2}}_{\HS}^2 \). This is finite by hypothesis. So, we deduce \( \E^{x\sim\nu'}\norm{Tx}^2<\infty \).
\end{proof}

\begin{proof}[Proof of \cref{fact:forward_facts}]
	For \( N\in\N \), let \( Z=(z_1,\ldots,z_N) \in H_{\cK}^N\setminus\{0\}\). By definition of \( K_Z \), the map \( \KZstarnormal\in\cL(\HS(H_{\cK};H)) \) acts as the right multiplication operator \( T\mapsto T\cC_{\cK}^{(N)} \), where \( \cC_{\cK}^{(N)}=\frac{1}{N}\sum_{n=1}^{N}z_n\otimes_{H_{\cK}}z_n\in\cL(H_{\cK})\setminus\{0\} \) is the empirical covariance of \( Z \) on \( H_{\cK} \). Thus, \( \KZstarnormal=\id_{H}\otimes\hspace{1mm}\cC_{\cK}^{(N)} \) is a tensor product operator on \( H\otimes H_{\cK} \). But \( \id_{H}\in\cL(H) \) is not compact on \( H \). By \cite[Cor. 1]{kubrusly2015note}, \( \KZstarnormal \) is not compact. Thus, \( K_Z \) is not compact either.
\end{proof}

\section*{Acknowledgments}
The authors thank Kamyar Azizzadenesheli and Joel A. Tropp for 
helpful discussions about statistical learning. The authors are also grateful to the associate editor and two anonymous referees for their helpful feedback.
The computations presented in this paper were conducted on the Resnick High Performance Computing Center, a facility supported by the Resnick Sustainability Institute at the California Institute of Technology.

%%%%% References
\bibliographystyle{siamplain}
\bibliography{references}

\end{document}